\documentclass[reqno,a4paper]{amsart} %

\usepackage[T1]{fontenc} 
\usepackage[utf8]{inputenc}
\usepackage[english]{babel}
\usepackage{newpxtext} %
\usepackage{eulervm} %

\usepackage{amssymb,textcomp,mathrsfs,bm,braket,commath,relsize}
\usepackage{stmaryrd}
    \SetSymbolFont{stmry}{bold}{U}{stmry}{m}{n} %

\usepackage{mathtools}
    \mathtoolsset{showonlyrefs} %

\theoremstyle{plain} %
    \newtheorem{theorem}{Theorem}[section]
    \newtheorem*{theorem*}{Theorem}
    \newtheorem{proposition}{Proposition}[section]
    \newtheorem*{proposition*}{Proposition}
    \newtheorem{corollary}{Corollary}[section]
    \newtheorem*{corollary*}{Corollary}
    \newtheorem{lemma}{Lemma}[section]
    \newtheorem*{lemma*}{Lemma}
    
    \newtheorem*{conjecture*}{Conjecture}
    
\theoremstyle{definition} %
    \newtheorem{definition}{Definition}[section]
    \newtheorem*{definition*}{Definition}

\theoremstyle{remark} %
    \newtheorem{remark}{Remark}[section]
    \newtheorem*{remark*}{Remark}

\usepackage{etoolbox} 
    \newcommand{\addQEDstyle}[2]{\AtBeginEnvironment{#1}{\pushQED{\qed}\renewcommand{\qedsymbol}{#2}}
    \AtEndEnvironment{#1}{\popQED}} %
    \addQEDstyle{remark}{$\triangle$}
    \addQEDstyle{remark*}{$\triangle$} %

\AtBeginDocument{\def\MR#1{}} %
\apptocmd{\sloppy}{\hbadness 10000\relax}{}{} %

\makeatletter %
\@namedef{subjclassname@2020}{
    \textup{2020} Mathematics Subject Classification}
\makeatother

\usepackage{mymacros}

\usepackage[colorlinks,psdextra,pdfencoding=auto,backref=page]{hyperref}
    \hypersetup{colorlinks,linkcolor=red,filecolor=green,urlcolor=blue,citecolor=blue}

\begin{document}

\title[Singular modules and irregular conformal blocks]{Singular modules for affine Lie algebras, and applications to irregular WZNW conformal blocks} %

\author[G. Felder \and G. Rembado]{Giovanni Felder \and Gabriele Rembado} %

\address[G. Felder]{ETH Z\"{u}rich, Department of Mathematics, R\"{a}mistrasse 101, 8092 Z\"{u}rich, Switzerland}
\email{giovanni.felder@math.ethz.ch}
\thanks{During this project, G.F. was supported in part by the National Centre of Competence
	in Research SwissMAP---The Mathematics of Physics---of the Swiss
    National Science Foundation.}

\address[G. Rembado]{Hausdorff Centre for Mathematics (HCM), Endenicher Allee 60, D-53115, Bonn, Germany}
\curraddr{Institut Montpelliérain Alexander Grothendieck,
	University of Montpellier,
	Place Eugène Bataillon,
	34090 Montpellier,
	France}
\email{gabriele.rembado@umontpellier.fr}
\thanks{During this project, G.R. was supported by the Deutsche Forschungsgemeinschaft (DFG, German Research Foundation) under Germany’s Excellence Strategy - GZ 2047/1, Projekt-ID 390685813.}



\begin{abstract}
	We give a mathematical definition of spaces of irregular vacua/covacua in genus zero, for any simple Lie algebra, working at generic noncritical level.
	This uses coinvariants of affine-Lie-algebra modules whose parameters match up with those of moduli spaces of irregular-singular meromorphic connections: the de Rham spaces.
	The Segal--Sugawara representation of the Virasoro algebra is used to show that the spaces of irregular conformal blocks assemble into a flat vector bundle over the space of isomonodromy times à la Klarès, and we provide a universal version of the resulting flat connection generalising the irregular KZ connection of Reshetikhin and the dynamical KZ connection of Felder--Markov--Tarasov--Varchenko.
\end{abstract}

{\let\newpage\relax\maketitle} %

\setcounter{tocdepth}{1}  %
\tableofcontents

\section*{Introduction and main results}

\renewcommand{\thetheorem}{\arabic{theorem}} %

In this paper we pursue the viewpoint that a natural mathematical formulation of conformal field theory (CFT) lies within the geometry of moduli spaces of meromorphic connections, and we take a step in this direction.

The prototype are the Knizhnik--Zamolodchikov equations (KZ)~\cite{knizhnik_zamolodchikov_1984_current_algebra_and_wess_zumino_model_in_two_dimensions}, in the genus-zero Wess--Zumino--Novikov--Witten model (WZNW) for 2-dimensional CFT~\cite{wess_zumino_1971_consequences_of_anomalous_ward_identities,witten_1984_nonabelian_bosonization_in_two_dimensions,novikov_1982_the_hamiltonian_formalism_and_a_multivalued_analogue_of_morse_theory}.
They were originally introduced as the partial differential equations satisfied by $n$-point correlators, and mathematically they amount to a flat connection on a vector bundle over the space of configurations of $n$-tuples of points in the complex plane~\cite{etingof_frenkel_kirillov_1998_lectures_on_representation_theory_and_knizhnik_zamolodchikov_equations}.

The construction of the flat connection relies on representation-theoretic constructions for affine Lie algebras, and on the Segal--Sugawara representation of the Virasoro algebra on affine-Lie-algebra modules~\cite{kohno_2002_conformal_field_theory_and_topology}.
An alternative derivation is possible via deformation quantisation of the Hamiltonian system controlling isomonodromic deformations of Fuchsian systems on the Riemann sphere~\cite{reshetikhin_1992_the_knizhnik_zamolodchikov_system_as_a_deformation_of_the_isomonodromy_problem,harnad_1996_quantum_isomonodromic_deformations_and_the_knizhnik_zamolodchikov_equations}, the Schlesinger system~\cite{schlesinger_1905_ueber_die_loesungen_gewisser_linearer_differentialgleichungen_als_funktionen_der_singularen_punkte}.
In particular the vector bundle where the KZ connection is defined comes from the quantisation of moduli spaces of meromorphic connections with tame/regular singularities (simple poles).

\vspace{5pt}

Thus here we consider a representation-theoretic setup for any simple Lie algebra $\mf g$ (over $\mb C$), in order to go beyond the case of regular singularities and allow for irregular (`wild') ones.
We will thus define a family of modules for $\mf g$ and for the affine Lie algebra $\wh{\mf g}$ associated with $\mf g$, which we call `singular' modules.\footnote{`Singularity modules' is also a fitting name (cf.~\cite{calaque_felder_rembado_wentworth_2024_wild_orbits_and_generalised_singularity_modules_stratifications_and_quantisation}), since they are attached to the singularity of a connection---at a point on Riemann surface.}
Their parameters match up with those of symplectic moduli spaces of (typically irregular-singular) meromorphic connections on the sphere, generalising Verma modules.

\vspace{5pt}

Indeed the regular-singular case will correspond to `tame' modules $V_{\lambda} \subseteq \wh V_{\lambda}$, which are standard Verma modules for $\mf g \subseteq \wh{\mf g}$, whose defining representations depend on characters $\mf b^+ \to \mb C$ and $\wh{\mf b}^+ \to \mb C$ of Borel subalgebras $\mf b^+ \subseteq \wh{\mf b}^+$---corresponding to positive roots within the root system given by a Cartan subalgebra $\mf h \subseteq \mf b^+$.
Such characters are encoded by linear maps $\lambda \in \mf h^{\dual}$, which in turn match up with local normal forms for (germs of) logarithmic connections, via the natural residue-pairing $\mc L\mf g \dif z \otimes_{\mb C} \mc L\mf g \to \mb C$, where $\mc L\mf g = \mf g \otimes_{\mb C} \mb C (\!( z )\!)$ is the (formal) loop algebra of $\mf g$.
Moreover, if $G$ is a Lie group integrating $\mf g$, then the $G$-action on the coadjoint orbit $\mc O \subseteq \mf g^{\dual}$ through $\lambda$ corresponds to the truncated gauge action on (the germ of) such a logarithmic connection---keeping only its principal part.
Repeating this construction at $n \geq 1$ marked points on the sphere provides a finite-dimensional description of the moduli space of isomorphism classes of logarithmic connections with prescribed positions of the poles and residue orbits, defined on holomorphically trivial bundles: this is the open part $\mc M^*_{\dR} \subseteq \mc M_{\dR}$ of the de Rham space, that enters into the nonabelian Hodge correspondence on (complex) curves.
(The full de Rham space $\mc M_{\dR}$ is obtained by removing the requirement that the bundle be holomorphically trivial, but rather just topologically trivial~\cite[Rmk.~2.1]{boalch_2001_symplectic_manifolds_and_isomonodromic_deformations}.)

Hence semiclassically one finds a complex symplectic reduction of a product of coadjoint $G$-orbits $\mc O_i \subseteq \mf g^{\dual}$, i.e.
\begin{equation}
	\label{eq:de_rham}
	\mc M^*_{\dR} = \Bigl( \prod_i \mc O_i \Bigr) \sslash_0 G \, .
\end{equation}
Its quantum counterpart is intimately related with the vector space $\ms H = \mc H_{\mf g}$ of $\mf g$-coinvariants of the tensor product $\mc H = \bigotimes_i V_{\lambda_i}$ of tame modules:\fn{
	Starting from the observation that one can deformation-quantize semisimple coadjoint orbits using the Verma modules (cf.~\cite{donin_gurevich_shnider_1996_quantization_of_function_algebras_on_semisimple_orbits_in_g_star,alekseev_lachowska_2005_invariant_star_product_on_coadjoint_orbits_and_the_shapovalov_pairing,calaque_felder_rembado_wentworth_2024_wild_orbits_and_generalised_singularity_modules_stratifications_and_quantisation}),
	and setting up quantum Hamiltonian reduction.}
it is the space of \emph{covacua}, dual to the space of \emph{vacua}.
If the level is a positive integer, and the weights are (admissible and) integral~\cite{sorger_1996_la_formule_de_verlinde}, one typically replaces affine Verma modules by integrable ones: this leads to the actual vector space of WZWN conformal blocks, while here we will work at any (noncritical) level---%
so that generically there are no nontrivial quotients.
(In particular this is better suited to taking continuous limits of the level, seen as a deformation parameter in quantisation.)

Now one very important feature are the admissible deformations, both in the semiclassical and quantum setting.
Namely as the positions of the simple poles vary the moduli spaces assemble into a symplectic fibre bundle
\begin{equation*}
	\wt{\mc M}^*_{\dR} \lra \Conf_n(\mb C) \, ,
\end{equation*}
over the space $\Conf_n(\mb C) \subseteq \mb C^n$ of (ordered) configurations of points on the complex affine line, equipped with a flat symplectic Ehresmann connection: it is the \emph{isomonodromy} connection~\cite{hitchin_1997_frobenius_manifolds}, defined here by the integrable (nonautonomous) Schlesinger system~\cite{schlesinger_1905_ueber_die_loesungen_gewisser_linearer_differentialgleichungen_als_funktionen_der_singularen_punkte}.
The leaves of this (nonlinear) connection are \emph{isomonodromic} families of (linear) meromorphic connections, i.e. the monodromy data are kept locally constant.\footnote{These data correspond to the isomorphism class of the monodromy representation of the fundamental group of the punctured sphere, with the poles removed.
	Indeed the de Rham isomonodromy connection is the pullback of the (nonabelian) Gau\ss--Manin connection on the associated family of $G$-character varieties along the Riemann--Hilbert correspondence, viz. the map taking monodromy data (which can be thought as a `global' version of the exponential $\mf g \to G$, cf. the abstract of~\cite{boalch_2017_wild_character_varieties_meromorphic_hitchin_systems_and_dynkin_graphs}).}
Hence in brief semiclassically we find a flat symplectic fibre bundle.

On the quantum side one thus looks for a (linear) flat connection on the conformal block bundle, to yield identifications of different fibres up to the braiding of the marked points, analogously to the symplectomorphisms defined by the nonlinear isomonodromy connection.
This `quantum' flat connection is precisely the KZ connection, which is intrinsically defined via the slotwise action of the Sugawara operator $L_{-1} \in \mf{Vir}$ on the tensor product $\wh{\mc H} = \bigotimes_i \wh V_{\lambda_i}$ of tame modules for the affine Lie algebra, where $\mf{Vir}$ is the Virasoro algebra (this underlies the Virasoro uniformization of $\wh{\mc M}_{0,n}$~\cite{kontsevich_1987_the_virasoro_algebra_and_teichmueller_space,benzvi_frenkel_2004_geometric_realization_of_the_segal_sugawara_construction}).
The action is compatible with that of the Lie algebra of $\mf g$-valued meromorphic functions on the punctured sphere, by Laurent expansions at the marked points: hence there is a well-defined connection on the bundle of coinvariants.

\vspace{5pt}

This is the picture that we wish to generalise on the side of the representation theory of affine Lie algebras.
Namely, to define generalisations of Verma modules we look at the symplectic geometry of moduli spaces of \emph{irregular}-singular meromorphic connections, which has been studied in much greater generality: for arbitrary genus, complex reductive structure group, polar divisor, and for any nongeneric/twisted irregular types~\cite{boalch_2001_symplectic_manifolds_and_isomonodromic_deformations,boalch_2002_g_bundles_isomonodromy_and_quantum_weyl_groups,boalch_2007_quasi_hamiltonian_geometry_of_meromorphic_connections,boalch_2014_geometry_and_braiding_of_stokes_data_fission_and_wild_character_varieties,boalch_yamakawa_2015_twisted_wild_character_varieties}.
Intrinsic definitions allow for the construction of symplectic local systems of wild character varieties generalising the above, also entering the \textit{wild} nonabelian Hodge correspondence on (complex) curves~\cite{sabbah_1999_harmonic_metrics_and_connections_with_irregular_singularities,biquard_boalch_2004_wild_nonabelian_hodge_theory_on_curves}.
We concern ourselves here with the case of genus zero, of a simple group, and untwisted irregular types; cf.~\cite{boalch_2012_hyperkaehler_manifolds_and_nonbelian_hodge_theory_of_irregular_curves,boalch_2017_wild_character_varieties_meromorphic_hitchin_systems_and_dynkin_graphs} for terminology and motivation.

Hence the open parts of de Rham spaces $\mc M^*_{\dR}$ are still defined.
Importantly one now considers isomorphism classes of connections with higher-order poles, which have local moduli at each pole parametrising the principal parts---beyond the residue term.
This may be formalised in terms of `deeper' coadjoint orbits of the dual Lie algebra $\mf g_p^{\dual}$, where
\begin{equation}
	\mf g_p = \mf g \llbracket z \rrbracket \big\slash z^p \mf g \llbracket z \rrbracket \simeq \bigoplus_{i = 0}^{p-1} \mf g \otimes z^i \, ,
\end{equation}
which is a Lie algebra of truncated $\mf g$-currents, holomorphic at $z = 0$ (a.k.a. the `Takiff' Lie algebra~\cite{takiff_1971_rings_of_invariant_polynomials_for_a_class_of_lie_algebras}).
Indeed the residue-pairing matches up $\mf g_p$ with a space of meromorphic $\mf g$-valued 1-forms, which we see as principal parts of meromorphic connections on a trivial principal $G$-bundle on the formal disc $\on{Spec} \bigl( \mb C \llbracket z \rrbracket \bigr)$.
The upshot is that one still has the description~\eqref{eq:de_rham}: now however one considers coadjoint $G_p$-orbit $\mc O_i \subseteq \mf g_p^{\dual}$, where
\begin{equation*}
	G_p = G \bigl( \mb C \llbracket z \rrbracket \big\slash z^p \mb C \llbracket z \rrbracket \bigr) \, .
\end{equation*}
This is the group of $(p-1)$-jets of bundle automorphisms of the trivial principal $G$-bundle on a (formal) disc, integrating $\mf g_p$.
The diagonal $G$-action corresponds to a change of global trivialisation of the bundle, as in the tame case, cf. the proof of~\cite[Prop.~2.1]{boalch_2001_symplectic_manifolds_and_isomonodromic_deformations}.

\vspace{5pt}

Hence we will define modules $W^{(p)}_{\chi} \subseteq \wh W^{(p)}_{\chi}$ (at depth $p \geq 1$) for $\mf g_p$ and $\wh{\mf g}$ respectively, whose defining representations depend on elements of $\mf h^{\dual}_p \subseteq \mf g_p^{\dual}$.
In turn the latter correspond to characters for Lie subalgebras generalising the positive (affine) Borels, so that for $p = 1$ they reduce to the usual tame modules (else they are `wild/irregular').
This is done in Def.~\ref{def:affine_singular_module}, which is a variations of similar definitions considered elsewhere, and which is the best suited to our viewpoint on the moduli spaces~\eqref{eq:de_rham}.
For example~\eqref{eq:affine_singular_module} has a more general scope than the `confluent Verma modules' of~\cite{jimbo_nagoya_sun_2008_remarks_on_the_confluent_kz_equations_for_sl_2_and_quantum_painleve_equations,nagoya_sun_2011_confluent_kz_equations_for_sl_n_with_poincare_rank_2_at_infinity}, since we allow for an arbitrary simple Lie algebra and for arbitrary irregular singularities (of arbitrary Poincar\'{e} rank)---note that we do not use the viewpoint of confluence.
Also we do not work in Liouville theory, i.e. we do not consider modules for the Virasoro algebra, e.g. as in~\cite{nagoya_2015_irregular_conformal_blocks_with_an_application_to_the_fifth_and_fourth_painleve_equations}.
The approach in this paper is closer to the `level subalgebra' of~\cite{fedorov_2010_irregular_wakimoto_modules_and_the_casimir_connection} (cf. also~\S~\ref{sec:dynamical_term}), or rather to one of its `more reasonable' variants (see Rmk.~4 of op. cit.).
The other variant is used in~\cite[\S~2.8]{feigin_frenkel_toledanolaredo_2010_gaudin_models_with_irregular_singularities}: in this setup the natural pairing~\eqref{eq:residue_pairing} matches the parameter of the modules with \emph{half} of principal parts of irregular meromorphic connections, contrary to~\eqref{eq:affine_singular_module}.\footnote{The viewpoint of op. cit. on meromorphic connections is different: at critical level $\kappa = -h^{\dual}$ one identifies quotients of the `universal Gaudin algebra' with algebras of functions on spaces of opers for the Langlands dual group $\prescript{L}{}{G}$ of $G$, with a view towards the geometric Langlands correspondence for loop groups~\cite{frenkel_2007_langlands_correspondence_for_loop_groups}.
(This seems closer to the quantization of the Dolbeault moduli spaces of $G$-Higgs bundles, and the Hitchin system, on the other side of the nonabelian Hodge correspondence.)}
The other two important differences with~\cite{feigin_frenkel_toledanolaredo_2010_gaudin_models_with_irregular_singularities} is that we work at \emph{noncritical} level, and that our $\mf g_p$-modules are highest-weight, leading to finite-dimensional spaces of coinvariants.
Finally, the $\mf g_p$-modules $W_\chi^{(p)}$ also appear in~\cite{wilson_2011_highest_weight_theory_for_truncated_current_lie_algebras}, and enter into the category $\mc O$ for truncated current Lie algebras of~\cite{chaffe_topley_2023_category_o_for_truncated_current_lie_algebras} (cf.~\cite{chaffe_2023_category_o_for_takiff_lie_algebras} when $p = 2$).

\vspace{5pt}

In any event, the singular modules enjoy several natural generalisations of the standard properties of tame modules, some of which we gather here.
(We will refer to `affine' modules when $\wh{\mf g}$ is involved, and to `finite' modules when $\mf g_p$ is.)

\begin{theorem}~
	\label{thm:properties_singular_modules}
	\begin{itemize}
		\item The singular modules admit explicit PBW-generators (Corr.~\ref{cor:pbw_basis_affine_singular_module}--\ref{cor:pbw_basis_finite_singular_module}).

		\item The singular modules are smooth (Lem.~\ref{lem:smooth_modules}).\footnote{Recall a $\mf g \llbracket z \rrbracket$-module is smooth if every vector is annihilated by $z^{N} \mf g \llbracket z \rrbracket \subseteq \mf g \llbracket z \rrbracket$ for $N \gg 0$.}

		\item The singular modules are $\mf h$-semisimple (Prop.~\ref{prop:diagonalisable_singular_modules}), and the finite singular modules have finite-dimensional $\mf h$-weight spaces (Prop.~\ref{prop:finite_dimensional_weight_spaces}).

		\item The finite singular modules are highest-weight $\mf g_p$-modules (Lem.~\ref{lem:highest_weight}).

		\item The singular modules are cyclically generated by a common eigenvector for the Sugawara operators $\Set{ L_n }_{n \geq p-1}$  (Prop.~\ref{prop:sugawara_eigenvalues}), which is a Gaiotto--Teschner/Bonelli--Maruyoshi--Tanzini irregular state of order $p-1$~\cite{gaiotto_teschner_2012_irregular_singularities_in_liouville_theory_and_argyres_douglas_type_gauge_theories,bonelli_maruyoshi_tanzini_2012_wild_quiver_gauge_theories}.
	\end{itemize}
\end{theorem}

We also give a formula for the (finite) dimension of the $\mf h$-weight spaces of the finite modules, generalising the usual Weyl characters of Verma modules, in~\eqref{eq:dimension_weight_space}.
The combinatorial complexity still lies in the positive root lattice, so in the archetypal case of $\mf g = \mf{sl}(2,\mb C)$ there is a simple expression (see~\eqref{eq:dimension_weight_space_sl_2}).

After establishing these properties we consider tensor products of singular modules labeled by marked points on the Riemann sphere, and study their space $\ms H$ of coinvariants for the action of $\mf g$-valued meromorphic functions with poles at the marked points.
Introducing generalisations of the standard filtrations/gradings of tame modules we prove the following.

\begin{theorem}~
	\label{thm:coinvariants_intro}
	\begin{itemize}
		\item The space $\ms H$ is canonically identified with the space of $\mf g$-coinvariants for the tensor product of finite modules (Props.~\ref{prop:surjectivity_coinvariants}--\ref{prop:kernel_coinvariants} and~\ref{prop:coinvariants_with_dual}).

		\item The space $\ms H$ is finite-dimensional if one module is tame (Cor.~\ref{cor:finite_dimensional_irregulal_conformal_blocks}).
	\end{itemize}
\end{theorem}

To ensure nontriviality of the space of coinvariants we explore two options: either replacing one of the modules at the marked points with its associated contragredient representation (see Prop.~\ref{prop:nontrial_irregular_conformal_blocks}), or restricting the action of rational function to the subalgebra of those which vanish at an unmarked point (see Rmk.~\ref{rem:nontriviality}).

\vspace{5pt}

Finally we consider noncoalescing deformations of the distinct marked points, i.e. variations of the \emph{tame} isomonodromy times.
This is not the full set of isomonodromy times, as in the most general setup one may also vary the irregular types/classes and give nonlinear differential equations for the invariance of \emph{Stokes} data along these deformation.
This goes beyond the isomonodromy times of~\cite{klares_1979_sur_une_classe_de_connexions_relatives}, as well as the `generic' case of~\cite{jimbo_miwa_mori_sato_1980_density_matrix_of_an_impenetrable_bose_gas_and_the_fifth_painleve_transcendent}.\footnote{A more recent series of papers gives an intrinsic description of irregular isomonodromy times:~\cite{doucot_rembado_tamiozzo_2022_local_wild_mapping_class_groups_and_cabled_braids,doucot_rembado_2023_topology_of_irregular_isomonodromy_times_on_a_fixed_pointed_curve,boalch_doucot_rembado_2022_twisted_local_wild_mapping_class_groups_configuration_spaces_fission_trees_and_complex_braids,
		doucot_rembado_tamiozzo_2024_moduli_spaces_of_untwisted_wild_riemann_surfaces,
		doucout_rembado_yamakawa_twisted_g_local_wild_mapping_class_groups}.}
We briefly discuss one natural setup to introduce a space of irregular isomonodromy times in \S~\ref{sec:irregular_isomonodromy_times}, and we plan to pursue its quantum version in future work, which should be more closely related to~\cite{fedorov_2010_irregular_wakimoto_modules_and_the_casimir_connection,feigin_frenkel_toledanolaredo_2010_gaudin_models_with_irregular_singularities}.
(Part of this has been taken on in~\cite{calaque_felder_rembado_wentworth_2024_wild_orbits_and_generalised_singularity_modules_stratifications_and_quantisation}.)

\vspace{5pt}

Thus in brief we allow for variations of marked points at finite distance on the sphere.
Then we use the Sugawara operators to define a flat connection on the trivial vector bundle whose fibre is the tensor product of affine singular modules, and show this is compatible with the action of rational functions on the punctured sphere.
Hence the spaces of coinvariants assemble into a flat vector bundle over the space of tame isomonodromy times, so in particular their dimension is a deformation-invariant---when finite.

Using the above results it is possible to give descriptions of the flat connections on the space $\ms H$ of coinvariants.
Considering all possible cases of our setup we recover (by design):
\begin{enumerate}
	\item the KZ connection~\cite{knizhnik_zamolodchikov_1984_current_algebra_and_wess_zumino_model_in_two_dimensions} (\S~\ref{sec:reduced_connection_tame});

	\item a variation of the Cartan term of the dynamical KZ connection~\cite{felder_markov_tarasov_varchenko_2000_differential_equations_compatible_with_kz_equations} (\S~\ref{sec:reduced_connection_wild_at_infinity}), and the very same Cartan term with a slightly different setup (\S~\ref{sec:dynamical_term});

	\item the general case of~\cite{reshetikhin_1992_the_knizhnik_zamolodchikov_system_as_a_deformation_of_the_isomonodromy_problem} (\S~\ref{sec:reduced_connection_tame_at_infinity}), which generalises the KZ connection;

	\item a generalisation of op. cit. with nontrivial action on the module at infinity (\S~\ref{sec:reduced_connection_general}).
\end{enumerate}

In particular the semiclassical limit of the flat connections indeed yields isomonodromy systems for irregular meromorphic connections on the sphere~\cite{klares_1979_sur_une_classe_de_connexions_relatives}, as wanted.

Note the last two items in principle descend from a more general setup, where the point at infinity is not fixed, provided one can show how horizontal sections transform under the pull-back diagonal $\PSL(2,\mb C)$-action.
Going in this direction, in \S~\ref{sec:conformal_transformations} we prove that horizontal sections of the bundle of coinvariants are naturally equivariant under the action of the subgroup of affine transformations of the affine line, with the explicit transformation~\eqref{eq:affine_transformations}.

\vspace{5pt}

Finally we abstract the formul\ae{} for the reduced connections and define a family of \emph{universal} ones: these are connections $\nabla_p$ on the trivial vector bundle with fibre $U (\mf g_p )^{\otimes n}$ for $p \geq 1$, over the space of tame isomonodromy times, which induce the above connections on $\ms H$ by taking representations.\footnote{Note the connection of~\cite{reshetikhin_1992_the_knizhnik_zamolodchikov_system_as_a_deformation_of_the_isomonodromy_problem} is given in universal terms: $\mf g_p$-modules and (co)invariants are not discussed, nor are the irregular types of irregular meromorphic connections.}
Since all induced connections are flat and well defined on $\mf g$-coinvariants, it is natural to conjecture that the same holds for the universal connections \emph{before} taking representations.

\begin{theorem}[Thms.~\ref{thm:cybe} and~\ref{thm:universal_flatness}, and Prop.~\ref{prop:g_invariance_universal_connection}]
	\label{thm:flatness_intro}
	The connection $\nabla_p$ is flat, and descends to a connection on $\mf g$-coinvariants of the tensor power $U (\mf g_p)^{\otimes n}$.
\end{theorem}

Hence the singular modules provide a reasonable mathematical notion of conformal blocks in irregular versions of the genus-zero WNZW model,
which are thus related with~\cite{gaiotto_teschner_2012_irregular_singularities_in_liouville_theory_and_argyres_douglas_type_gauge_theories,yamada_2011_a_quantum_isomonodromy_equation_and_its_application_to_n_equal_2_su_n_gauge_theories}.
More precisely,
recall that irregular Liouville conformal blocks are central objects in the recent literature on the asymptotically-free extension of the Alday--Gaiotto--Tachikawa correspondence (AGT)~\cite{alday_gaiotto_tachikawa_2010_liouville_correlation_functions_from_four_dimensional_gauge_theories,gaiotto_2013_asymptotically_free_n_equal_2_theories_and_irregular_conformal_blocks}.
This uses the formalism of Whittaker modules, cf~\cite{felinska_jaskolski_kosztolowicz_2012_whittaker_pairs_for_the_virasoro_algebra_and_the_gaiotto_bonelli_maruyoshi_tanzini_states,nagoya_2015_irregular_conformal_blocks_with_an_application_to_the_fifth_and_fourth_painleve_equations}; in principle it should be possible to relate the latter with our construction for $\mf g = \mf{sl}(2,\mb C)$, perhaps passing through the duality between Liouville theory and the $H_3^+$-WZNW model~\cite{ribault_teschner_2005_h3plus_wznw_correlators_from_liouville_theory} (then in turn our construction should generalise~\cite{gaiotto_lamypoirier_2013_irregular_singularities_in_the_h3plus_wzw_model} beyond $\mf{sl}(2,\mb C)$, which is compatible with the duality~\cite{ribault_teschner_2005_h3plus_wznw_correlators_from_liouville_theory}).
(Note however that the affine singular modules are \emph{not} Whittaker modules for the standard Bonelli--Maruyoshi--Tanzini Whittaker pair, because they are not generated by the corresponding Whittaker vector; rather, they properly contain a copy of such a Whittaker module.)

\section*{Layout of the paper}

In \S~\ref{sec:setup} we consider a \emph{depth} $p \geq 1$ to introduce singular Lie algebras $\mf S^{(p)} \subseteq \wh{\mf g}$, singular characters $\chi \colon \mf S^{(p)} \to \mb C$, and affine/finite induced singular modules $W^{(p)}_{\chi} \subseteq \wh W_{\chi}$.

In \S~\ref{sec:meromorphic_connections} we explicitly match up the data $(p,\chi)$ with the local moduli for the isomorphism class of (the germ of) an irregular-singular meromorphic connection.

In \S~\ref{sec:bases_gradings_filtrations} we introduce countable PBW-bases $\mc B_W \subseteq W$ of the finite singular modules, as well as gradings and filtrations on the finite and affine singular modules: notably gradings $\mc F^+_{\bullet}$ and $\wh{\mc F}^{\pm}_{\bullet}$ for the degree in the variable `$z$', their associated filtrations, and then $\mf h$-weight gradings.

In \S~\ref{sec:dual_modules} we introduce left-module structures on (graded/restricted) dual vector spaces $\wh W^* \twoheadrightarrow W^*$.

In \S~\ref{sec:sugawara} we recall the definition of the Sugawara operators $L_n$ for $n \in \mb Z$, and prove that the cyclic vector $w \in W \subseteq \wh W$ is a Whittaker vector for the subalgebra with generators $L_n$, $n \geq p-1$.
This concludes proving the properties of Thm.~\ref{thm:properties_singular_modules}.

In \S~\ref{sec:irregular_conformal_blocks} we define the spaces of irregular covacua $\ms H$.
They are quotients of tensor products $\wh{\mc H}$ of affine singular modules labeled by marked points on the Riemann sphere with respect to the action of $\mf g$-valued meromorphic functions (and we globalise the action by introducing suitable sheaves over the space of tame isomonodromy times, in routine fashion).

In \S\S~\ref{sec:coinvariants} and~\ref{sec:irregular_conformal_blocks_space_with_duals} we study coinvariants, and we prove Thm.~\ref{thm:coinvariants_intro} using the material of \S\S~\ref{sec:bases_gradings_filtrations} and~\ref{sec:dual_modules}.

In \S~\ref{sec:connection_marked_points} we introduce the flat connection on the bundle of coinvariants, using the Sugawara operator $L_{-1}$ and fixing the point at infinity.
In \S~\ref{sec:description_finite_modules} we give explicit formul\ae{} for the reduced connection.

In \S~\ref{sec:yang_baxter} we introduce the universal connection $\nabla_p$ at depth $p \geq 1$, on the trivial vector bundle with fibre $U (\mf g_p)^{\otimes n}$ over the (restricted) space of tame isomonodromy times, and we prove Thm.~\ref{thm:flatness_intro}.

In \S~\ref{sec:conformal_transformations} we introduce the action of M\"{o}bius transformations on horizontal sections of the bundle of coinvariants, and establish equivariance under affine transformations.

Finaly in \S~\ref{sec:dynamical_term} we slightly modify the setup of \S~\ref{sec:setup} to generalise the dynamical KZ connection, i.e.~\cite[Eq.~3]{felder_markov_tarasov_varchenko_2000_differential_equations_compatible_with_kz_equations}.

\vspace{5pt}

Some standard notion is recalled in the appendix~\ref{sec:appendix_1}, and a few computations are gathered in~\ref{sec:appendix_2}.

Unless otherwise specified: affine spaces, vector spaces, vector bundles, associative/Lie algebras and tensor products are tacitly defined over $\mb C$.
The end of a remark is signaled by a `$\triangle$'.

\renewcommand{\thetheorem}{\arabic{section}.\arabic{theorem}} %

\section{Setup}
\label{sec:setup}

Let $\mf g$ be a (finite-dimensional) simple Lie algebra, and $\mf h \subseteq \mf g$ a Cartan subalgebra.
Let then $\mc R^+ \subseteq \mc R \subseteq \mf h^{\dual}$ be a choice of positive roots within the root system $\mc R = \mc R \bigl( \mf g, \mf h \bigr)$, and $\mc R^- \ceqq - \mc R^+$ the subset of negative roots.
There is a triangular/Cartan decomposition
\begin{equation*}
	\mf g = \mf n^- \oplus \mf h \oplus \mf n^+ \, ,
\end{equation*}
where $\mf n^{\pm}$ is the maximal positive/negative nilpotent subalgebra defined by the subset of positive/negative roots:
\begin{equation}
	\mf n^{\pm} \ceqq \bigoplus_{\alpha \in \mc R^{\pm}} \mf g_{\alpha} \, , \quad \mf g_{\alpha} \ceqq \Set {X \in \mf g | \bigl( \ad_{H} - \alpha(H) \bigr) X = 0 \text{ for } H \in \mf h } \, .
\end{equation}

Equip $\mf g$ with the minimal nondegenerate $\ad_{\mf g}$-invariant symmetric bilinear form $( \cdot \mid \cdot ) \colon \mf g \otimes \mf g \to \mb C$---so the highest root has length $\sqrt{2}$.
Consider then the (formal) loop algebra
\begin{equation*}
	\mc L \mf g = \mf g (\!( z )\!) \ceqq \mf g \otimes \mb C (\!( z )\!) \, ,
\end{equation*}
and let $\wh{\mf g}_{(\cdot \mid \cdot)} = \wh{\mf g} \simeq \mc L \mf g \oplus \mb C K$ be the associated affine Lie algebra (which one can further extend to the Kac--Moody version~\cite{moody_1967_lie_algebras_associated_with_generalized_cartan_matrices,kac_1990_infinite_dimensional_lie_algebras}).
The Lie bracket of $\wh{\mf g}$ is defined by imposing that $K \in \mf{Z}(\wh{\mf g})$, and
\begin{equation}
	\label{eq:affine_lie_bracket}
	\bigl[X \otimes f,Y \otimes g\bigr]_{\wh{\mf g}} = \bigl[X,Y\bigr]_{\mf g} \otimes fg + c(X \otimes f,Y \otimes g) K \, , \quad \text{for } f,g \in \mb C (\!( z )\!), \, X,Y \in \mf g \, ,
\end{equation}
where $c \colon \mc L \mf g \wedge \mc L \mf g \to \mb C$ is the Lie-algebra cocycle given by
\begin{equation}
	\label{eq:cocycle_affine_lie_algebra}
	c(X \otimes f,Y \otimes g) \ceqq (X \mid Y) \cdot \Res_{z = 0}(g df) \, ,
\end{equation}
and where in turn $\Res_{z = 0} \bigl( \omega \bigr) \ceqq f_{-1}$ for $\omega = \sum_i f_i z_i \dif z \in \mb C (\!( z )\!) \dif z$.

Then there is an analogous decomposition $\wh{\mf g} = \wh{\mf n}^- \oplus \wh{\mf h} \oplus \wh{\mf n}^+$, where
\begin{equation}
	\wh{\mf n}^+ \ceqq (\mf n^+ \otimes 1)\oplus z \mf g \llbracket z \rrbracket \, , \quad \wh{\mf n}^- \ceqq z^{-1} \mf g \bigl[ z^{-1} \bigr] \oplus (\mf n^- \otimes 1) \, , \quad \wh{\mf h} \ceqq (\mf h \otimes 1) \oplus \mb CK \, .
\end{equation}

Finally let $\mf b^{\pm} \ceqq \mf h \oplus \mf n^{\pm}$ be the positive/negative Borel subalgebras associated with the sets of positive/negative roots, and $\wh{\mf b}^{\pm} \ceqq (\mf b^{\pm} \otimes 1) \oplus z \mf g \llbracket z \rrbracket \oplus \mb CK$.

(Hereafter we drop the `$\otimes 1$' from the notation for vector subspaces of the constant part $\mf g \subseteq \mc L \mf g$, and the subscripts from the Lie brackets.)

\begin{remark*}
	The dual Coxeter number $h^{\dual}$ of the quadratic Lie algebra $\bigl( \mf g, ( \cdot \mid \cdot ) \bigr)$ is half of the eigenvalue for the adjoint action of the standard quadratic tensor on $\mf g$~\cite{kac_1990_infinite_dimensional_lie_algebras}.

	More precisely let $(X_k)_k$ be a basis of $\mf g$, $(X^k)_k$ the $(\cdot \mid \cdot)$-dual basis, and define
	\begin{equation}
		\Omega \ceqq \sum_k X_k \otimes X^k \in \mf g^{\otimes 2} \, .
	\end{equation}
	Intrinsically, this is the element corresponding to $\on{Id}_{\mf g} \in \mf g \otimes \mf g^{\dual}$ in the duality $\mf g^{\dual} \simeq \mf g$ induced by $( \cdot \mid \cdot)$.
	The projection of $\Omega$ to the universal enveloping algebra is the quadratic Casimir
	\begin{equation}
		\label{eq:quadratic_casimir}
		C = \sum_{K} X_kX^k \in U(\mf g) \, ,
	\end{equation}
	which is a central element---by the invariance of $( \cdot \mid \cdot)$.
	The adjoint action of $C$ on $\mf g$ is thus a homothety, and we define $h^{\dual}$ by
	\begin{equation}
		\ad_{C} X = %
		\sum_k \bigl[ X_k,[X^k,X] \bigr] = 2h^{\dual} X \, , \qquad \text{ for } X \in \mf g \, .
	\end{equation}

	We will also need a generalisation of the standard quadratic tensor $\Omega$.
	Namely,for $m,l \in \mb Z$ define
	\begin{equation}
		\label{eq:quadratic_tensor}
		\Omega_{ml} \ceqq \sum_k X_kz^m \otimes X^k z^l \in \mc L\mf g^{\otimes 2} \, ,
	\end{equation}
	with the shorthand notation $X z^i = X \otimes z^i$ for $X \in \mf g$ and $i \in \mb Z$.
	Then the identity $[C,X] = \sum_k \bigl[ X_kX^k,X \bigr] = 0$, valid for all $X \in \mf g$, also implies
	\begin{equation}
		\label{eq:commutation_quadratic_tensor}
		\sum_k X_k z^m \cdot \bigl[ X^k,X \bigr] z^l + \bigl[ X_k,X \bigr] z^m \cdot X^k z^l = 0 \, , \qquad \text{for } m,l \in \mb Z_{\geq 0} \, .
	\end{equation}
\end{remark*}

\subsection{Singular modules}

For an integer $p \geq 1$ consider the \emph{singular} Lie subalgebra $\mf S^{(p)} \subseteq \wh{\mf b}^+$ (of depth $p$), defined by
\begin{equation}
	\label{eq:singular_lie_algebra}
	\mf S^{(p)} \ceqq \mf b^+ \llbracket z \rrbracket + z^p \mf g \llbracket z \rrbracket \oplus \mb CK \, ,
\end{equation}
so that $\mf S^{(1)} = \wh{\mf b}^+$.

\begin{lemma}
	\label{lem:abelianisation}
	There is an identification of abelian Lie algebras
	\begin{equation}
		\label{eq:abelianised_lie_algebra_singular}
		\mf S^{(p)}_{\ab} \simeq \mf h_p \oplus \mb C K \, .
	\end{equation}
\end{lemma}

\begin{proof}
	We can define a linear surjection $\pi \colon \mf S^{(p)} \twoheadrightarrow \mf h_p \oplus \mb CK$ with kernel
	\begin{equation}
		\label{eq:derived_lie_algebra_singular}
		\bigl[ \, \mf S^{(p)},\mf S^{(p)} \bigr] = \mf n^+ \llbracket z \rrbracket + z^p \mf g \llbracket z \rrbracket \, ,
	\end{equation}
	by setting
	\begin{equation}
		\sum_{i = 0}^{p-1} (H_i + X_i) \otimes z^i + z^p f + a K \longmapsto \sum_{i = 0}^{p-1} H_i \otimes z^i + a K \, ,
	\end{equation}
	where $f \in \mf g \llbracket z \rrbracket$, $a \in \mb C$, $H_i \in \mf h$, and $X_i \in \mf n^+$ for $i \in \Set{ 0,\dc,p-1}$.
\end{proof}

Characters of~\eqref{eq:singular_lie_algebra} are coded by linear maps $\mf S^{(p)}_{\ab} \to \mb C$, i.e. by elements of $\mf h_p^{\dual}$ plus the choice of a \emph{level} $\kappa \in \mb C$ for the central element---using~\eqref{eq:abelianised_lie_algebra_singular}. We split the notation: for $p = 1$ write $\lambda \in \mf h^{\dual}$ the linear map, and for $p \geq 2$ write it $(\lambda,q) \in \mf h_p^{\dual}$, where
\begin{equation*}
	q = (a_1,\dc,a_{p-1}) \in \bigl(\mf h_p \big\slash \mf h\bigr)^{\dual} \simeq \bigoplus_{i = 1}^{p-1} \bigl( \mf h \otimes z^i \bigr)^{\dual} \, .
\end{equation*}

We will refer to $\chi = \chi(\lambda,q,\kappa) \colon \mf S^{(p)} \to \mb C$ as a \emph{singular} character (of depth $p$), and we denote $\mb C_{\chi}$ the 1-dimensional left $U \bigl( \mf S^{(p)} \bigr)$-module defined by it.
We also refer to $\lambda$ as the \emph{tame} part of the singular character, and to $q$ as the \emph{wild} part.

\begin{remark*}
	This hints to the dictionary with irregular-singular meromorphic connections on the Riemann sphere: $\lambda$ corresponds to a semisimple formal residue at a simple pole (a tame/regular singularity), and $q$ to an untwisted irregular type at a higher-order pole (a wild/irregular singularity), see \S~\ref{sec:meromorphic_connections}.

	We will use the uniform notation $\lambda = a_0$ when this distinction is not relevant.
\end{remark*}

\begin{definition}[Affine singular modules]~
	\label{def:affine_singular_module}
	\begin{itemize}
		\item The affine singular module (of depth $p$) for the singular character $\chi$ is
		      \begin{equation}
			      \label{eq:affine_singular_module}
			      \wh W = \wh W^{(p)}_{\chi} \ceqq \on{Ind}_{U ( \mf S^{(p)} )}^{U ( \wh{\mf g} )} \mb C_{\chi} = U \bigl( \wh{\mf g} \bigr) \otimes_{U ( \mf S^{(p)} )} \mb C_{\chi} \, .
		      \end{equation}

		\item We write $\wh V = \wh V_{\chi} \ceqq \wh W^{(1)}_{\chi}$, and call it the tame affine module for the character $\chi = \chi(\lambda,\kappa) \colon \wh{\mf b}^+ \to \mb C$.
	\end{itemize}
\end{definition}

The latter item is the standard definition of an affine (generic) Verma module, and by definition these are level-$\kappa$ modules.\footnote{Beware a `regular' Verma modules is a Verma module defined by a \emph{dominant} weight $\lambda \in \mf h$, so `tame' is better terminology here.}

Now denote by $w = 1 \otimes_{U(\mf S^{(p)})} 1 \in \wh W$ the canonical cyclic vector; then using~\eqref{eq:abelianised_lie_algebra_singular} and \eqref{eq:derived_lie_algebra_singular} yields
\begin{equation}
	\label{eq:annihilator_cyclic_vector}
	\begin{cases}
		z^p \mf g \llbracket z \rrbracket w = (0) = \mf n^+ \llbracket z \rrbracket w \, , &                                                                 \\
		H z^i w = \Braket{ a_i | H z^i } w \, ,                                            & \quad \text{for } H \in \mf h, \, i \in \Set{ 0, \dc, p-1} \, ,
	\end{cases}
\end{equation}
plus $K w = \kappa w$.
This generalises the relations satisfied by the highest-weight vector in a tame module.

Consider now the subspace
\begin{equation*}
	\wh W^- \ceqq U \bigl( \mf g \bigl[ z^{- 1} \bigr] \bigr) w \subseteq \wh W \, .
\end{equation*}
Because of~\eqref{eq:annihilator_cyclic_vector} it equals $\wh W^- = U \bigl( \wh{\mf n}^- \bigr) w$, so it is naturally a left $U \bigl( \wh{\mf n}^- \bigr)$-module with cyclic vector $w$---and it is canonically isomorphic to $U \bigl( \wh{\mf n}^- \bigr)$ as vector space.
Furthermore matching up cyclic vectors yields an isomorphism $\wh W^- \simeq \wh V$ of left $U \bigl( \mf g \bigl[ z^{-1} \bigr] \bigr)$-modules, regardless of $p \geq 1$ and $q \in \bigl( \mf h_p \big\slash \mf h \bigr)^{\dual}$.
(Note we implicitly use a $\mb C$-basis of $U \bigl( \wh{\mf g} \bigr)$ as provided by the Poincar\'{e}--Birkhoff--Witt theorem (PBW) for countable-dimensional Lie algebras.)

Consider then the subspace
\begin{equation}
	\label{eq:embedding_finite_singular_module}
	W \ceqq U \bigl( \mf g \llbracket z \rrbracket \bigr) w \subseteq \wh W \, ,
\end{equation}
which is naturally a left $U \bigl( \mf g \llbracket z \rrbracket \bigr)$-module, and which will play a more important role.
An inductive proof on the length of monomials---with base~\eqref{eq:annihilator_cyclic_vector}---shows that $z^p \mf g \llbracket z \rrbracket W = (0)$, so thaa the $\mf g \llbracket z \rrbracket$-action factors through the finite-dimensional quotient $\mf g \llbracket z \rrbracket \twoheadrightarrow \mf g_p$: thus~\eqref{eq:embedding_finite_singular_module} is naturally a left $U(\mf g_p)$-module.
Moreover $W = U \bigl( \mf n^-_p \bigr) w$ since $\mf b^+_p w = \mb C w$, so in particular $W \simeq U (\mf n^-_p )$ as vector spaces, independently of $\chi$.

\begin{remark*}
	Here we use the triangular decomposition $\mf g_p = \mf n^-_p \oplus \mf h_p \oplus \mf n^+_p$ and the inclusion $\mf b^+_p = \mf n^+_p \oplus \mf h_p \subseteq \mf g_p$.
\end{remark*}

Now one has $\mf n^+_p = \bigl[ \mf b_p^+,\mf b_p^+ \bigr]$ and $\bigl( \mf b_p^+ \bigr)_{\ab} \simeq \mf h_p$, so by~\eqref{eq:annihilator_cyclic_vector} there is a canonical identification
\begin{equation}
	\label{eq:finite_singular_module}
	W \simeq \on{Ind}_{U (\mf b^+_p) }^{U (\mf g_p)} \mb C_{\chi} = U(\mf g_p) \otimes_{U (\mf b^+_p)} \mb C_{\chi} \, ,
\end{equation}
where we keep the notation $\chi \colon \mf b^+_p \to \mb C$ for the character defined by $(\lambda,q) \in \mf h_p^{\dual}$---the level $\kappa$ is lost.

\begin{definition}[Finite singular modules]~
	\begin{itemize}
		\item We call $W = W^{(p)}_{\chi} \subseteq \wh W$ the finite singular module (of depth $p$) for the singular character $\chi$.

		\item We write $V = V_{\chi} = W^{(1)}_{\chi}$, and call it the tame finite singular module for the character $\chi = \chi(\lambda) \colon \mf b^+ \to \mb C$.
	\end{itemize}
\end{definition}

The latter is the standard definition of a finite Verma module.
Analogously to the above, the finite tame module is canonically embedded as a $U(\mf g)$-submodule, namely as the subspace $\wh W^- \cap W = U(\mf g) w \subseteq W$.

Overall there is an identification of left $U \bigl( \wh{\mf n}^- \bigr)$-modules
\begin{equation}
	\label{eq:vector_space_isomorphism_singular_module}
	\wh W \simeq U \bigl( \wh{\mf n}^- \bigr) \otimes_{U (\mf n^-)} U (\mf n^-_p ) \, ,
\end{equation}
independent of $\chi$.

\subsection{Algebricity}

The structure of $\wh W$ as left-module is controlled by algebraic elements, not by arbitrary formal power series.

More precisely define
\begin{equation*}
	\mc L \mf g_{\alg} = \mf g \bigl[ z^{\pm 1} \bigr] \ceqq \mf g \otimes \mb C \bigl[ z^{\pm 1} \bigr] \, ,
\end{equation*}
and then $\wh{\mf g}_{\alg} \twoheadrightarrow \mc L\mf g_{\alg}$ using the restriction of the cocycle~\eqref{eq:cocycle_affine_lie_algebra}.
These are the \emph{algebraic} loop algebra and the \emph{algebraic} affine Lie algebra of $\mf g$, respectively.
Replacing `$\mf g \llbracket z \rrbracket$' by `$\mf g [z]$' in~\eqref{eq:affine_singular_module} then yields left $U \bigl( \wh{\mf g}_{\alg} \bigr)$-modules, temporarily denoted by $\wh W_{\alg}$, generated by a cyclic vector $w_{\alg} \in \wh W_{\alg}$.

On the other hand the modules $\wh W$ are left $U \bigl( \wh{\mf g}_{\alg} \bigr)$-modules via the inclusion $U \bigl( \wh{\mf g}_{\alg} \bigr) \hookrightarrow U \bigl( \wh{\mf g} \bigr)$, and composing with the canonical projection
\begin{equation}
	U \bigl( \wh{\mf g} \bigr) \twoheadrightarrow \wh W \simeq U \bigl( \wh{\mf g} \bigr) \big\slash \Ann_{U \bigl( \wh{\mf g} \bigr)} (w)
\end{equation}
yields a linear map $\iota \colon U \bigl( \wh{\mf g}_{\alg} \bigr) \to \wh W$.

\begin{lemma}
	The map $\iota$ induces an isomorphism $\wh W_{\alg} \simeq \wh W$ of left $U \bigl( \wh{\mf g}_{\alg} \bigr)$-modules.
\end{lemma}

\begin{proof}
	By~\eqref{eq:annihilator_cyclic_vector} the map $\iota$ is surjective, since $\wh W$ is generated by the cyclic vector over $U \bigl( \mc L \mf g_{\alg} \bigr)$.
	Then its kernel is
	\begin{equation*}
		\Ker (\iota) = \Ann_{U \bigl( \wh{\mf g} \bigr)} (w) \cap U \bigl( \wh{\mf g}_{\alg} \bigr) = \Ann_{U \bigl( \wh{\mf g}_{\alg} \bigr)} (w_{\alg}) \, . \qedhere
	\end{equation*}
\end{proof}

Hence the action of meromorphic $\mf g$-valued functions on the singular modules is given by Laurent polynomials only.
(Hereafter we drop the subscript `alg' from all notations.)

\section{Relation with (irregular) meromorphic connections}
\label{sec:meromorphic_connections}

There are canonical vector space isomorphisms $\bigl( \mf g \otimes z^i \bigr)^{\dual} \simeq \mf g \otimes z^{-(i + 1)} \dif z$, for $i \in \mb Z$.
They are induced from the nondegenerate residue-pairing
\begin{equation}
	\label{eq:residue_pairing}
	\mc L \mf g \dif z \times \mc L \mf g \longrightarrow \mb C, \qquad (X \otimes \omega,Y \otimes g) \longmapsto (X \mid Y) \cdot \Res_{z = 0}(g \omega) \, ,
\end{equation}
where $\mc L \mf g \dif z \ceqq \mf g \otimes \mb C (\!( z )\!) \dif z$, $G$ is a (simple) Lie group with Lie algebra $\mf g$, and $\mc L G$ is the associated loop group.

Thus after fixing a level $\kappa \in \mb C$ the families of singular modules~\eqref{eq:affine_singular_module} and~\eqref{eq:finite_singular_module} are both naturally parametrised by elements
\begin{equation}
	\label{eq:principal_part}
	\mc{A} = \dif Q + \Lambda \frac{\dif z}{z} \in z^{-1} \mf h \bigl[ z^{-1} \bigr] \dif z \, .
\end{equation}
Namely the residue term $\Lambda z^{-1} \dif z \in \mf h \otimes z^{-1} \dif z$ corresponds to the tame part $\lambda \in \mf h^{\dual}$ of a singular character, and the irregular type
\begin{equation}
	\label{eq:irregular_type}
	Q = \sum_{i = 1}^{p - 1} \frac{A_i}{z^i} \in \mf h (\!( z )\!) \big\slash \mf h \llbracket z \rrbracket \ \, , \qquad \text{with } A_i \in \mf h \text{ for all } i  \, ,
\end{equation}
is such that $\dif \, ( A_i z^{-i}) = -i A_i z^{-i-1} \dif z \in \mf h \otimes z^{-i-1} \dif z$ corresponds to the wild part $a_i \in \bigl( \mf h \otimes z^i \bigr)^{\dual}$.
The meromorphic 1-forms~\eqref{eq:principal_part} are viewed as principal parts of germs of meromorphic connections at a point on a Riemann surface (with semisimple formal residue and untwisted irregular type; here we are considering `very good' orbits in the terminology of~\cite{boalch_2017_wild_character_varieties_meromorphic_hitchin_systems_and_dynkin_graphs}).

As mentioned in the introduction, the important facts are that:
\begin{enumerate}
	\item $\mf g_p = \Lie(G_p)$, where $G_p \ceqq G \bigl( \mb C \llbracket z \rrbracket \bigl\slash z^p \mb C \llbracket z \rrbracket \bigr)$ is the group of $(p-1)$-jets of bundle automorphisms for the trivial principal $G$-bundle on a (formal) disc;

	\item the level-zero complex symplectic reduction for the diagonal $G$-action---on products of coadjoint $G_p$-orbits---yields a description of (the open part of) a de Rham space, parameterising isomorphism classes of irregular-singular meromorphic connections on a holomorphically-trivial principal $G$-bundle over the Riemann sphere (with prescribed positions of poles and irregular types~\cite[\S~5]{boalch_2007_quasi_hamiltonian_geometry_of_meromorphic_connections}; see ~\cite{boalch_2001_symplectic_manifolds_and_isomonodromic_deformations} for $G = \GL_m(\mb C)$).
\end{enumerate}

Moreover the diagonal $G$-action will correspond to taking $\mf g$-coinvariants for the tensor product of finite singular modules, generalising the tame case (see \S\S~\ref{sec:coinvariants} and~\ref{sec:irregular_conformal_blocks_space_with_duals}).

\begin{remark}[Birkhoff groups/Lie algebras]
	\label{rem:birkhoff}

	Consider the subgroup $B_p \subseteq G_p$ of elements with constant term 1.
	Then $G$ acts on $B_p$ by conjugation, and there are natural identification $G_p \simeq G \ltimes B_p$ and $\mf g_p \simeq \mf g \ltimes \mf b_p$, where $\mf b_p = \Lie(B_p)$.\footnote{Beware to distinguish the positive/negative deeper Borel subalgebra $\mf b^{\pm}_p$ from the Birkhoff subalgebra $\mf b_p$.}
	This yields a vector space decomposition $\mf g^{\dual}_p \simeq \mf g^{\dual} \oplus \mf b_p^{\dual}$: in the duality~\eqref{eq:residue_pairing} the former summand corresponds to formal residues with zero irregular types, and the latter to irregular types with zero residue (so in particular $q \in \mf b_p^{\dual}$).
\end{remark}

\section{Bases, gradings and filtrations}
\label{sec:bases_gradings_filtrations}

Denote by $\Pi = \Set{\theta_i}_i \subseteq \mc R^+$ the base of simple roots, and choose an order $\mc R_+ = (\alpha_1, \dc, \alpha_s)$ for the system of positive roots.
If $r \ceqq \rk(\mf g)$ then we may assume that $\theta_i = \alpha_i$ for $i \in \Set{ 1,\dc,r}$.
Let then $(F_{\alpha})_{\alpha \in \mc R^+}$ and $(E_{\alpha})_{\alpha \in \mc R^+}$ be bases of $\mf n^-$ and $\mf n^+$ with $(F_{\alpha}, E_{\alpha}) \in \mf g_{-\alpha} \oplus \mf g_{\alpha}$, and such that $( F_{\alpha}, H_{\alpha} \ceqq [E_{\alpha},F_{\alpha}],E_{\alpha} )$ is an $\mf{sl}_2$-triple.
(We may at times write $E_{-\alpha} \ceqq F_{\alpha}$ for the sake of a uniform notation.)

In particular $(H_{\theta})_{\theta \in \Pi}$ is a basis of $\mf h$, and we get an ordered Cartan--Weyl basis of $\mf g$:
\begin{equation}
	\label{eq:cartan_weyl_basis}
	(X_1, \dc, X_{2s + r}) \ceqq (F_{\alpha_1}, \dc, F_{\alpha_s}, H_{\theta_1}, \dc, H_{\theta_r}, E_{\alpha_1}, \dc, E_{\alpha_s}) \, .
\end{equation}

For a multi-index $\bm n \in \mb Z_{\geq 0}^{2s + r}$ define
\begin{equation}
	X^{\bm n} \ceqq X_1^{n_1} \dm X_{2s+r}^{n_{2s+r}} \in U (\mf g) \, .
\end{equation}
By the PBW theorem these monomials provide a $\mb C$-basis of $U (\mf g)$ (cf. the monograph~\cite{dixmier_1996_algebres_enveloppantes}).

\subsection{PBW-bases of singular modules}

Let $\bm \beta = (\beta_i)_{i \geq 0}$ be a sequence of non-negative integers with finite support, and consider another sequence $\bm k$ with values in the index set of the Cartan--Weyl basis~\eqref{eq:cartan_weyl_basis}, i.e. $\bm k = (k_i)_{i \geq 0} \in \Set {1,\dc,r+2s }^{\mb Z_{\geq 0}}$.
Then define
\begin{equation}
	X_{\bm k} z^{\bm \beta} \ceqq \prod_{i \in \bm \beta^{-1} ( \mb Z_{> 0} )} X_{k_i} z^{\beta_i} \in U \bigl( \mc L\mf g_{\alg} \bigr) \, .
\end{equation}

\begin{lemma}[PBW-basis of algebraic affine enveloping algebras]~\newline
	\label{lem:pbw_basis_enveloping_algebra_affine}
	A $\mb C$-basis of $U \bigl( \mc L\mf g_{\alg} \bigr)$ is given by
	\begin{equation}
		\label{eq:pbw_basis_enveloping_algebra_affine}
		\mc B \ceqq \Set{ X_{\bm k'} z^{\bm{-\beta}'} \cdot X^{\bm n} \cdot X_{\bm k} z^{\bm \beta} }_{\bm k',\bm \beta',\bm n,\bm k,\bm \beta} \, ,
	\end{equation}
	where $\bm \beta'$ is nonincreasing, $\bm \beta$ is nondecreasing, and $k'_j \leq k'_{j+1}$ (resp. $k_j \leq k_{j+1}$) if $\beta_j = \beta_{j+1}$ (resp. $\beta'_j = \beta'_{j+1}$).
\end{lemma}

This is one statement of the PBW theorem for the countable-dimensional Lie algebra $\mc L \mf g_{\alg} = \mf g \otimes \mb C \bigl[ z^{\pm 1} \bigr]$---we consider monomials over a totally-ordered basis.

\begin{corollary}[PBW-basis of affine singular modules]~\newline
	\label{cor:pbw_basis_affine_singular_module}
	A $\mb C$-basis of the affine singular module $\wh W$ can be extracted from
	\begin{equation}
		\label{eq:pbw_basis_affine_singular_module}
		\mc B_{\wh W} \ceqq \Set{ X_{\bm k'} z^{-\bm \beta'} \cdot X^{\bm n} \cdot X_{\bm k} z^{\bm \beta} w }_{\bm k',\bm \beta',\bm n,\bm k,\bm \beta} \, ,
	\end{equation}
	where $\bm \beta'$, $\bm k'$, $\bm n$, $\bm k$, and $\bm \beta$ are as above.
\end{corollary}

\begin{proof}
	The family generates over $\mb C$ since $ U \bigl( \mc L\mf g_{\alg} \bigr) w = \wh W$, and using Lem.~\ref{lem:pbw_basis_enveloping_algebra_affine}.
\end{proof}

\begin{remark*}
	\label{rem:generating_vectors_singular_module}
	In~\eqref{eq:pbw_basis_affine_singular_module} one may take $\bm \beta$ bounded above by $p-1$, as $z^p \mf g \llbracket z \rrbracket w = (0)$.
\end{remark*}


\begin{lemma}
	\label{lem:smooth_modules}
	The singular modules are smooth.
\end{lemma}

\begin{proof}
	This is clear in the finite case, as $z^p \mf g \llbracket z \rrbracket W = (0)$.

	In the affine case choose $X \in \mf g$ and an element $\wh w = X_{\bm k'} z^{-\bm \beta'} X^{\bm n} X_{\bm k} z^{\bm \beta} w$ of~\eqref{eq:pbw_basis_affine_singular_module}.
	Then the vanishing $X z^{N} \wh w = 0$ holds for
	\begin{equation}
		N \geq p + \sum_{i \geq 0} \beta'_i \in \mb Z_{> 0} \, ,
	\end{equation}
	and the conclusion follows since~\eqref{eq:pbw_basis_affine_singular_module} is a set of generators.
\end{proof}

Then we will state the PBW theorem for the finite-dimensional Lie algebras $\mf g_p$ and $\mf n^-_p$, as follows:

\begin{lemma}[PBW-basis of depth-$p$ finite enveloping algebras]~\newline
	\label{lem:pbw_basis_enveloping_algebra_finite}
	A $\mb C$-basis of $U \bigl( \mf g_p \bigr)$ is given by
	\begin{equation}
		\label{eq:pbw_basis_finite_enveloping_algebra}
		\mc B \ceqq \Set{ X^{\bm n} \cdot X_{\bm k} z^{\bm \beta} }_{\bm n,\bm k,\bm \beta} \, ,
	\end{equation}
	where $\bm n$, $\bm k$ and $\bm \beta$ are as above, with the condition of Rmk.~\ref{rem:generating_vectors_singular_module}.
	Moreover restricting to $X_i, X_{k_j} \in \mf n^-$ for $i \in \Set{ 1,\dc, 2s+r }$ and $j \geq 0$ yields a $\mb C$-basis of $U (\mf n^-_p )$.
\end{lemma}

\begin{corollary}[PBW-basis of finite singular modules]~\newline
	\label{cor:pbw_basis_finite_singular_module}
	A $\mb C$-basis of the finite singular module $W \subseteq \wh W$ is given by
	\begin{equation}
		\label{eq:pbw_basis_finite_singular_module}
		\mc B_W \ceqq \Set{ X^{\bm n} \cdot X_{\bm k} z^{\bm \beta} w }_{\bm n,\bm k,\bm \beta} \, ,
	\end{equation}
	where all conditions of Lem.~\ref{lem:pbw_basis_enveloping_algebra_finite} apply.
\end{corollary}

\begin{proof}
	The family generates since $W = U (\mf n^-_p ) w$, and using Lem.~\ref{lem:pbw_basis_enveloping_algebra_finite} (the generating part).
	But $U(\mf n_p^-)$ has trivial intersection with the annihilator of $w$, hence the family is free by Lem.~\ref{lem:pbw_basis_enveloping_algebra_finite} (the linear independence part).
\end{proof}

In particular $W$ is a free rank-1 left $U(\mf n^-_p)$-module.

\subsection{Gradings for z-degree}
\label{sec:gradings}

We first define two positive $\mb Z$-gradings on $\wh W$.

\begin{definition}
	\label{def:z_grading_affine_singular_module}
	Choose $k \in \mb Z$. Then:
	\begin{itemize}
		\item the subspace $\wh{\mc F}^-_k = \wh{\mc F}^-_k\bigl( \wh W \bigr) \subseteq \wh W$ is the $\mb C$-span of the vectors of~\eqref{eq:pbw_basis_affine_singular_module} with $\sum_i \beta'_i = k$;

		\item the subspace $\wh{\mc F}^+_k = \wh{\mc F}^+_k \bigl(\wh W\bigr) \subseteq \wh W$ is the $\mb C$-span of the vectors of~\eqref{eq:pbw_basis_affine_singular_module} with $\sum_i \beta_i = k$.
	\end{itemize}
\end{definition}

By definition $\wh{\mc F}^-_0 = W$, $\wh{\mc F}^+_0 = \wh W^-$, and
\begin{equation}
	\label{eq:negative_degree_grading_shift}
	\mf g \otimes z^{-i} \bigl( \wh{\mc F}^-_k \bigr) = \wh{\mc F}^-_{k + i} \, , \qquad \text{for } i \geq 0 \, .
\end{equation}
In particular $\bigl(\wh W,\wh{\mc F}^-_{\bullet} \bigr)$ is a $\mb Z$-graded $\mf g \bigl[ z^{-1} \bigr]$-module, where $\mf g \bigl[ z^{-1} \bigr]$ is a $\mb Z$-graded Lie algebra with grading defined by $\deg (\mf g \otimes z^{-i}) = i$.

The other grading instead does not yield a graded module;
but we can obtain one inducing a (positive) grading on $W \subseteq \wh W$.

\begin{definition}
	\label{def:z_grading_finite_singular_module}
	For $k \in \mb Z$ set $\mc F^+_k \ceqq \wh{\mc F}^+_k \cap W$.
\end{definition}

It follows that $\mc F^+_0 = U(\mf g) w \subseteq W$, and
\begin{equation}
	\label{eq:positive_degree_grading_shift}
	\mf n^- \otimes z^i \bigl( \mc F^+_k \bigr) \subseteq \mc F^+_{k+i} \, , \qquad \text{for } k,i \geq 0 \, ,
\end{equation}
so the space $\bigl(W,\mc F^+_{\bullet} \bigr)$ is a $\mb Z$-graded $\mf n^- \llbracket z \rrbracket$-module, where $\mf n^- \bigl \llbracket z \rrbracket$ is a $\mb Z$-graded Lie algebra with grading defined by $\deg (\mf n^- \otimes z^i) = i$.

\subsection{Filtrations}
\label{sec:filtrations}

We consider the filtration $\wh{\mc F}^-_{\leq \bullet}$ on $\wh W$ associated with the grading of Def.~\ref{def:z_grading_affine_singular_module} for the negative $z$-degree.
It follows from~\eqref{eq:negative_degree_grading_shift} that
\begin{equation}
	\label{eq:negative_degree_filtration_shift}
	\wh{\mc F}^-_{\leq k + 1} = \sum_{m + l = k} \mf g \otimes z^{-m-1} \bigl( \wh{\mc F}^-_l \bigr) \, , \qquad \mf g \otimes z^i \bigl( \wh{\mc F}^-_{\leq k} \bigr) \subseteq \wh{\mc F}^-_{\leq k} \, ,
\end{equation}
for $k,i \geq 0$.

Finally we consider on $U (\mf g) w = U (\mf n^-) w \subseteq W$ the natural filtration $\mc{E}_{\leq \bullet}$ induced from that of $U (\mf n^-)$, so that $\mc{E}_{\leq 0} = \mb C w$.
Note
\begin{equation}
	\label{eq:constant_filtration_shift}
	\mf n^- \bigl( \mc{E}_{\leq k} \bigr) + \mc{E}_{\leq k} = \mc{E}_{\leq k+1} \, ,
\end{equation}
and further $\mf n^-$ acts nontrivially on the associated graded of $\bigl( U(\mf g)w, \mc{E}_{\leq \bullet} \bigr)$:
\begin{equation}
	\label{eq:constant_grading_shift}
	\mf n^- \bigl( \gr(\mc{E})_k \bigr) \subseteq \gr(\mc{E})_{k+1} \, ,
\end{equation}
where as customary $\gr(\mc{E})_k \ceqq \mc{E}_{\leq k} \big\slash \mc{E}_{\leq k-1}$ for $k \in \mb Z_{\geq 0}$---and $\mc{E}_{\leq -1} \ceqq (0)$.

\subsection{Weight gradings}
\label{sec:weight_gradings}

For $\mu \in \mf h^{\dual}$ define
\begin{equation}
	\wh{\mc F}_{\mu} \bigl( \wh W \bigr) = \wh{\mc F}_{\mu} \ceqq \Set{ \wh w \in \wh W | H \wh w = \mu (H) \wh w \text{ for } H \in \mf h } \subseteq \wh W \, ,
\end{equation}
and analogously $\mc F_{\mu}(W) = \mc F_{\mu} \ceqq W \cap \wh{\mc F}_{\mu} \subseteq W$.

\begin{proposition}
	\label{prop:diagonalisable_singular_modules}
	The singular modules are $\mf h$-semisimple, i.e.
	\begin{equation}
		\wh W = \bigoplus_{\mu \in \mf h^{\dual}} \wh{\mc F}_{\mu} \, , \qquad W = \bigoplus_{\mu \in \mf h^{\dual}} \mc F_{\mu} \, .
	\end{equation}
\end{proposition}

\begin{proof}
	This follows from the fact that all elements of~\eqref{eq:pbw_basis_affine_singular_module} and~\eqref{eq:pbw_basis_finite_singular_module} are $\mf h$-weight vectors, which in turn is proven recursively using the identities
	\begin{equation}
		\label{eq:weight_grading_shift_I}
		H \cdot X_{\alpha} z^i \wh w
		= \Braket{ \mu + \alpha | H } X_{\alpha} z^i \wh w \, , \qquad H \cdot H' z^i \wh w
		= \Braket{ \mu | H } \cdot H' z^i \wh w \, ,
	\end{equation}
	for $\alpha \in \mc R$, $H,H' \in \mf h$, $i \in \mb Z$ and $\wh w \in \wh{\mc F}_{\mu}$.
\end{proof}

\begin{remark*}
	In the finite case one may define the $\mf h_p$-weight spaces, i.e. the subspaces of vectors $\wh w \in W$ such that $H z^i \wh w = \Braket{ \mu_i | H z_i } \wh w$ for $\bm{\mu} = (\mu_0,\dc,\mu_{p-1}) \in \mf h_p^{\dual}$.
	However the very first recursion fails for $p \geq 2$: if $H \in \mf h$ is such  that $\Braket{ \alpha | H } \neq 0$ then
	\begin{equation}
		Hz \cdot X_{-\alpha} w = \Braket{ a_1 | Hz } X_{-\alpha} w - \Braket{ \alpha | H } X_{-\alpha} z \cdot w \not\in \mb C ( X_{-\alpha} w) \, ,
	\end{equation}
	where $w$ is the cyclic vector, so the finite singular modules are not $\mf h_p$-semisimple.
\end{remark*}

The proof of Prop.~\ref{prop:diagonalisable_singular_modules} implies that the weights are contained inside $\lambda + Q \subseteq \mf h^{\dual}$, where $Q \ceqq \mb Z \mc R$ is the root lattice.

\begin{remark*}
	Consider the $z$-linear extension of the adjoint action $\mf h \to \mf{gl}(\mf g)$ on $\mc L\mf g$.
	Decomposing $\mc L\mf g = \bigoplus_{\alpha \in \mc R} \mc L\mf g_{\alpha} \oplus \mc L \mf h$ we see that $\mc L\mf g$ is naturally a $\mf h^{\dual}$-graded Lie algebra (with nontrivial weights still given by $\mc R \cup \Set{0}$), and the proof of Prop.~\ref{prop:diagonalisable_singular_modules} shows that the singular modules are $\mf h^{\dual}$-graded.
\end{remark*}

In the finite case one can go further, recovering the standard notion of positivity.
Namely $\bigl( \mf h^{\dual},\preceq \bigr)$ is a poset, by letting $\mu' \preceq \mu$ if $\mu - \mu' \in Q_+$, where $Q_+ \ceqq \mb Z_{\geq 0} \mc R^+ \subseteq Q$ is the \emph{positive part} of the root lattice.

\begin{lemma}
	\label{lem:highest_weight}
	One has $\mc F_{\lambda} = \mb C w$ and $W = \bigoplus_{\mu \preceq \lambda} \mc F_{\mu}$.
\end{lemma}

\begin{proof}
	It follows from the fact that $W$ is generated over $U( \mf n^-_p )$ by a $\mf h_p$-weight vector annihilated by $\mf n^+_p$: it is a highest-weight $\mf g_p$-module.
\end{proof}

In particular~\eqref{eq:pbw_basis_finite_singular_module} consists of weight vectors, and the line $\mb C w \subseteq W$ has the highest weight.

In view of Lem.~\ref{lem:highest_weight} the weight spaces are naturally parametrised by elements $\nu \in Q_+$, via $\mc F_{\nu} \ceqq \mc F_{\lambda - \nu}$.
Now for an element $\nu \in \mf h^{\dual}$ denote
\begin{equation}
	\label{eq:multiplicity_positive_root_lattice}
	\Mult_{\mc R^+}(\nu) \ceqq \Set{ \bm m = (m_{\alpha})_{\alpha} \in \mb Z_{\geq 0}^{\mc R^+} | \sum_{\alpha \in \mc R^+} m_{\alpha} \cdot \alpha = \nu } \subseteq \mb Z_{\geq 0}^{\mc R^+} \, ,
\end{equation}
so that the cardinality of $\Mult_{\mc R^+}(\nu)$ is the \emph{finite} number of ways of expressing $\nu$ as a $\mb Z_{\geq 0}$-linear combination of positive roots. In particular $\Mult_{\mc R^+}(0) = \Set{ \bm{0} }$, and $\Mult_{\mc R^+}(\nu) = \varnothing$ for $\nu \not\in Q_+$.

Finally for $\bm m \in \mb Z_{\geq 0}^{\mc R^+}$ denote
\begin{equation}
	\WComp_p (\bm m) \ceqq \Set{ \bm \varphi = (\varphi_{\alpha})_{\alpha} | \varphi_{\alpha} \colon \Set{0,\dc, p-1} \to \mb Z_{\geq 0} \, , \sum_{i = 0}^{p-1} \varphi_{\alpha}(i) = m_{\alpha} } \, ,
\end{equation}
which is the \emph{finite} set of weak $p$-compositions of the integers $m_{\alpha} \geq 0$.\footnote{A \emph{composition} of $m_{\alpha}$ is a sequence of positive integers summing to $m_{\alpha}$; it is a $p$-composition if the sequence has finite length $p \geq 1$; and it is \emph{weak} if zero is allowed.}
In particular $\WComp_1 (\bm m)$ is a singleton containing the element $\bm \varphi$ with $\varphi_{\alpha}(0) = m_{\alpha}$ for all $\alpha \in \mc R^+$.

\begin{proposition}
	\label{prop:finite_dimensional_weight_spaces}
	For $\nu \in \mf h^{\dual}$ one has
	\begin{equation}
		\label{eq:dimension_weight_space}
		\dim \bigl( \mc F_{\nu} \bigr) = \sum_{\bm m \in \Mult_{\mc R^+}(\nu)} \binom{\bm m+p-1}{\bm m} < \infty \, ,
	\end{equation}
	where $\binom{\bm m+p-1}{\bm m} \ceqq \prod_{\alpha \in \mc R^+} \binom{m_{\alpha}+p-1}{m_{\alpha}}$.
\end{proposition}

\begin{proof}
	Choose $\mu \in \mf h^{\dual}$ and set $\nu = \lambda - \mu$.
	Then for $\bm m \in \Mult_{\mc R^+}(\nu)$ and $\bm \varphi \in \WComp_p (\bm m)$ consider the vector
	\begin{equation}
		\label{eq:basis_vector}
		w^{\bm \varphi} \ceqq \prod_{i = 0}^{p-1} \Biggl( \, \prod_{\alpha \in \mc R^+} \bigl( X_{-\alpha} z^i \bigr)^{\varphi_{\alpha}(i)} w \Biggr) \in \mc B_W \, .
	\end{equation}
	The family $\Set{ w^{\bm \varphi} }_{\bm \varphi} \subseteq W$ is free since it consists of distinct elements extracted from~\eqref{eq:pbw_basis_finite_singular_module} (beware of the ordering in the product), and by construction $w^{\bm \varphi} \in \mc F_{\nu}$.

	Conversely the vectors~\eqref{eq:basis_vector} exhaust~\eqref{eq:pbw_basis_finite_singular_module}, from which one can extract a basis of $\mc F_{\nu}$, so the conclusion follows from standard combinatorial identities.
\end{proof}

Thus Prop.~\ref{prop:finite_dimensional_weight_spaces} strengthen Lem.~\ref{lem:highest_weight}: the given sum is empty for $\nu \not\in Q_+$, and $\WComp_p (\bm{0})$ is a singleton containing the element $\bm \varphi$ with $\varphi_{\alpha}(i) = 0$ for $i \in \Set{0,\dc,p-1}$.

As expected~\eqref{eq:dimension_weight_space} generalises the standard fact that $\dim \bigl( \mc F_{\nu} \bigr) = \abs{\Mult_{\mc R^+}(\nu)}$ for Verma modules, i.e. it generalises the character of Verma modules.
The difference in the general case is that one must also specify a $z$-degree for each occurrence of a positive root.

\begin{remark*}
	This notion of positivity is lost with the (finite) modules of \S~\ref{sec:dynamical_term}: in particular they have \emph{infinite}-dimensional weight spaces.
\end{remark*}

For example consider the case where $\nu = \theta \in \Pi$ is a simple root.
One has $\Mult_{\mc R^+}(\theta) = \Set{ \bm m^{\theta} }$, with $\bm m^{\theta}_{\alpha} \ceqq \delta_{\theta,\alpha}$.
Also $\WComp_p (\bm m^{\theta}) = \Set{ \bm \varphi^{\theta,i}}_i$, where
\begin{equation*}
	\varphi^{\theta,i}_{\alpha} (j) = \delta_{\alpha,\theta}\delta_{ij} \, , \qquad  i,j \in \Set{ 0,\dc,p-1} \, .
\end{equation*}
Hence $X_{\bm \varphi^{\theta,i}} = X_{-\theta} z^i$, so we recover
\begin{equation*}
	\dim \bigl( \mc F_{\theta} \bigr) = p \, , \qquad \mc F_{\theta} = \spann_{\mb C} \Set{ X_{-\theta} w, \dc, X_{-\theta} z^{p-1} w } \, .
\end{equation*}

\begin{remark*}
	It follows from the above that
	\begin{equation}
		U \bigl( \mf n^+ \llbracket z \rrbracket \bigr) \mc F_{\nu} = \bigoplus_{0 \preceq \nu' \preceq \nu} \mc F_{\nu'} \, , \qquad \text{for } \nu \in Q_+ \, .
	\end{equation}
	Hence the module $W$ is \emph{locally} $\mf n^+\llbracket z \rrbracket$-finite, i.e. the vector spaces $U \bigl( \mf n^+ \llbracket z \rrbracket \bigr) \wh w \subseteq W$ are finite-dimensional for all $\wh w \in W$.

	One thus expects that $W$ lies in a `Bernstein--Gelfand--Gelfand category $\mc O \llbracket z \rrbracket$'~\cite{humphreys_2008_representations_of_semisimple_lie_algebras_in_the_BGG_category_O}---of $\mf h$-semisimple finitely generated left $U \bigl( \mf g \llbracket z \rrbracket)$-modules which are locally $\mf n^+\llbracket z \rrbracket$-finite.
	And indeed:~\cite{chaffe_2023_category_o_for_takiff_lie_algebras,chaffe_topley_2023_category_o_for_truncated_current_lie_algebras}.
\end{remark*}

\subsubsection{Archetypal case}
\label{sec:archetypal_case}

It is easy to give a closed expression for the dimension when $\mf g = \mf{sl}(2,\mb C)$, with the standard basis $(F,H,E)$ and the standard $A_1$-root system $\mc R = \Set{ \pm \alpha }$---i.e. $\alpha$ is positive and $\Braket{ \alpha | H } = 2$.
Here $Q_+ = \mb Z_{\geq 0} \alpha$, so simply $\Mult_{\mc R^+} (\nu) = \Set{ m }$ for elements $\nu = m\alpha$ with $m \in \mb Z_{\geq 0}$.

Thus~\eqref{eq:dimension_weight_space} reduces to
\begin{equation}
	\label{eq:dimension_weight_space_sl_2}
	\dim \bigl( \mc F_{m\alpha} \bigr) = \abs{ \WComp_p(m)} = \binom{m+p-1}{m}.
\end{equation}
In the tame case one recovers the line generated by $F^m v$, whereas in the general case a basis is given by
\begin{equation}
	\label{eq:basis_weight_space_sl_2}
	w^{\varphi} = \prod_{i = 0}^{p-1} \bigl( F z^i \bigr)^{\varphi(i)} \cdot v \, , \qquad \text{for } \varphi \in \WComp_p (m) \, .
\end{equation}

\section{Dual modules}
\label{sec:dual_modules}

In view of Prop.~\ref{prop:diagonalisable_singular_modules} we consider the restricted duals of the $\mf h^{\dual}$-graded singular modules, i.e. the $\mf h^{\dual}$-graded vector spaces
\begin{equation}
	\label{eq:restricted_duals}
	\wh W^* \ceqq \bigoplus_{\mu \in \mf h^{\dual}} \wh{\mc F}_{\mu}^{\dual} \subseteq \wh W^{\dual} \, , \qquad W^* \ceqq \bigoplus_{\mu \in \mf h^{\dual}} \mc F_{\mu}^{\dual} \subseteq W^{\dual} \, .
\end{equation}
They are naturally equipped with a \emph{right} $U \bigl( \wh{\mf g} \bigr)$- and $U \bigl( \mf g_p \bigr)$-module structure (respectively), namely
\begin{equation}
	\Braket{\wh \psi X z^i | \wh w} = \Braket{ \wh \psi | Xz^i \wh w } \, , \quad \wh \psi K = \kappa \wh \psi \, , \qquad\text{for } i \in \mb Z, \, X \in \mf g, \, \wh \psi \in \wh W^*, \, \wh w \in \wh W \, ,
\end{equation}
and analogously in the finite case.

To get a left action, let us compose with a Lie-algebra morphism $\wh{\mf g} \to \wh{\mf g}^{\op}$ (resp. $\mf g_p \to \mf g_p^{\op})$, or rather with the induced ring morphism $U \bigl( \wh{\mf g} \bigr) \to U \bigl( \wh{\mf g}^{\op} \bigr) = U \bigl( \wh{\mf g} \bigr)^{\op}$ (resp. $U (\mf g_p) \to U(\mf g_p)^{\op}$).
In particular a Lie-algebra morphism $\theta \colon \mf g \to \mf g^{\op}$ has a unique $\mb Z$-graded extension $\wh{\theta} \colon \mc L \mf g \to \mc L \bigl( \mf g^{\op} \bigr) = \bigl(\mc L\mf g\bigr)^{\op}$: in the finite case one can then consider the restriction $\wh{\theta} \colon \mf g \llbracket z \rrbracket \to \mf g \llbracket z \rrbracket^{\op}$, which is compatible with the projections $\mf g \llbracket z \rrbracket \twoheadrightarrow \mf g_p$ and $\mf g \llbracket z \rrbracket^{\op} \twoheadrightarrow \mf g_p^{\op}$; in the affine case one may further ask that $\theta$ is $(\cdot \mid \cdot)$-orthogonal, and extend the definition by $\wh{\theta}(K) \ceqq -K$.

In what follows we only consider morphisms of this type.

\begin{definition}[Dual singular modules]~\newline
	The affine (resp. finite) $\theta$-dual singular module $\wh W_{\theta}^*$ (resp. $W_{\theta}^*$) is the left $U \bigl( \wh{\mf g} \bigr)$-module (resp. $U(\mf g_p)$-module) defined by the morphism $\theta \colon \mf g \to \mf g^{\op}$.
\end{definition}

The $U \bigl( \mf g \llbracket z \rrbracket)$-linear inclusion map $W \hookrightarrow \wh W$ then dually corresponds to a $U \bigl( \mf g \llbracket z \rrbracket)$-linear restriction map $\wh W_{\theta}^* \twoheadrightarrow W_{\theta}^*$.

\begin{remark}[Dual/contragredient modules]
	Basic examples of morphisms $\theta \colon \mf g \to \mf g^{\op}$ preserving $( \cdot \mid \cdot)$ are the tautological $\theta_0 = -\on{Id}_{\mf g}$, and the \emph{transposition} $\theta_1$, defined by
	\begin{equation}
		\theta_1 (E_{\alpha}) = E_{-\alpha} \, , \quad \eval[1]{\theta_1}_{\mf h} = \on{Id}_{\mf h} \, , \qquad \text{for } \alpha \in \mc R \, .
	\end{equation}
	We refer to $\theta_0$-duals simply as \emph{dual} modules, and to $\theta_1$-duals as \emph{contragredient} modules.
\end{remark}

Consider then the element $\psi \in W^*$ dual to the cyclic vector in the basis~\eqref{eq:pbw_basis_finite_singular_module}, i.e. $\Braket{ \psi | w } = 1$ and $\psi$ vanishes on all other vectors of~\eqref{eq:pbw_basis_finite_singular_module}---whence  $\mc F_{\lambda}^{\dual} = \mb C \psi$.

Assume hereafter that $\theta(\mf h) = \mf h^{\op}$ (up to conjugating $\theta$ by an inner automorphism of $\mf g$), and canonically identify $\mf h \simeq \mf h^{\op}$ and their duals.
Then we have a well-defined pullback map $\theta^* \in \GL(\mf h^{\dual})$, which we extend $z$-linearly to $\bigl( \mf h \otimes z^i \bigr)^{\dual} \simeq \mf h^{\dual} \otimes z^i$.
Moreover by orthogonality the subspace $\mf n^+ \oplus \mf n^- \subseteq \mf g$ is $\theta$-stable.

\begin{lemma}
	\label{lem:annihilator_cyclic_vector_dual_contragredient}
	The vector $\psi \in W^*$ satifies the relations
	\begin{equation}
		\label{eq:annihilator_cyclic_vector_theta_dual}
		\begin{split}
			z^p \mf g \llbracket z \rrbracket \psi              & = (0) = \theta^{-1} (\mf n^-) \llbracket z \rrbracket \psi \, , \\
			Hz^i \psi = \Braket{ \theta^* a_i | H z^i } \psi \, & , \qquad \text{for } H \in \mf h, \, i \in \Set{0,\dc,p-1} \, .
		\end{split}
	\end{equation}
\end{lemma}

\begin{proof}
	Use ~\eqref{eq:annihilator_cyclic_vector},~\eqref{eq:positive_degree_grading_shift}, the identity $z^p \mf g \llbracket z \rrbracket W = (0)$, and the fact that $\wh{\theta} \colon \mf g \llbracket z \rrbracket \to \mf g \llbracket z \rrbracket^{\op}$ preserves the $z$-grading of Def.~\ref{def:z_grading_finite_singular_module}.
\end{proof}

In particular $\mf n^- \llbracket z \rrbracket \psi = (0)$ in the dual case, and $\mf n^+ \llbracket z \rrbracket \psi = (0)$ in the contragredient case.

\subsection{Dual weight grading}

Write now $\theta_{*} \ceqq \bigl( \theta^* \bigr)^{-1} = \bigl( \theta^{-1} \bigr)^*$, and introduce the notation $\wh{\mc F}^*_{\mu} \subseteq \wh W^*_{\theta}$ and $\mc F^*_{\mu} \subseteq W^*_{\theta}$ for the $\mf h$-weight spaces.

\begin{lemma}
	\label{lem:weight_spaces_dual}
	One has $\wh{\mc F}_{\mu}^{\dual} = \wh{\mc F}^*_{\theta^*\mu}$ and $E_{\alpha} z^i \wh{\mc F}_{\mu}^{\dual} \subseteq \wh{\mc F}_{\mu + \theta_{*}\alpha}^{\dual}$, for $\mu \in \mf h^{\dual}$, $\alpha \in \mc R$, $i \in \mb Z$, and analogously in the finite case---restricting to $i \in \mb Z_{\geq 0}$.
\end{lemma}

\begin{proof}
	Let $\wh{\on{I}}_{\nu} \colon \wh W \to \wh W$ be the idempotent for the direct summand $\wh{\mc F}_{\mu} \subseteq \wh W$, viz. the endomorphism such that $\eval[1]{\wh{\on{I}}_{\mu}}_{\wh W(\mu')} = \delta_{\mu,\mu'} \on{Id}_{\wh W(\mu')}$.
	Then by definition $\wh \psi \in \wh{\mc F}_{\mu}^{\dual}$ means $\wh \psi = \wh \psi \circ \wh{\on{I}}_{\mu}$, and by construction
	\begin{equation}
		\theta(H) \wh{\on{I}}_{\mu} = \wh{\on{I}}_{\mu} \theta(H) = \Braket{ \theta^*\mu | H } \wh{\on{I}}_{\mu} \in \End \bigl( \wh W \bigr) \, , \qquad \text{for } \mu \in \mf h^{\dual} , \, H \in \mf h \, .
	\end{equation}
	Hence for $\wh w \in \wh W$ one has
	\begin{equation}
		\Braket{ H\wh \psi | \wh w } = \Braket{ \wh \psi | \wh{\on{I}}_{\mu} \bigl( \theta(H) \wh w \bigr) } = \Braket{ \theta^*\mu | H } \Braket{ \wh \psi | \wh w } \, ,
	\end{equation}
	whence the inclusion $\wh{\mc F}_{\mu}^{\dual} \subseteq \wh{\mc F}^*_{\theta^*\mu}$, and the equality follows from~\eqref{eq:restricted_duals}.

	The latter inclusion follows from $\theta(E_{\alpha})z^i \wh{\mc F}_{\mu} \subseteq \wh{\mc F}_{\mu - \theta_{*}\alpha}$ for $\alpha \in \mc R$, which is a straightforward computation using~\eqref{eq:affine_lie_bracket}.

	The same pair of arguments applies verbatim to the finite case.
\end{proof}

Hence~\eqref{eq:restricted_duals} is the $\mf h$-weight decomposition of $\theta$-dual singular modules, and the weights are contained inside $\theta^*(\lambda + Q) \subseteq \mf h^{\dual}$ (resp. $\theta^* (\lambda + Q^+)$) in the affine (resp. finite) case.
By Lem.~\ref{lem:highest_weight} we conclude that $\psi \in W^*_{\theta_0}$ is a \emph{lowest}-weight vector of lowest weight $\theta_0^*\lambda = -\lambda$, whereas $\psi \in W^*_{\theta_1}$ is a highest-weight vector of highest weight $\theta_1^*\lambda = \lambda$.

In particular in the contragredient case matching the cyclic vector with its dual yields a canonical morphism $\Phi \colon W \to W^*_{\theta_1}$, hence a generalisation of the Shapovalov form~\cite{shapovalov_1972_a_certain_bilinear_form_on_the_universal_enveloping_algebra_of_a_complex_semisimple_lie_algebra}
\begin{equation}
	S \colon W \otimes W \longrightarrow \mb C \, , \qquad \wh w \otimes \wh w' \longmapsto \Braket{ \Phi\bigl( \wh w \bigr) | \wh w' } \, .
\end{equation}
This may be degenerate, particularly since the image of the canonical morphism is the submodule
\begin{equation*}
	W'_{\theta} \ceqq U \bigl( \mf g \llbracket z \rrbracket \bigr) \psi \subseteq W^*_{\theta} \, ,
\end{equation*}
which in general is a proper submodule.\fn{More general Shapovalov forms are defined and studied in~\cite{calaque_felder_rembado_wentworth_2024_wild_orbits_and_generalised_singularity_modules_stratifications_and_quantisation}, which also provides a conjectural necessary/sufficient condition of nondegeneracy. (A proof of the conjecture might follow from the usage of the parabolic induction functors of~\cite{chaffe_topley_2023_category_o_for_truncated_current_lie_algebras}.)}
Nonetheless we can recursively find the obstruction for $\psi$ to generate the $\theta$-dual module.
To give a necessary condition consider the vector
\begin{equation*}
	\wh w = E_{-\alpha}z^{p-1} w \in \mc F_{\lambda - \alpha} \, , \qquad \alpha \in \mc R^+ \, .
\end{equation*}
By Lem.~\ref{lem:weight_spaces_dual} a linear form $\wh \psi \in W'_{\theta}$ that vanishes on $\mc B_W \setminus \Set{\wh w}$ must lie in the span of $\Set{ \theta^{-1}(E_{\alpha}) \psi, \dc, \theta^{-1}(E_{\alpha})z^{p-1} \psi} \subseteq \mc F_{\lambda - \alpha}^{\dual}$, so consider a generic element
\begin{equation}
	\wh \psi = \wh \psi(b_0,\dc,b_{p-1}) = \sum_{j = 0}^{p-1} b_j \theta^{-1} (E_{\alpha}) z^j \psi \, , \qquad b_j \in \mb C \, .
\end{equation}
Using $z^p \mf g \llbracket z \rrbracket W = (0) = \mf n^+ \llbracket z \rrbracket w$ and $\Braket{ \psi | w } = 1$ yields
\begin{equation}
	\Braket{ \wh \psi | \wh w } = b_{p-1} \Braket{ a_{p-1} | H_{\alpha} z^{p-1} },
\end{equation}
so we need the highest irregular part to be \emph{regular} (cf. \S~\ref{sec:irregular_conformal_blocks}).

Conversely:

\begin{proposition}
	One has $W'_{\theta} = W^*_{\theta}$ by choosing parameters $(\lambda,q)$ in a dense subspace of the affine space $\mf h^{\dual}_p$---with respect to the strong/classical topology.
\end{proposition}

\begin{proof}
	Clearly $\mc F_{\lambda}^{\dual} \subseteq W'_{\theta}$, and then we reason recursively on the $\mf h^{\dual}$-weight space decomposition of $W$.

	Choose $\wh w \in \mc B_W \cap \mc F_{\mu}$, and consider the vectors $\wh w_{\alpha}(k) \ceqq E_{-\alpha} z^k \wh w \in \mc F_{\mu - \alpha}$, for $\alpha \in \mc R^+$ and $k \in \Set{ 0,\dc,p-1}$.
	As $\wh w$, $\alpha$ and $k$ vary, the vectors $\wh w_{\alpha}(k)$ exhaust $\mc B_W \cap \mc F_{\mu - \alpha}$, so we must find coefficients $b_{ij} \in \mb C$ such that $\Braket{ \wh \psi_{\alpha}(i) | \wh w_{\alpha}(k) } = \delta_{ik}$, where
	\begin{equation}
		\wh \psi_{\alpha}(i) = \sum_{j = 0}^{p-1} b_{ij} \theta^{-1} (E_{\alpha}) z^j \wh \psi \in \mc F_{\mu - \alpha}^{\dual} \, , \qquad \text{for } i \in \Set{ 0,\dc,p-1} \, ,
	\end{equation}
	and where $\wh \psi \in \mc F_{\mu}^{\dual}$ is the dual of $\wh w$---lying in $W'_{\theta}$ by the recursive hypothesis.

	Now one has
	\begin{equation}
		\Braket{ \wh \psi_{\alpha}(i) | \wh w_{\alpha}(k) }	= \sum_{j = 0}^{p-1} b_{ij} \Braket{ \wh \psi | E_{\alpha}z^j E_{-\alpha} z^k \wh w } \, ,
	\end{equation}
	hence the given condition means $BM = \on{Id}_{\mb C^p}$, where $B$ and $M$ are the $p$-by-$p$ matrices with coefficients $B_{ij} = b_{ij}$ and $M_{jk} = \Braket{ \wh \psi | E_{\alpha} z^j E_{-\alpha} z^k \wh w }$, respectively (the latter selects the component of $E_{\alpha} z^j E_{-\alpha} z^k \wh w \in \mc F_{\mu}$ along the line $\mb C \wh w$, in the basis~\eqref{eq:pbw_basis_finite_singular_module}).
	A solution exists if and only if $\det (M) \neq 0$.

	Now the determinant of $M = M(\wh w,\alpha)$ is a degree-$p$ polynomial whose coefficients depend polynomially on $(\lambda,q)$, hence it amounts to a polynomial function $\mf h^{\dual}_p \to \mb C$.
	Thus $W'_{\theta} = W^*_{\theta}$ by taking $(\lambda,q)$ in a countable intersection of open dense subsets.
\end{proof}

Finally we can choose a complementary subspace to $W$ inside $\wh W$, and extend $\psi \colon W \to \mb C$ by zero to the whole of $\wh W$---e.g. extract a PBW-basis from~\eqref{eq:pbw_basis_affine_singular_module}.
Then one can consider the module $\wh W'_{\theta} \subseteq \wh W^*_{\theta}$ generated by this extension over $\mc L\mf g$, and define gradings/filtrations on $\wh W'_{\theta} \twoheadrightarrow W'_{\theta}$ analogously to \S\S~\ref{sec:gradings} and~\ref{sec:filtrations}, using the generating set~\eqref{eq:pbw_basis_enveloping_algebra_affine}, the basis~\eqref{eq:pbw_basis_finite_enveloping_algebra}, and the standard filtration of $U \bigl( \theta^{-1} (\mf n^+) \bigr)$.
These satisfy the analogous identities of~\eqref{eq:negative_degree_grading_shift}--\eqref{eq:constant_filtration_shift}.

\section{Segal--Sugawara operators}
\label{sec:sugawara}

For $n \in \mb Z$ define
\begin{equation}
	\label{eq:sugawara_operator}
	L_n \ceqq \frac{1}{2 \bigl( \kappa + h^{\dual} \bigr)} \sum_{j \in \mb Z} \Biggl( \sum_k \colon X_k z^{-j} \cdot X^k z^{n+j} \colon \Biggr) \, ,
\end{equation}
where $(X_k)_k$ and $(X^k)_k$ are $( \cdot \mid \cdot)$-dual bases of $\mf g$, $\kappa \neq -h^{\dual}$ is a \emph{noncritical} level, and in the normal-ordered product one puts elements of $\mf g \llbracket z \rrbracket \subseteq \mc L\mf g$ to the right.

The Sugawara operators~\eqref{eq:sugawara_operator} (due to Segal in this particular form) are well-defined elements of the level-$\kappa$ completion of $U \bigl( \wh{\mf g} \bigr)$ with respect to the chain of left ideals $U \bigl( \wh{\mf g} \bigr) z^{\bullet} \mf g \llbracket z \rrbracket \sse U\bigl(\wh{\mf g} \bigr)$.
If follows from Lem.~\ref{lem:smooth_modules} that there are well-defined actions of~\eqref{eq:sugawara_operator} on the modules $W \subseteq \wh W$.

\subsection{Cyclic vector as Sugawara eigenvector}

The cyclic vector $w \in \wh W$ is a common eigenvector for the Sugawara operators when $n \gg 0$.
To get explicit formul\ae{} for the eigenvalues we recall further euclidean properties of the Cartan--Weyl basis~\eqref{eq:cartan_weyl_basis}.

\begin{remark}[On bases and dualities]~\newline
	\label{rem:root_basis}
	Recall that
	$(H_{\alpha} \mid H_{\alpha}) E^{\pm \alpha} = 2 E_{\mp \alpha}$.
	Using the pairing $( \cdot \mid \cdot) \colon \mf h^{\dual} \otimes \mf h^{\dual} \to \mb C$ induced by the minimal-form duality $\mf h \simeq \mf h^{\dual}$ this can be written $2 E^{\pm \alpha} = (\alpha \mid \alpha) E_{\mp \alpha}$.

	Then we replace the simple-root basis of $\mf h$ with a $( \cdot \mid \cdot)$-orthonormal basis, denoted $(H_k)_k$---i.e. we `divide' by the Cartan matrix---, and for $i \in \mb Z$ we transfer the basis and the pairings to $\mf g \otimes z^i$ and $\bigl( \mf g \otimes z^i \bigr)^{\dual} \simeq \mf g^{\dual} \otimes z^i$ using the canonical vector space isomorphism $\mf g \simeq \mf g \otimes z^i$.
	Then one has the tautological basis-independent identity
	\begin{equation}
		( \mu \mid \mu') = \sum_{k = 1}^{r} \Braket{ \mu | H_k z^i} \Braket{ \mu' | H_k z^j } \, \qquad \text{for } \mu \in \mf h^{\dual} \otimes z^i, \mu' \in \mf h^{\dual} \otimes z^j \, .
	\end{equation}
\end{remark}

Denote as customary $\rho \ceqq \frac{1}{2} \sum_{\alpha \in \mc R^+} \alpha \in \mf h^{\dual}$ the half-sum of positive roots.

\begin{proposition}
	\label{prop:sugawara_eigenvalues}
	The cyclic vector $w$ is a common eigenvector for the operators~\eqref{eq:sugawara_operator} with $n \geq p-1$.
	If $n > 2(p-1)$ then $L_n w = 0$, else $L_n w = l_n w$ with
	\begin{equation}
		\label{eq:sugawara_eigenvalues_I}
		l_n \ceqq \frac{1}{2 \bigl( \kappa + h^{\dual} \bigr)} \sum_{j = 1-p+n}^{p-1} (a_j \mid a_{n-j}) \, , \qquad \text{for } n \in \Set{ p, \dc, 2(p-1) } \, ,
	\end{equation}
	and
	\begin{equation}
		\label{eq:sugawara_eigenvalues_II}
		l_{p-1} \ceqq \frac{1}{2 \bigl( \kappa + h^{\dual} \bigr)} \Biggl( \sum_{j = 0}^{p-1} (a_j \mid a_{p - 1 - j}) + 2p (\rho \mid a_{p-1}) \Biggr) \, .
	\end{equation}
\end{proposition}

\begin{proof}
	Postponed to \S~\ref{sec:proof_prop_sugawara_eigenvalues}.
\end{proof}

Hence the cyclic vector is a Gaiotto--Teschner/Bonelli--Maruyoshi--Tanzini irregular state of order $p-1$~\cite{gaiotto_teschner_2012_irregular_singularities_in_liouville_theory_and_argyres_douglas_type_gauge_theories,bonelli_maruyoshi_tanzini_2012_wild_quiver_gauge_theories}, but arising from modules for affine Lie algebras.

\begin{remark}
	\label{rem:conformal_weight}
	This generalises the standard fact that $L_n v = 0$ for $n > 0$, and that $v$ is an $L_0$-eigenvector, with nonzero eigenvalue for generic values of $\lambda \in \mf h^{\dual}$.
	Namely if $p = 1$ then~\eqref{eq:sugawara_eigenvalues_II} reduces to
	\begin{equation}
		L_0 v = \Delta_{\lambda} v \, , \qquad \Delta_{\lambda} = \frac{(\lambda \mid \lambda + 2\rho)}{2(\kappa + h^{\dual})} \, ,
	\end{equation}
	reverting to the notation $\lambda = a_0$, which recovers the \emph{conformal weight} corresponding to the action of the quadratic Casimir~\eqref{eq:quadratic_casimir}.
\end{remark}

\subsection{Action on finite modules}

Later we will use the action of the operator $L_{-1}$ on the finite module $W \subseteq \wh W$.

Using $z^p \mf g \llbracket z \rrbracket W = (0)$ we see nonvanishing terms arise for $1 - p \leq j \leq p$ in~\eqref{eq:sugawara_operator}, and resolving the ordered product yields
\begin{equation}
	\label{eq:action_sugawara_minus_one_finite_module_general}
	L_{-1} \wh w = \frac{1}{\kappa + h^{\dual}} \sum_{j = 1}^p \Biggl( \sum_k X_kz^{-j} X^k z^{j-1} \Biggr) \wh w \, , \qquad \wh w \in W \, .
\end{equation}

As expected $L_{-1} \wh w \not\in W$, but it can be put back into the finite module via the loop-algebra action (see \S~\ref{sec:coinvariants}).

\begin{remark*}
	We see that~\eqref{eq:action_sugawara_minus_one_finite_module_general} generalises the usual formula from the tame case:
	\begin{equation}
		\label{eq:action_sugawara_minus_one_finite_module_tame}
		\begin{split}
			L_{-1} \wh v = \frac{1}{\kappa + h^{\dual}} \sum_k X_k z^{-1} X^k \wh v \, , \qquad \wh v \in V \, .
		\end{split}
	\end{equation}
\end{remark*}

\section{Irregular blocks: first version}
\label{sec:irregular_conformal_blocks}

Consider the Riemann sphere $\Sigma \ceqq \mb CP^1$, choose an integer $n \geq 1$, and mark distinct points $p_1, \dc, p_n \in \Sigma$.
Denote by $J = \Set{ 1, \dc, n }$ the ordered set of labels for the points, and by $\bm p = ( p_1, \dc, p_n)$ the ordered set of points.

Let $\ms O_{\Sigma}$ be the structure sheaf of regular functions on $\Sigma$, seen as a (nonsingular) complex projective curve.
Then consider the stalks $\ms O_j = \ms O_{\Sigma,p_j}$ at the marked points, their (unique) maximal ideals $\mf M_j = \mf M_{p_j} \subseteq \ms O_j$ of germs of functions vanishing at $p_j$, the completions $\wh{\ms O}_j \ceqq \varprojlim_n \ms O_j \big\slash \mf M^n_j$, and their field of fractions $\wh{\ms O}_j \hookrightarrow \wh{\ms K}_j$.

\begin{remark*}
	If $z_j$ is a local coordinate on $\Sigma$ vanishing at $p_j$ then
	\begin{equation*}
		\ms O_j \simeq \mb C [z_j] \, , \qquad \mf M_j = z_j \mb C [z_j] \, , \qquad \wh{\ms O}_j \simeq \mb C \llbracket z_j \rrbracket \, , \qquad \wh{\ms K}_j \simeq \mb C (\!( z_j ) \!) \, .
	\end{equation*}
	More generally this follows by choosing a uniformiser, i.e. a generator of the maximal ideal(s).
\end{remark*}

Then consider the loop algebras $(\mc L\mf g)_j \ceqq \mf g \otimes \wh{\ms K}_j$ and the associated affine Lie algebras $\wh{\mf g}_j \twoheadrightarrow (\mc L\mf g)_j$.
There are canonical isomorphisms $\wh{\mf g}_i \simeq \wh{\mf g}_j$ for $i,j \in J$, and the subscripts distinguish the local picture at the marked points.

Now for $j \in J$ choose also an integer $r_j \geq 1$, and define singular modules as in \S~\ref{sec:setup}.
Hence consider the Lie subalgebras $\mf S^{(r_j)} \subseteq \wh{\mf g}_j$, a common level $\kappa \in \mb C$ for the central elements, and singular characters $\chi_j = \chi(\lambda_j,q_j,\kappa)$, where $\lambda_j \in \mf h^{\dual}$ and
\begin{equation*}
	q_j = \bigl( (a_j)_1, \dc, (a_j)_{r_j - 1} \bigr) \, , \qquad  (a_j)_i \in \bigl( \mf h \otimes z^i \bigr)^{\dual} \, .
\end{equation*}
This yields singular modules $W^{(r_j)}_{\chi_j} \eqqcolon W_j \subseteq \wh W_j \ceqq \wh W^{(r_j)}_{\chi_j}$, and we consider the vector spaces
\begin{equation}
	\label{eq:tensor_product_singular_modules}
	\wh{\mc H} = \wh{\mc H}_{\bm p,\bm \chi} \ceqq \bigotimes_{j \in J} \wh W_j \, , \qquad \mc H = \mc H_{\bm p,\bm \chi} \ceqq \bigotimes_{j \in J} W_j \, ,
\end{equation}
where $\bm \chi = ( \chi_j )_{j \in J}$.
Clearly $\mc H \subseteq \wh{\mc H}$, and the dependence on the choice of marked points is void (it becomes relevant after considering the action of $\mf g$-valued meromorphic functions in \S~\ref{sec:action_meromorphic_functions}).

The spaces~\eqref{eq:tensor_product_singular_modules} are endowed with natural structures of left modules for the associative algebras  $U \bigl( \wh{\mf g} \bigr)^{\otimes n} \simeq \bigotimes_{j \in J} U \bigl( \wh{\mf g}_j \bigr)$ and $\bigotimes_{j \in J} U \bigl( \mf g_{r_j} \bigr)$, respectively.

Moreover for indices $i \neq j \in J$ denote by $\iota^{(ij)} \colon U(\mc L\mf g)^{\otimes 2} \to U(\mc L\mf g)^{\otimes n}$ the natural inclusion on the $i$-th and $j$-th slot, defined on pure tensors by
\begin{equation}
	\label{eq:embedding_quadratic_tensors}
	X \otimes Y \longmapsto 1^{\otimes i-1} \otimes X \otimes 1^{\otimes j-i-1} \otimes Y \otimes 1^{\otimes n-j} \, ,
\end{equation}
for $i < j$, and analogously for $i > j$.
Finally define $\iota^{(ii)} \colon U(\mc L\mf g)^{\otimes 2} \to U(\mc L\mf g)^{\otimes n}$ by $X \otimes Y \mapsto 1^{\otimes i-1} \otimes X Y \otimes 1^{\otimes n - i}$.
This yields an action of quadratic loop-algebra tensors on~\eqref{eq:tensor_product_singular_modules}.

\subsection{Tame isomonodromy times}

We now vary part of the parameters defining the spaces~\eqref{eq:tensor_product_singular_modules}, namely the marked points.
An \emph{admissible} deformation is one where they do not coalesce, so marked points vary inside the configuration space
\begin{equation*}
	C_n \ceqq \Conf_n(\Sigma) \subseteq \Sigma^n \, ,
\end{equation*}
of ordered $n$-tuples of (labeled) points on $\Sigma$.

The space $C_n$ is the space of \emph{tame} isomonodromy times.
It is a complex manifold of dimension $n$.

\begin{remark*}
	The terminology points again to meromorphic connections on the sphere, cf. introduction.

	Let us repeat that the positions of the poles, and the irregular types, are the natural deformation parameters for the \emph{isomonodromic} deformations of irregular-singular meromorphic $G$-connections over the sphere.
	This means keeping Stokes data locally constant,
	which in turn generalise the isomorphism class of the monodromy representation $\nu \colon \pi_1 \bigl( \Sigma^{\circ},b \bigr) \to G$, where $\Sigma^{\circ} \ceqq \Sigma \setminus \set{ p_j }_{j \in J}$ is the punctured sphere with the poles removed and $b \in \Sigma^{\circ}$ a base point (cf.~\cite{boalch_2002_g_bundles_isomonodromy_and_quantum_weyl_groups} and~\cite{boalch_2007_quasi_hamiltonian_geometry_of_meromorphic_connections} in this generic case).

	Concretely,
	isomonodromic deformations amount to a system of nonlinear differential equations where the positions of the poles and the irregular types are precisely the independent variables, hence they become the `times': the former are the tame/regular times, and the rest are wild/irregular.

	Geometrically, these differential equations constitute a \emph{nonlinear} flat/integrable symplectic connection in the local system of moduli spaces $\mc M^*_{\dR}$ of meromorphic connections, as the marked points and the irregular types vary (i.e. as the \emph{wild} Riemann surface structure on the sphere varies~\cite{boalch_2014_geometry_and_braiding_of_stokes_data_fission_and_wild_character_varieties}).
\end{remark*}

\begin{remark}
	\label{rem:local_chart_configuration_space}
	Let $\Sigma \supseteq U \xrightarrow{z}{} \mb C$ be a local affine chart on $\Sigma$---so $\Sigma \simeq \mb C \cup \Set{\infty}$.
	Then $\bm t \colon C_n(U) \to \mb C^n$ are coordinates on the open subset $C_n(U) \ceqq \Conf_n(U) \subseteq C_n$, where $\bm t = (t_j)_{j \in J}$ and $t_j (\bm p) \ceqq z(p_j)$---so that overall $C_n(U) \simeq \Conf_n (\mb C)$.
	This yields an atlas on the configuration space.
\end{remark}

Now for a $J$-tuple $\bm \chi$ of singular characters we consider the vector bundles $\wh{\bm{\mc H}} = \wh{\mc H}_{\bullet,\bm \chi} \to C_n$ and $\bm{\mc H} = \mc H_{\bullet,\bm \chi} \to C_n$, whose fibres over $\bm p \in C_n$ are the spaces~\eqref{eq:tensor_product_singular_modules}, respectively.
We have an inclusion $\bm{\mc H} \subseteq \wh{\bm{\mc H}}$, as well as global trivialisations:
\begin{equation}
	\wh{\bm{\mc H}} \simeq \bigotimes_{J \in J} U \bigl( \wh{n}^- \bigr) \otimes_{U (\mf n^-)} U \bigl( \mf n^-_{r_j} \bigr) \times C_n \longrightarrow C_n \, ,
\end{equation}
by~\eqref{eq:vector_space_isomorphism_singular_module}, and the simpler
\begin{equation}
	\bm{\mc H} \simeq \bigotimes_{j \in J} U \bigl( \mf n^-_{r_j} \bigr) \times C_n \longrightarrow C_n \, ,
\end{equation}
by $W_j \simeq U \bigl(\mf n^-_{r_j} \bigr)$.
The point here is that both vector-space isomorphisms do \emph{not} depend on the choice of the marked points (nor on the character, cf.~\ref{sec:irregular_isomonodromy_times}).

\subsection{Action of meromorphic functions: punctual version}
\label{sec:action_meromorphic_functions}

Given marked points $p_j \in \Sigma$ consider the effective divisor $D \ceqq \sum_{j \in J} [ p_j ]$ on $\Sigma$, and denote as customary by $\ms O_{\ast D} ( \Sigma ) = \ms O_{\Sigma,\ast D} (\Sigma)$ the vector space of meromorphic functions along $\Sigma$ with poles at most on (the support of) $D$.
Then let $\mf g_{\ast D} (\Sigma) \ceqq \mf g \otimes \ms O_{\ast D} (\Sigma)$ be the Lie algebra of $\mf g$-valued such meromorphic functions, with bracket coming from $\mf g$:
\begin{equation*}
	[f,g] (p) \ceqq \bigl[ f(p),g(p) \bigr] \in \mf g \, , \qquad \text{for } f,g \in \mf g_{\ast D} ( \Sigma ) \, , p \in \Sigma \, .
\end{equation*}

Taking Laurent expansions at $p_j$ yields a linear map $\tau_j \colon \ms O_{\ast D}(\Sigma) \to \wh{\ms K}_j$, and tensoring with $\mf g$ a linear map $\mf g_{\ast D}(\Sigma) \to \mc L\mf g_j \subseteq \wh{\mf g}_j$.

\begin{remark*}
	If $z_j$ is a local coordinate on $\Sigma$ vanishing at $p_j$, and $f \in \ms O_{\ast D}(\Sigma)$, then there are coefficients $f_i \in \mb C$ such that
	\begin{equation}
		\tau_j (f) = f(z_j) = \sum_{i \geq -\ord_{p_j}(f)} f_i z_j^i \in \mb C (\!( z_j ) \!) \, ,
	\end{equation}
	where $\ord_p(f) \geq 0$ is the order of $p \in \Sigma$ as a pole of $f$.
\end{remark*}

Thus there is an arrow
\begin{equation}
	\label{eq:action_meromorphic_functions}
	\tau \colon \mf g_{\ast D}(\Sigma) \longrightarrow \End \bigl( \wh{\mc H} \bigr) \, , \qquad \tau(X \otimes f) \ceqq \sum_{j \in J} \bigl( X \otimes \tau_j(f) \bigr)^{(j)} \, .
\end{equation}
Using~\eqref{eq:affine_lie_bracket}, and the fact that the sum of the residues of a meromorphic 1-form on $\Sigma$ vanishes, shows that~\eqref{eq:action_meromorphic_functions} is a morphism of Lie algebras (cf.~\cite{kohno_2002_conformal_field_theory_and_topology}).
Then the action $\tau \colon \mf g_{\ast D}(\Sigma) \to \mf{gl} \bigl( \wh{\mc H} \bigr)$ endows $\wh{\mc H}$ with a left $\mf g_{\ast D}(\Sigma)$-module structure.

\begin{definition}[Irregular covacua, first version]~\newline
	\label{def:irregular_conformal_blocks_space}
	The space of \emph{irregular covacua} at the pair $(\bm p,\bm \chi)$ is the space of coinvariants of the $\mf g_{\ast D}(\Sigma)$-module $\wh{\mc H}$:
	\begin{equation}
		\label{eq:irregular_conformal_blocks_space}
		\ms H \ceqq \wh{\mc H}_{\mf g_{\ast D}} = \wh{\mc H}_{\bm p,\bm \chi} \big\slash \mf g_{\ast D}(\Sigma) \wh{ \mc H}_{\bm p,\bm \chi} \, .
	\end{equation}
\end{definition}

Then the space of \emph{vacua} is the dual of~\eqref{eq:irregular_conformal_blocks_space}:
\begin{align}
	\label{eq:vacua}
	\ms H^{\dagger} & \ceqq \Hom_{\mf g_{\ast D}(\Sigma)} (\wh{\mc H}_{\bm p,\bm \chi},\mb C)                                                                                                                                                                \\
	                & = \Set{ \Bra{\psi} \in \wh{\mc H}^{\dual}_{\bm p,\bm \chi} | \Braket{\psi | \tau(X \otimes f) \wh{\bm w} } = 0 \text{ for } X \in \mf g \, , \, f \in \ms O_{\ast D}(\Sigma) \, , \, \wh{\bm w} \in \wh{\mc H}_{\bm p,\bm \chi} } \, .
\end{align}
Note that we change the usual notation, since in this paper we focus on the space of coinvariants of $\wh{\mc H}$---rather than the space of invariants of its dual (cf.~\cite{biswas_mukhopadhyay_wentworth_2021_geometrization_of_the_tuy_wzw_kz_connection}).

By~\eqref{eq:action_meromorphic_functions}, the fundamental identity inside the space of covacua is
\begin{equation}
	\label{eq:identity_coinvariants}
	\Bigl[ \bigl( X \otimes \tau_i(f) \bigr)^{(i)} \wh{\bm w} \Bigr] = - \sum_{j \in J \setminus \Set{i}} \Bigl[ \bigl( X \otimes \tau_j(f) \bigr)^{(j)} \wh{\bm w} \Bigr] \, ,
\end{equation}
for $i \in J$, where square brackets denote equivalence classes modulo $\mf g_{\ast D}(\Sigma) \wh{ \mc H}_{\bm p,\bm \chi}$.

\subsection{Action of meromorphic functions: global version}

Now we want to globalise the action~\eqref{eq:action_meromorphic_functions} over the space of configurations of $n$-tuples of points on the sphere, i.e. we want a map of sheaves of Lie algebras on $C_n$.

To define the domain sheaf consider the projection
\begin{equation}
	\pi_{\Sigma} \colon \Sigma^{n+1} \longrightarrow \Sigma^n, \qquad (p,p_1,\dc,p_n) \longmapsto (p_1,\dc,p_n) \, .
\end{equation}
Then set
\begin{equation}
	Y \ceqq \pi_{\Sigma}^{-1} (C_n) = \Set {(p,p_1,\dc,p_n) | p_i \neq p_j \text{ for } i \neq j } \subseteq \Sigma^{n+1} \, .
\end{equation}
Now for $j \in J$ define the hyperplane $P_j \ceqq \Set { p = p_j } \subseteq \Sigma^{n+1}$, consider the effective divisor $\mc D \ceqq \sum_{j \in J} \bigl[ Y \cap P_j \bigr]$ on $Y$, and let $\ms O_{\ast \mc D} = \ms O_{Y,\ast \mc D}$ be the sheaf of meromorphic functions on $Y$ with poles at most along (the support of) $\mc D$.
Then we consider the pushforward sheaf $(\pi_{\Sigma})_{\ast} \ms O_{\ast \mc D}$ on $C_n$, and tensoring yields the sheaf $\mf g_{\ast \mc D} \ceqq \mf g \otimes (\pi_{\Sigma})_{\ast} \ms O_{\ast \mc D}$---of Lie algebras.

\begin{remark*}
	If $U' \subseteq C_n$ is open then $\mf g_{\ast \mc D} (U')$ is the Lie algebra of $\mf g$-valued meromorphic functions on $\Sigma \times U'$, such that the restriction to $\Sigma \times \Set{ \bm p } \simeq \Sigma$ has poles at most at the set $\set { p_j }_{j \in J}$ for all $\bm p \in U'$, as wanted.
\end{remark*}

Now for an open subspace $U' \subseteq C_n$ consider the Laurent expansion $\tau_j(U') (f)$ of functions $f \in \mc O_{\ast \mc D} \bigl( \pi_{\Sigma}^{-1}(U') \bigr)$ along the divisor $Y \cap P_j$.
Tensoring with $\mf g$ yields a map of sheaves
\begin{equation}
	\tau_j \colon \mf g_{\ast \mc D} \longrightarrow \ms O_{C_n} \otimes \mc L\mf g_j \subseteq \ms O_{C_n} \otimes \wh{\mf g}_j \, ,
\end{equation}
where $\ms O_{C_n}$ is the structure sheaf on the configuration space.

\begin{remark*}
	Explicitly, if $z_j$ is a local coordinate on $\Sigma$ vanishing at $p_j$, and $U' = \Conf_n(U)$ where $U \subseteq \Sigma$ is an open affine subset, and $f \in (\pi_{\Sigma})_{\ast} \ms O_{\ast \mc D} (U')$, then there are suitable functions $f_i \colon U' \to \mb C$ such that
	\begin{equation}
		\tau_j(U') (f) = f(z_j,t_1,\dc,t_n) = \sum_i f_i(t_1,\dc,t_n) z_j^i \in \ms O_{C_n}(U') \otimes \mb C (\!( z_j ) \!) \, ,
	\end{equation}
	using the local coordinates $(t_j)_{j \in J}$ on $U' \subseteq C_n$ of Rmk.~\ref{rem:local_chart_configuration_space}.
	By definition the functions $f_i$ may have poles on the hyperplanes $\Set { t_i = t_j } \subseteq \mb C^n$.
\end{remark*}

Finally summing the action over each slot of the tensor product we have a sheaf-theoretic analogue of~\eqref{eq:action_meromorphic_functions}, acting on sections of $\wh{\bm{\mc H}}$.

\subsection{Irregular isomonodromy times}
\label{sec:irregular_isomonodromy_times}

One may add the other possible deformations, e.g. with the following setup.

Recall the \emph{regular} parts of the Cartan subalgebra and its dual are the complements of (co)root hyperplanes:
\begin{equation}
	\mf h_{\reg} \ceqq \mf h \setminus \bigcup_{\alpha \in \mc R} \Ker (\alpha) \, , \qquad \mf h_{\reg}^{\dual} \ceqq \mf h^{\dual} \setminus \bigcup_{\alpha \in \mc R} \Ker \bigl( \ev_{H_{\alpha}} \bigr) \, ,
\end{equation}
and analogously for $\mf h \otimes z^i$ and its dual.

Then consider irregular parts $q_j \in \mf b_{r_j}^{\dual}$ such that the most irregular coefficient $(a_j)_{r_j-1}$ is regular, and define an \emph{admissible} deformation of as one in which the most irregular coefficient does not cross coroot hyperplanes.

\begin{remark*}
	This is the analogous condition as for the marked points: the open charts $C_n(\mb C) \subseteq \mb C^n$ are regular parts for Cartan subalgebras in type $A$.
\end{remark*}

Doing so we get to the space of isomonodromy times
\begin{equation}
	\label{eq:isomonodromy_times_space}
	\bm{B} = C_n \times \prod_{j \in J} \bigl( \mf h_{r_j}^{\dual} \bigr)_{\reg} \, ,
\end{equation}
where
\begin{equation}
	\bigl( \mf h^{\dual}_{r_j} \bigr)_{\reg} = \prod_{i = 1}^{r_j-2} \bigl( \mf h \otimes z^i \bigr)^{\dual} \times \bigl( \mf h \otimes z^{r_j-1} \bigr)^{\dual}_{\reg} \, ,
\end{equation}
and in turn $(\mf h \otimes z^i)_{\reg} \ceqq \set{ Hz^i | H \in \mf h_{\reg} } \subseteq \mf h \otimes z^i$ for $i \in \mb Z$.

The space~\eqref{eq:isomonodromy_times_space} is a complex manifold of dimension $d = n + r \sum_{j \in J} (r_j - 1)$, where $r = \rk(\mf g)$.
As expected it coincides with the space of tame isomonodromy times if $r_j = 1$ for $j \in J$.\fn{Again, this has now been studied in much more detail and generality in~\cite{doucot_rembado_tamiozzo_2022_local_wild_mapping_class_groups_and_cabled_braids,doucot_rembado_2023_topology_of_irregular_isomonodromy_times_on_a_fixed_pointed_curve,boalch_doucot_rembado_2022_twisted_local_wild_mapping_class_groups_configuration_spaces_fission_trees_and_complex_braids,doucot_rembado_tamiozzo_2024_moduli_spaces_of_untwisted_wild_riemann_surfaces,doucout_rembado_yamakawa_twisted_g_local_wild_mapping_class_groups}.}

\begin{remark*}
	If there is just one irregular module $W_j$ with $r_j = 2$ then
	\begin{equation}
		\bm{B} = C_n \times \bigl( \mf h \otimes z \bigr)^{\dual}_{\reg} \, ,
	\end{equation}
	and one recovers the base space for the FMTV connection~\cite{felder_markov_tarasov_varchenko_2000_differential_equations_compatible_with_kz_equations}---up to the canonical vector-space isomorphism $\mf h \otimes z \simeq \mf h$.
	If further the variations of marked points are neglected then~\eqref{eq:isomonodromy_times_space} becomes the base space for the `Casimir' connection of De Concini/Millson--Toledano Laredo (DMT)~\cite{millson_toledanolaredo_2005_casimir_operators_and_monodromy_representations_of_generalised_braid_groups,toledanolaredo_2002_a_kohno_drinfeld_theorem_for_quantum_weyl_groups}.
\end{remark*}

Then in~\eqref{eq:tensor_product_singular_modules} one can let both $\bm p \in C_n$ and $\bm \chi \in \prod_{j \in J} \bigl( \mf h^{\dual}_{r_j} \bigr)_{\reg}$ vary, getting a vector bundle over the base space~\eqref{eq:isomonodromy_times_space}.
This also comes with a canonical trivialisation, reasoning in the same way as for $\bm{\mc H} \subseteq \wh{\bm{\mc H}}$ (namely~\eqref{eq:vector_space_isomorphism_singular_module} is also independent of $\chi$).

Finally one may extend the sheaf $\mf g_{\ast \mc D}$ trivially along the Cartan directions, taking the pullback sheaf along the canonical projection $\bm{B} \twoheadrightarrow C_n$.

\section{Irregular blocks in terms of finite modules}
\label{sec:coinvariants}

Throughout this section fix a pair $(\bm p,\bm \chi)$ to define the spaces $\mc H \subseteq \wh{\mc H}$ as in~\eqref{eq:tensor_product_singular_modules}.
Then compose the inclusion $\mc H \hookrightarrow \wh{\mc H}$ with the canonical projection $\pi_{\ms H} \colon \wh{\mc H} \twoheadrightarrow \ms H$ to obtain a map $\iota \colon \mc H \to \ms H$.

To study the image of $\iota$ consider the tensor-product filtration
\begin{equation}
	\label{eq:tensor_product_filtration_negative_z_degree}
	\wh{\bm{\mc F}}^-_{\leq \bullet} \ceqq \bigotimes_{j \in J} \bigl( \wh{\mc F}^-_j \bigr)_{\leq \bullet} \, ,
\end{equation}
where $\bigl( \wh{\mc F}^-_j \bigr)_{\leq \bullet}$ is the filtration defined in \S~\ref{sec:filtrations} on $\wh W_j$.
By definition $\wh{\bm{\mc F}}^-_{\leq 0} = \mc H$, and we push~\eqref{eq:tensor_product_filtration_negative_z_degree} forward to a filtration $\wh{\ms{F}}^-_{\leq \bullet}$ on $\ms H$, along the surjection $\pi_{\ms H}$.
(Note that $\wh{\ms{F}}^-_{\leq \bullet}$ is exhaustive, since $\wh{\pmb{\mc F}}^-_{\leq \bullet}$ is.)

\begin{proposition}
	\label{prop:surjectivity_coinvariants}
	The map $\iota$ is surjective,
\end{proposition}

\begin{proof}
	We will show that $\wh{\ms{F}}^-_{\leq k}$ lies in the image of $\iota$ by induction on $k \geq 0$.
	The base is given by $\wh{\ms{F}}^-_{\leq 0} = \pi_{\ms H} \bigl( \mc H \bigr)$.

	Now we use~\eqref{eq:identity_coinvariants} for a function $f_i \in \ms O_{\ast D}(\Sigma)$ with a pole at $p_i$, and only there.
	Such a function is e.g. defined by $f_i(z) = (z - t_i)^{-m}$, with the notations of Rmk.~\ref{rem:local_chart_configuration_space}, working in a local chart containing $\bm p$.

	Hence $\tau_j \bigl(f_i\bigr) \in \wh{\ms O}_j$ for $j \neq i$, and if $\wh{\bm w} \in \wh{\bm{\mc F}}^-_{\leq k}$ the rightmost identity of~\eqref{eq:negative_degree_filtration_shift} shows that the right-hand side of~\eqref{eq:identity_coinvariants} lies in $\wh{\ms{F}}^-_{\leq k}$.
	Then by induction the image of $\iota$ contains $\wh{\ms{F}}^-_{\leq k}$, as well as the vectors on the left-hand side of~\eqref{eq:identity_coinvariants}: the conclusion follows from the leftmost identity of~\eqref{eq:negative_degree_filtration_shift}.
\end{proof}

\begin{proposition}
	\label{prop:kernel_coinvariants}
	One has $\Ker (\iota ) = \mf g \mc H \subseteq \mc H$.
\end{proposition}

\begin{proof}
	To prove the nontrivial inclusion choose an element $\wh{\bm w} = \tau(X \otimes f) \wh{\bm{u}}$ with $\wh{\bm{u}} = \bigotimes_{j \in J} \wh{u}_j \in \wh{\mc H}$, $f \in \mf g_{\ast D}(\Sigma)$ and $X \in \mf g$.

	If the function $f$ is noncostant then it has a pole, say at $p_j \in \Sigma$.
	It follows that $\tau_j(f) \not\in \wh{\ms O}_j$, whence $X \otimes \tau_j (f)^{(j)} \wh{u}_j \not\in \bigl( \wh{\mc F}^+_j \bigr)_{\leq 0}$ by~\eqref{eq:negative_degree_grading_shift}, and $\wh{\bm w} \not\in \mc H = \wh{\bm{\mc F}}^+_{\leq 0}$.

	Thus to have element of the kernel we must restrict to $f \in \mb C$.
	Then using~\eqref{eq:negative_degree_grading_shift} again we see that $X \otimes f = X \otimes \tau_j(f) \in \mf g$ preserves the grading $\bigl( \wh{\mc F}^-_j \bigr)_{\bullet}$ on $\wh W_j$, so $(X \otimes f) \wh{u}_j \in W_j = \bigl( \wh{\mc F}^-_j \bigr)_0$ implies $\wh{u}_j \in W_j$.
\end{proof}

Hence there is an identification $\ms H \simeq \mc H _{\mf g} = \mc H \big\slash \mf g \mc H$, generalising the analogous standard fact for the tame case.

To go further one may appeal to the tensor product of the weight gradings of \S~\ref{sec:weight_gradings}, which is a $\bigl( \mf h^{\dual} \bigr)^{\! J}$-grading on $\mc H$.
Namely we consider the subspaces
\begin{equation}
	\label{eq:tensor_product_weight_grading}
	\mc F_{\bm{\mu}} = \mc F_{\bm{\mu}}(\mc H) \ceqq \bigotimes_{j \in J} \mc F_{\mu_j} (W_j) \subseteq \mc H \, , \qquad \text{for } \bm{\mu} = (\mu_j)_{j \in J} \in \bigl( \mf h^{\dual} \bigr)^{\! J} \, .
\end{equation}
By~\eqref{eq:action_meromorphic_functions} the subspace $\bm{\mc F}_{\bm{\mu}}$ lies inside the weight space of weight $\abs{\bm{\mu}} \ceqq \sum_j \mu_j \in \mf h^{\dual}$ for the tensor product $\mf h$-action.

If $\abs{ \bm{\mu}} \neq 0$ then $\bm{\mc F}_{\bm{\mu}} \subseteq \mf h \bm{\mc F}_{\bm{\mu}}$ is annihilated by $\pi_{\ms H}$, so there is still a surjective map
\begin{equation}
	\label{eq:surjection_conformal_block_weight_decomposition}
	\mc H \supseteq \bigoplus_{\abs{ \bm{\mu}} = 0} \bm{\mc F}_{\bm{\mu}} \xrightarrow{\pi_{\ms H}}{} \ms H \, ,
\end{equation}
and by construction the $\mf h$-action is trivialised on this subspace.

\begin{remark*}
	The condition $\abs{\bm{\mu}} = 0$ is also reminiscent of meromorphic connections: it is equivalent to the vanishing of the sum of the residues over $\Sigma$---in the duality~\eqref{eq:residue_pairing}.
\end{remark*}

\subsection{Auxiliary tame module}
\label{sec:auxiliary_tame_module}

Suppose one of the modules is tame, e.g. the last one: $r_n = 1$ and $W_n = V_n$.
Then we split the tensor product as
\begin{equation}
	\label{eq:tensor_product_without_infinity}
	\mc H = \mc H' \otimes V_n \, , \qquad \mc H' \ceqq \bigotimes_{j \in J'} W_j \, ,
\end{equation}
where $J' \ceqq J \setminus \Set{n}$, and we embed
\begin{equation}
	\mc H' \longrightarrow \mc H \, , \qquad \bigotimes_{j \in J'} \wh w_j \longmapsto \bigotimes_{j \in J'} \wh w_j \otimes v_n \, ,
\end{equation}
where $v_n \in V_n$ is the cyclic/highest-weight vector.

\begin{proposition}
	\label{prop:auxiliary_tame_module}
	One has $\iota({\mc H'}) = \ms H$.
\end{proposition}

\begin{proof}
	Denote by $\mc{E}^{(n)}_{\leq \bullet}$ the filtration on $U(\mf g) v_n \subseteq V_n$ defined in \S~\ref{sec:filtrations}, which is exhaustive in this (tame) case.
	We will prove by induction on $k \geq 0$ that $\iota \bigl( \mc H'\bigr)$ contains the classes of all vectors inside $\mc H' \otimes \mc{E}^{(n)}_{\leq k}$, noting that the base follows from the identity $\mc{E}^{(n)}_{\leq 0} = \mb C v_n$.

	For the inductive step we use the constant version of~\eqref{eq:identity_coinvariants}.
	For $X \in \mf g$ this shows that the class of $X^{(n)} \wh{\bm w}$ lies in $\iota \bigl( \mc H' \bigr)$ as soon as that of $\wh{\bm w} \in \mc H' \otimes \mc{E}^{(n)}_{\leq k}$ does, which is precisely the inductive hypothesis: the conclusion follows from~\eqref{eq:constant_filtration_shift}.
\end{proof}

Now $\mb C v_n = \mc F_{\lambda_n} (V_n)$, so~\eqref{eq:surjection_conformal_block_weight_decomposition} yields a surjection:
\begin{equation}
	\label{eq:surjection_weight_spaces}
	\mc H' \supseteq \bigoplus_{\abs{\bm{\mu}} = -\lambda_n} \bm{\mc F}'_{\bm{\mu}} \xrightarrow{\pi_{\ms H}} \ms H \, , \qquad \text{where } \bm{\mu} = (\mu_j)_{j \in J'} \in \bigl( \mf h^{\dual} \bigr)^{\! J'} \, ,
\end{equation}
writing $\abs{\bm{\mu}} = \sum_{j \in J'} \mu_j\in \mf h^{\dual}$ analogously to the above, and where $\bm{\mc F}'_{\bm{\mu}} \subseteq \mc H'$ is the tensor product of the weight-gradings over $J' \subseteq J$ (analogously to~\eqref{eq:tensor_product_weight_grading}).
Note the direct sum is just the weight space of weight $-\lambda_n \in \mf h^{\dual}$ for the (tensor) action of $\mf h$ on $\mc H'$; let us temporarily denote this space by $\mc H'(-\lambda_n)$.

\begin{lemma}
	\label{lem:final_identification_coinvariants}
	The kernel of~\eqref{eq:surjection_weight_spaces} equals $\mf n^+ \mc H' \cap \mc H'(-\lambda_n) \subseteq \mc H'(-\lambda_n)$.
\end{lemma}

\begin{proof}
	We must show that no coinvariants can arise from the residual $\mf n^-$-action.

	To this end recall that $\mf n^-$ has nontrivial action on the associated graded of the filtration $\mc{E}_{\leq \bullet}$ of \S~\ref{sec:filtrations}: more precisely~\eqref{eq:constant_grading_shift} yields
	\begin{equation*}
		\mf n^- \Bigl( \mc H' \otimes \gr \bigl(\mc{E}^{(n)}\bigr)_k \Bigr) \subseteq \Bigl( \mc H' \otimes \gr \bigl(\mc{E}^{(n)}\bigr)_k \Bigr) \oplus \Bigl( \mc H' \otimes \gr \bigl(\mc{E}^{(n)}\bigr)_{k+1} \Bigr) \subseteq \mc H \, ,
	\end{equation*}
	for $k \in \mb Z_{\geq 0}$; but there can be no vanishing of components in the latter direct summand since $V_n$ is freely generated over $U(\mf n^-)$, and this applies in particular to $v_n \in \mc{E}^{(n)}_{\leq 0} \simeq \gr \bigl(\mc{E}^{(n)}_0 \bigr)$.
\end{proof}

The punchline is the final identification
\begin{equation}
	\label{eq:final_identification_coinvariants}
	\ms H \simeq \mc H'(-\lambda_n) \big\slash \bigl(\mf n^+ \mc H' \cap \mc H'(-\lambda_n) \bigr) \, .
\end{equation}

\subsubsection{On dimensions}
\label{sec:dimension_coinvariants}

To go further we use the results of \S~\ref{sec:weight_gradings}; in particular we employ the notation $\mc F_{\nu} (W_j) \ceqq \mc F_{\lambda_j - \nu} \subseteq W_j$ for $\nu \in Q^+$---i.e. we parametrise the weights $\mu_j = \lambda_j - \nu \preceq \lambda_j$ by $\nu \in Q^+$.

By definition the weight space of weight $-\lambda_n \in \mf h^{\dual}$ for the $\mf h$-action on $\mc H'$, denoted by $\mc H'(-\lambda_n)$ above, is the direct sum of the spaces $\bm{\mc F}'_{\bm{\nu}} \subseteq \mc H'$ such that $0 = \lambda_n + \sum_{j \in J'} (\lambda_j - \nu_j)$.
Thus, only elements such that $\abs{\bm{\nu}} = \abs{\bm \lambda} \in \mf h^{\dual}$ will contribute to coinvariants.
This actually depends on the sum of the tame parts of the singular characters, and so we now rather employ the following notation:
\begin{equation}
	\label{eq:zero_weight_space_tensor_product}
	\mc H'_{\abs{\bm \lambda}} \ceqq \bigoplus_{\abs{\bm{\nu}} = \, \abs{\bm \lambda}} \bm{\mc F}'_{\bm{\nu}} \subseteq \mc H' \, .
\end{equation}

\begin{proposition}
	\label{prop:dimension_weight_space_tensor_product}
	The $\mf h$-weight space $\mc H'_{\abs{\bm \lambda}} \subseteq \mc H'$ has dimension
	\begin{equation}
		\label{eq:dimension_weight_space_tensor_product}
		\dim \bigl(\mc H'_{\abs{\bm \lambda}} \bigr) = \sum_{ \abs{\bm{\nu}} = \, \abs{\bm \lambda}} \Biggl( \, \prod_{j \in J'} \sum_{\bm m \in \Mult_{\mc R^+}(\nu_j)} \binom{\bm m+r_j-1}{\bm m} \Biggr) < \infty \, .
	\end{equation}
\end{proposition}

\begin{proof}
	This follows from~\eqref{eq:dimension_weight_space}, taking the products of the dimensions of the weight spaces $\mc F_{\nu_j} \subseteq W_j$.

	The dimension is finite since for $\nu \in Q_+$ there are finitely many $J'$-tuples $\bm{\nu} \in \bigl( Q^+ \bigr)^{J'}$ such that $\abs{\bm{\nu}} = \nu$---analogously to $\abs{\Mult_{\mc R^+} (\nu)} < \infty$.
\end{proof}

We deduce the following.

\begin{corollary}
	\label{cor:finite_dimensional_irregulal_conformal_blocks}
	If one module is tame then the space~\eqref{eq:irregular_conformal_blocks_space} is finite-dimensional for all choices of marked points and singular characters.
\end{corollary}

In particular the weight space is trivial if $\abs{\bm \lambda} \not\in Q^+$, and the simplest nontrivial case is when $\abs{\bm \lambda} = 0$.
Then $\abs{\bm{\nu}} = 0$ implies $\nu_j = 0$ for $j \in J'$, so $\mc H'_0$ is the line generated by the tensor product $\bigotimes_{i \in J'} w_i$ of the cyclic vectors $w_i \in W_i$.

The next nontrivial example is when $\abs{\bm \lambda} = \theta \in \Pi$ is a simple root.
Now $\abs{\bm{\nu}} = \theta$ implies $\bm{\nu} \in \Set{ \bm{\nu}^{\theta,i} }_i$, with $\nu^{\theta,i}_j = \delta_{ij} \theta$ for $i,j \in J'$.
Then one finds the singleton $\Mult_{\mc R^+} \bigl( \nu^{\theta,i}_j \bigr) = \Set{ \delta_{ij} \bm m^{\theta}}$, with $m^{\theta}_{\alpha} = \delta_{\alpha\theta}$.
On the whole~\eqref{eq:dimension_weight_space_tensor_product} reduces to
\begin{equation}
	\dim \bigl( \mc H'_{\theta} \bigr) = \sum_{i \in J'} \Biggl( \,  \prod_{j \in j'} \binom{\delta_{ij} \bm m^{\theta} + r_j - 1}{\delta_{ij}\bm m^{\theta}} \Biggr) = \sum_{i \in J'} \binom{\bm m^{\theta}+r_i-1}{\bm m^{\theta}} = \sum_{i \in J'} r_i \, ,
\end{equation}
independently of the choice of simple root.

A basis is given by the pure tensors
\begin{equation}
	\wh{\bm w}^j_i \ceqq \bigotimes_{k = 1}^{i-1} w_k \otimes F_{\theta} z^j w_i \otimes \bigotimes_{k = i+1}^{n-1} w_k \, ,
\end{equation}
for $i \in J'$ and $j \in \Set{ 0,\dc,r_i-1}$.

\begin{remark}
	\label{rem:nontriviality}
	One way to ensure coinvariants are nontrivial is the following: for a given configurations of points $\bm p = (p_j)_{j \in J}$ consider the Lie subalgebra of $\mf g$-valued meromorphic functions with poles at $p_j$, and further with a \emph{zero} elsewhere, say at $p' \in \Sigma \setminus \set{ p_j}_j$.
	Then the proof of Prop.~\ref{prop:surjectivity_coinvariants} can easily be adapted working in the chart where $p' = \infty$---as the function $f_i(z) = (z - t_i)^{-m}$ vanishes at infinity.

	Thus there is still a surjection of $\mc H$ on the space of coinvariants, and similarly to Prop.~\ref{prop:kernel_coinvariants} only constant functions lie in the kernel.
	Hence in this setup the kernel is trivial and $\mc H \neq (0)$ itself is the space of coinvariants.

	Another way to ensure nontriviality is to put a $\theta$-dual module in the tensor product (introduced in \S~\ref{sec:dual_modules}).
	Further when it is tame then one still has a finite-dimensional space, see \S~\ref{sec:irregular_conformal_blocks_space_with_duals}.
\end{remark}

\subsubsection{Archetypal case}

Consider the same setup of \S~\ref{sec:archetypal_case} for $\mf g = \mf{sl}(2,\mb C)$.
In this case $\abs{\bm \lambda} = m \alpha$ for an integer $m \geq 0$.

\begin{proposition}
	One has
	\begin{equation}
		\label{eq:dimension_weight_space_tensor_product_sl_2}
		\dim \bigl( \mc H'_{m \alpha} \bigr) = \binom{m + R - 1}{m} \, , \qquad \text{where } R \ceqq \sum_{j \in J'} r_j \, .
	\end{equation}
	A basis is provided by the pure tensors
	\begin{equation}
		\label{eq:basis_weight_space_sl_2_tensor_product}
		\wh{\bm w}^{\Phi} = \bigotimes_{j \in J'} \Biggl( \, \prod_{i = 0}^{r_j-1} \bigl( F z^i \bigr)^{\Phi(i,j)} w_j \Biggr) \, ,
	\end{equation}
	where $\Phi \in \WComp_{R} (m)$---identifying $\Set{1,\dc,R} \simeq \coprod_{j \in J'} \Set{ 0,\dc, r_j-1}$.
\end{proposition}

\begin{proof}
	Fix an integer $m \geq 0$ and look for $\bm{\nu} \in (\mb Z_{\geq 0} \alpha)^{J'}$ satisfying $\abs{\bm{\nu}} = m \alpha$.
	Such elements are given by weak $J'$-compositions of $m$, i.e. functions $\phi \colon J' \to \mb Z_{\geq 0}$ satisfying $\sum_{j \in J'} \phi(j) = m$, with bijection
	\begin{equation}
		\phi \longmapsto \bm{\nu}^{\phi} \, , \qquad \nu^{\phi}_j \ceqq \phi(j) \alpha \, .
	\end{equation}
	Then by definition $\Mult_{\mc R^+} \bigl( \nu^{\phi}_j \bigr) = \set{ \phi(j) }$ for $j \in J'$, so we need only give elements $\varphi_j \in \WComp_{r_j} \bigl( \phi(j) \bigr)$ to allocate the $z$-degrees of the occurrences of $-\alpha$ at each slot of the tensor product.

	The data of $\phi$ and $\bm \varphi = (\varphi_j)_j$ is equivalent to that of the weak $R$-composition $\Phi \colon R \to \mb Z_{\geq 0}$ defined by $\Phi(i,j) = \varphi_j(i)$, and the result follows.
\end{proof}

\begin{remark*}
	In the tame case~\eqref{eq:dimension_weight_space_tensor_product_sl_2} simplifies to
	\begin{equation}
		\dim \bigl( \mc H'_{m\alpha} \bigr) = \abs{\WComp_{J'} (m)} = \binom{m + \abs{J'} - 1}{m} \, ,
	\end{equation}
	and a basis is given by the pure tensors
	\begin{equation}
		\wh{\bm v}^{\phi} = \bigotimes_{j \in J'} F^{\phi(j)} v_j \, \qquad \text{for } \phi \in \WComp_{J'} (m) \, .
	\end{equation}
	(This is somehow the opposite of~\eqref{eq:basis_weight_space_sl_2}: there we had an arbitrary singular module, here we have a tensor product of arbitrarily many tame modules.)
\end{remark*}

\section{Irregular blocks: second version}
\label{sec:irregular_conformal_blocks_space_with_duals}

We now vary the setup of \S~\ref{sec:irregular_conformal_blocks}, giving a special role to one of the marked points (e.g. the last one): choose then a $(\cdot \mid \cdot)$-orthogonal morphism $\theta \colon \mf g \to \mf g^{\op}$, and put at the last marked point a $\theta$-dual module $\wh W'_{\theta} \twoheadrightarrow W'_{\theta}$---as in \S~\ref{sec:dual_modules}.

In this case the tensor product splits as
\begin{equation}
	\label{eq:tensor_product_singular_modules_with_dual}
	\wh{\mc H} = \wh W'_n \otimes \bigotimes_{j \in J'} \wh W_j \, , \qquad \mc H = W'_n \otimes \bigotimes_{j \in J'} W_j \, ,
\end{equation}
where $J' = J \setminus \Set{n}$ as in \S~\ref{sec:auxiliary_tame_module}---and omitting the subscript $\theta$.
These are naturally subspaces of $\Hom( \wh W_n,\wh{\mc H}' )$ and $\Hom( W_n,\mc H')$, respectively, where $\mc H'$ is as in \S~\ref{sec:auxiliary_tame_module} and $\wh{\mc H}' \ceqq \bigotimes_{j \in J'} \wh W_j$.
Moreover they still assemble into trivial vector bundles $\wh{\bm{\mc H}} \twoheadrightarrow \bm{\mc H}$ over the space $C_n = \Conf_n(\Sigma)$---but also over the full space~\eqref{eq:isomonodromy_times_space} of isomonodromy times.

The Lie algebra of $\mf g$-valued meromorphic functions on $\Sigma$ acts on the leftmost tensor product of~\eqref{eq:tensor_product_singular_modules_with_dual}.
Namely, for any linear map $\wh{\bm{\psi}} \colon \wh W_n \to \wh{\mc H}'$, using~\eqref{eq:action_meromorphic_functions} and the dual actions of \S~\ref{sec:dual_modules} yields
\begin{equation}
	\Braket{ \tau(X \otimes f ) \wh{\bm{\psi}} | \wh w } = \sum_{j \in J'} \bigl( X \otimes \tau_j(f) \bigr)^{(j)} \Braket{ \wh{\bm{\psi}} | (\theta(X) \otimes \tau_n(f) ) \wh w } \in \wh{\mc H}' \, ,
\end{equation}
where $X \in \mf g$, $f \in \ms O_{\ast D}(\Sigma)$ and $\wh w \in \wh W_n$.
Taking coinvariants of the resulting left module yields a second version of a space of irregular covacua, still denoted by $\ms H$: then again the space of vacua $\ms H^{\dagger}$ is defined as in~\eqref{eq:vacua}.
Moreover the material of \S~\ref{sec:irregular_conformal_blocks} goes through, and there is an action of the sheaf of Lie algebras $\mf g_{\ast \mc D}$ on sections of $\wh{\bm{\mc H}}$ and $\bm{\mc H}$.

\subsection{On coinvariants}

Consider first the natural inclusion $\iota \colon \wh W'_n \otimes \mc H' \hookrightarrow \wh{\mc H}$, which can be composed with the canonical projection $\pi_{\ms H} \colon \wh{\mc H} \to \ms H$.

Reasoning as in Prop.~\ref{prop:surjectivity_coinvariants} (which may be thought of as the case $\wh W'_n = \mb C$) shows that this composition is surjective.
Then reasoning as in Prop.~\ref{prop:kernel_coinvariants} shows that the kernel is obtained from the action of meromorphic functions with no poles at $\Set{ p_1,\dc,p_{n-1} } \subseteq \Sigma$, but only (at most) at the point $p_n$.
Hence there is a $\mb C$-linear isomorphism
\begin{equation}
	\ms H \simeq \wh W'_n \otimes \mc H' \big\slash \mf g_{\ast p_n} (\Sigma) \bigl( \wh W'_n \otimes \mc H' \bigr) \, ,
\end{equation}
thinking of $p_n \in \Sigma$ as a divisor.

Now a function with a pole at most at $p_n$ is either constant, or its Laurent expansion at $p_n$ lies in $z_n^{-1} \mf g \bigl[ z_n^{-1} \bigr] \subseteq (\mc L\mf g)_n$, where as usual $z_n$ is a local coordinate on $\Sigma$ vanishing at $p_n$.
Hence a coinvariant function is uniquely determined by its restriction to $W_n \subseteq \wh W_n$, and since now poles are not allowed we get the following:

\begin{proposition}
	\label{prop:coinvariants_with_dual}
	There is a canonical isomorphism of vector spaces
	\begin{equation}
		\ms H \simeq W'_n \otimes \mc H' \bigl\slash \mf g ( W'_n \otimes \mc H' ) \, .
	\end{equation}
\end{proposition}

Thus also in this case we can reduce the discussion to $\mf g$-coinvariants for the tensor product of finite modules.

Now suppose the dual module is tame, and adapt the discussion of \S~\ref{sec:auxiliary_tame_module}.
There is a surjective map $\mc H' \to \ms H$, where again $\mc H' = \bigotimes_{j \in J'} W_j$---embedded in $\mc H$ via $\bm{\wh w} \mapsto \psi \otimes \bm{\wh w}$, where $\psi \in V'_n$ is the cyclic vector.
Reasoning as in Lem.~\ref{lem:final_identification_coinvariants} the $\theta^{-1}(\mf n^-)$-action cannot give coinvariant elements, so we are left with the action of $\mf h \oplus \theta^{-1}(\mf n^+)$.

In the dual case where $\theta = \theta_0 = -\on{Id}_{\mf g}$ we have $\theta^{-1}(\mf n^+) = \mf n^+$, so we are essentially back to \S~\ref{sec:dimension_coinvariants}.
The contragredient case where $\theta = \theta_1$ (the transposition) instead allows to go further.
In this case $\theta^{-1}(\mf n^+) = \mf n^-$, whence a new identification $\ms H \simeq \mc H'_{\mf b^-}$, and to trivialise the $\mf h$-action we consider once more the zero-weight subspace inside $\mc H'$.
This is again~\eqref{eq:zero_weight_space_tensor_product}, whose (finite) dimension is given in Prop.~\ref{prop:dimension_weight_space_tensor_product}.

Finally in this setup we can recover nontriviality, as follows.
Recall that we attach weights $\bm \lambda = (\lambda_j)_{j \in J} \in \bigl( \mf h^{\dual} \bigr)^{J}$ to the marked points, and that we consider the sum $\abs{\bm \lambda} = \sum_{j \in J} \lambda_j \in \mf h^{\dual}$.
The weight space is $\mc H'_{\abs{\bm \lambda}} \subseteq \mc H'$, hence
\begin{equation*}
	\ms H \simeq \mc H'_{\abs{\bm \lambda}} \big\slash \bigl( \mf n^- \mc H' \cap \mc H'_{\abs{\bm \lambda}} \bigr) \, .
\end{equation*}
(Compare with~\eqref{eq:final_identification_coinvariants}: as expected the roles of the nilpotent subalgebras $\mf n^{\pm}$ are exchanged---by $\theta$.)

\begin{proposition}
	\label{prop:nontrial_irregular_conformal_blocks}
	Suppose $n \geq 3$ and $\abs{\bm \lambda} \in Q^+$: then the space of coinvariants is nontrivial---for any choice of wild parts.
\end{proposition}

A fortiori then nontriviality holds if the $n$-th module is not tame.

\begin{proof}
	For $\wh w \in \mc F_{\abs{\lambda}}(W_1) \subseteq W_1$ consider the pure tensor
	\begin{equation*}
		\wh{\bm w} \ceqq \wh w \otimes \bigotimes_{2 \leq i \leq n-1} w_i \in \mc H'_{\abs{\bm \lambda}} \, .
	\end{equation*}
	(In brief, put the cyclic vector in all slots except the first, and put a vector of suitable weight in the first slot.)
	An argument analogous to the proof of Lem.~\ref{lem:final_identification_coinvariants} shows that $\bm{\wh w}_i \not\in \mf n^-\mc H'$, which means exactly that $\bigl[ \wh{\bm w}_i \bigr] \neq 0$ inside $\ms H$.

	More precisely denote by $\mc{E}^{(j)}_{\leq \bullet}$ the filtration on $U(\mf g) w_j = U(\mf n^-) w_j \subseteq W_j$ induced from $U(\mf n^-)$, as in \S~\ref{sec:filtrations}, with associated grading $\gr \bigl( \mc{E}^{(j)} \bigr)_{\bullet}$.
	Then consider the tensor-product grading
	\begin{equation*}
		\gr \bigl( \bm{\mc{E}} \bigr)_{\bullet} \ceqq \bigotimes_{j \in J''} \gr \bigl( \mc{E}^{(j)} \bigr)_{\bullet} \, , \qquad \text{where } J'' \ceqq J' \setminus \set{1} \, .
	\end{equation*}
	Using~\eqref{eq:constant_grading_shift} yields
	\begin{equation*}
		\mf n^- \Bigl( W_1 \otimes \gr \bigl( \bm{\mc{E}} \bigr)_{\bm k} \Bigr) \subseteq \Bigl( W_1 \otimes \gr \bigl( \bm{\mc{E}} \bigr)_{\bm k} \Bigr) \oplus \bigoplus_{i = 2}^{n-1} \Bigl( W_1 \otimes \gr \bigl( \bm{\mc{E}} \bigr)_{\bm k + \varepsilon_i} \Bigr) \, ,
	\end{equation*}
	for $\bm k \in \mb Z_{\geq 0}^{J''}$, where $\varepsilon_i \in \mathbb{Z}^{J''}$ is the $i$-th vector of the canonical $\mb Z$-basis.
	Again the vanishing of components in the latter direct summands cannot happen, since the $U(\mf n^-)$-action is free on singular modules.
\end{proof}

\begin{remark*}
	If $n = 2$ instead simply $\ms H \simeq \mc F_{\nu}(W) \big\slash \bigl( \mf n^- W \cap \mc F_{\nu} (W) \bigr)$ for $\nu \in Q^+$, and we must further distinguish the tame/wild case.

	In the tame case $V = \mf n^- V \oplus \mb C v$, so nontriviality implies $v \in \mc F_{\nu}$: this forces $\nu = 0$ and $\ms H \simeq \mc F_{\lambda}(V) = \mb C v$.

	In the wild case instead write $\nu = \sum_{\alpha \in \mc R^+} m_{\alpha} \alpha$ for $m_{\alpha} \in \mb Z_{\geq 0}$, and consider the vector
	\begin{equation*}
		\wh w_{\nu} \ceqq \prod_{\alpha \in \mc R^+} \bigl( X_{-\alpha} z \bigr)^{m_{\alpha}} w \in \mc F_{\nu}(W) \, ,
	\end{equation*}
	ordering again the positive roots along the Cartan--Weyl basis~\eqref{eq:cartan_weyl_basis} (note this makes sense at all depths $p \geq 2$).
	Clearly $\wh w_{\nu} \not\in \mf n^- W$, since all occrrences of root vectors have positive $z$-degree, hence $\ms H \neq (0)$ always in this case.
\end{remark*}

\begin{remark*}
	One may also consider the tensor products of the grading of Def.~\ref{def:z_grading_finite_singular_module}, in addition to the $\mf h$-weight grading---i.e. use the fact that every finite module is a graded $\mf n^- \llbracket z \rrbracket$-module.
	Namely there is a decomposition
	\begin{equation}
		\mc H' = \bigoplus_{\bm k \in \mb Z^{J'}} \bm{\mc F}^+_{\bm k} \, , \qquad \text{where} \qquad \bm{\mc F}^+_{\bm k} = \bigotimes_{j \in J'} \mc F^+_{k_j} (W_j) \, ,
	\end{equation}
	which is preserved by the tensor product $\mf b^-$-action, so $\ms H \simeq \bigoplus_{k \in \mb Z^{J'}} (\bm{\mc F}^+_{\bm k})_{\mf b^-}$.
	This is a new feature: in the tame case the grading in positive $z$-degree is trivial.
\end{remark*}

\section{Flat connections}
\label{sec:connection_marked_points}

Consider a particular case of the setup of \S~\ref{sec:irregular_conformal_blocks}: mark $n+1$ (ordered) points on $\Sigma$, vary the first $n \geq 1$ of them, and \emph{fix} singular characters at those points.

Thus we work on a closed subspace of $\Conf_{n+1} (\Sigma)$, which is naturally identified with the local chart $U' = \Conf_n(U) \subseteq C_n$ of Rmk.~\ref{rem:local_chart_configuration_space} where $p_{n+1} = \infty$---whence $\Set{ p_1,\dc,p_n } \subseteq U \simeq \mb C$.
The label set becomes $J = \Set{ 1,\dc,n,\infty }$, and we write $J' \ceqq J \setminus \Set{\infty}$.

Then we have two versions of spaces of covacua: either we put a singular module at infinity, or a $\theta$-dual.
In any case we consider the restrictions of the vector bundles $\bm{\mc H} \subseteq \wh{\bm{\mc H}}$ over $U' \simeq C_n(\mb C) \ceqq \Conf_n(\mb C)$, as well as for the sheaves $(\pi_{\Sigma})_{\ast} \mc O_{\ast \mc D}$ and $\mf g_{\ast \mc D}$ on $U'$---and keep the same notation for them.

Then we want to define a connection $\wh{\nabla}$ on $\wh{\bm{\mc H}} \to C_n(\mb C)$ which is compatible with the action of the sheaf of Lie algebras $\mf g_{\ast \mc D}$.
In the given trivialisation this will be of the form $\wh{\nabla} = \dif - \wh{\varpi}$, where $\wh{\varpi}$ is a 1-form on $C_n(\mb C)$ with values in endomorphisms of the fibres, and with a view towards (generalisations of) KZ~\cite{knizhnik_zamolodchikov_1984_current_algebra_and_wess_zumino_model_in_two_dimensions} we set
\begin{equation}
	\Braket{ \wh{\varpi} | \partial_{t_i} } \ceqq L_{-1}^{(i)} \, , \qquad \text{for } i \in J' \, ,
\end{equation}
where we use the coordinates $\bm t \colon C_n(\mb C) \to \mb C^n$ of Rmk.~\ref{rem:local_chart_configuration_space} and the Sugawara operator~\eqref{eq:sugawara_operator}.
This is a translation-invariant 1-form on the parallelisable manifold $C_n(\mb C)$, so in particular $\dif \wh{\varpi} = 0$.
Further the actions of $L_{-1}$ on different slots commute, so $\bigl[ \wh{\varpi} \wedge \wh{\varpi} \bigr] = 0$, and the connection $\wh{\nabla}$ is (strongly) flat.

\subsection{Compatibility with the action of meromorphic functions}
\label{sec:compatibility_connections}

We now consider a natural connection $D$ on the sheaf $\mf g_{\ast \mc D}$---a linear map $D \colon \mf g_{\ast \mc D} \to \Omega^1_{C_n(\mb C)} \otimes \mf g_{\ast \mc D}$ satisfying Leibniz's rule.
Namely we set
\begin{equation}
	D (X \otimes f) \ceqq X \otimes \dif f \, ,
\end{equation}
where $\dif \colon \Omega^0_{C_n(\mb C)} \to \Omega^1_{C_n(\mb C)}$ is the standard de Rham differential.

\begin{proposition}
	\label{prop:compatibility_variations_marked_points}

	One has
	\begin{equation}
		\label{eq:compatibility_connections}
		\wh{\nabla} \bigl( \tau(X \otimes f) \wh{\bm w} \bigr) = \tau \bigl( D(X \otimes f) \bigr) \wh{\bm w} + \tau(X \otimes f) \wh{\nabla} \wh{\bm w} \, ,
	\end{equation}
	where $X \in \mf g$, and $f$ and $\wh{\bm w}$ are local sections of $(\pi_{\Sigma})_{\ast} \ms O_{\ast \mc D}$ and $\wh{\bm{\mc H}}$, respectively.
\end{proposition}

To prove this we use the following well-known fact.

\begin{lemma}[\cite{kac_1990_infinite_dimensional_lie_algebras}, Lem.~12.8]
	\label{lem:commutation_sugawara_meromorphic_functions}
	One has the identity $\bigl[ L_{-1}, Xz^m  \bigr] = -m X z^{m-1}$ inside the level-$\kappa$ completion of $U \bigl(\wh{\mf g} \bigr)$, for $X \in \mf g$ and $m \in \mb Z$.
\end{lemma}

\begin{proof}[Proof of Proposition~\ref{prop:compatibility_variations_marked_points}]
	For $i \in J'$, and for local sections $\wh{\bm w}$ and $X \otimes f$ of $\wh{\mc H}$ and $\mf g_{\ast \mc D}$---respectively---, we must prove that
	\begin{equation}
		\partial_{t_i}\bigl( \tau(X \otimes f) \wh{\bm w} \bigr) - \bigl[ L_{-1}^{(i)},\tau(X \otimes f) \bigr] \wh{\bm w} = \tau\bigl(X \otimes \partial_{t_i} f \bigr) \wh{\bm w} + \tau(X \otimes f) \partial_{t_i} \wh{\bm w} \, .
	\end{equation}

	Now for $j \in J'$ we have the expansions
	\begin{equation}
		\tau_j(f) = \sum_k f_k (t_1,\dc, t_n) z_j^k \, ,
	\end{equation}
	where $f_k$ is a regular function on an open subset of $C_n(\mb C)$, and we take the local coordinate $z_j = z - t_j$ on $\Sigma$---vanishing at $p_j$.
	Since $\partial_{t_i}(z_j) + \delta_{ij} = 0$ one has
	\begin{equation}
		\partial_{t_i} \bigl( \tau_j(f) \bigr)
		= \tau_j(\partial_{t_i} f) + \delta_{ij} \bigl[ L_{-1},\tau_j(f) \bigr] \, ,
	\end{equation}
	using Lem.~\ref{lem:commutation_sugawara_meromorphic_functions}.
	Hence by~\eqref{eq:action_meromorphic_functions}:
	\begin{equation}
		\partial_{t_i} \bigl( \tau(X \otimes f) \wh{\bm w} \bigr)
		= \tau \bigl(X \otimes \partial_{t_i} f \bigr) \wh{\bm w} + \bigl[ L_{-1},X \otimes \tau_i(f) \bigr]^{(i)} \wh{\bm w} + \tau(X \otimes f) (\partial_{t_i} \wh{\bm w}) \, ,
	\end{equation}
	and we conclude with
	\begin{equation*}
		\bigl[ L_{-1}^{(i)}, \tau(X \otimes f) \bigr] =
		\bigl[ L_{-1}^{(i)}, \bigl( X \otimes \tau_i (f) \bigr)^{(i)} \bigr] = \bigl[ L_{-1},X \otimes \tau_i (f) \bigr]^{(i)} \, . \qedhere
	\end{equation*}
\end{proof}

Thus a reduced connection is well defined on $\ms H \to C_n(\mb C)$, since $\wh{\nabla}$ preserves the sheaf of sections with values in the subspaces $\mf g_{\ast D} \wh{\mc H}_{\bm p,\bm \chi} \subseteq \wh{\mc H}_{\bm p,\bm \chi}$, by~\eqref{eq:compatibility_connections}.
We conclude that the sheaf of covacua has a natural structure of flat vector bundle over the space of tame isomonodromy times, so in particular the dimension of~\eqref{eq:irregular_conformal_blocks_space} is constant along variations of the marked points---when finite.
After dualising the connection, the same statements follow for the spaces of vacua.

\subsection{Description on finite modules: first version}
\label{sec:description_finite_modules}

By the results of \S~\ref{sec:coinvariants} it is possible to describe the reduction of $\wh{\nabla}$ as the $\mf g$-reduction of a connection $\nabla$ living on the vector sub-bundle $\bm{\mc H} \subseteq \wh{\bm{\mc H}}$, and further as a connections acting on $\bm{\mc H}' \subseteq \bm{\mc H}$ when the module at infinity is tame.

The goal is to find an explicit expression for $\nabla$.
For this we will use the following elementary fact, where we further set $z_{\infty} \ceqq z^{-1}$---a local coordinate vanishing at infinity.

\begin{lemma}[Expansions at irregular singularities]~\newline
	\label{lem:pole_expansion}
	For $i \in J'$ and for an integer $m > 0$ one has
	\begin{equation}
		\tau_j \bigl( z_i^{-m} \bigr) =
		\begin{cases}
			\label{eq:pole_expansion_finite}
			\sum_{l \geq 0} \binom{m+l-1}{l} \frac{z_j^l}{(t_i - t_j)^l(t_j - t_i)^m} \, , & j \in J \setminus \Set{i} \, , \\
			\sum_{l \geq 0} \binom{m+l-1}{l} t_i^l z_{\infty}^{m+l} \, ,                   & j = \infty \, .
		\end{cases}
	\end{equation}
\end{lemma}

\subsubsection{Tame case}
\label{sec:reduced_connection_tame}

Suppose $r_j = 1$ for $j \in J$. Then using~\eqref{eq:pole_expansion_finite} with $m = 1$ yields
\begin{equation}
	X \otimes \tau_j \bigl( z_i^{-1} \bigr) \wh v_j = \frac{ X }{t_j - t_i} \wh v_j \, , \qquad X \otimes \tau_{\infty} \bigl( z_i^{-1} \bigr) \wh v_{\infty} = 0 \, ,
\end{equation}
for $X \in \mf g$, $i \neq j \in J'$, $\wh v_j \in V_j$ and $\wh v_{\infty} \in V_{\infty}$---since $ z_j \mf g \llbracket z_j \rrbracket V_j = (0)$ for $j \in J$.
Hence by~\eqref{eq:identity_coinvariants} one has the following identity inside $\ms H$---with tacit use of $\pi_{\ms H}$:
\begin{equation}
	\bigl(X \otimes z_i^{-1} \bigr)^{(i)} \wh{\bm v} \otimes \wh v_{\infty} = \sum_{j \in J' \setminus \Set{i}} \frac{ X^{(j)} }{t_i - t_j} \wh{\bm v} \otimes \wh v_{\infty} \, ,
\end{equation}
where $\wh{\bm v} = \bigotimes_{j \in J'} \wh v_j \in \bm{\mc H}$.
In particular the action is trivial at infinity.

Looking at~\eqref{eq:action_sugawara_minus_one_finite_module_tame} and writing $L_{-1}^{(i)} (\wh{\bm v} \otimes \wh v_{\infty}) = \wh{\bm v}_i \otimes \wh v_{\infty}$ we find
\begin{equation}
	\label{eq:kz_connection}
	\wh{\bm v}_i = \frac{1}{\kappa + h^{\dual}} \sum_{j \in J' \setminus \Set{i}} \Biggl( \sum_k \frac{ (X^k)^{(i)}X_k^{(j)} }{ t_i - t_j } \Biggr) \wh{\bm v} = \frac{1}{\kappa + h^{\dual}} \sum_{j \in J' \setminus \Set{i}} \frac{\Omega^{(ij)}}{t_i - t_j} \wh{\bm v} \, ,
\end{equation}
where $\Omega^{(ij)} \ceqq \iota^{(ij)} (\Omega)$ denotes the embedding~\eqref{eq:embedding_quadratic_tensors} of the quadratic tensor~\eqref{eq:quadratic_tensor}---with $m = l = 0$.

One recovers the KZ connection~\cite{knizhnik_zamolodchikov_1984_current_algebra_and_wess_zumino_model_in_two_dimensions} on the sub-bundle $\mc H'_{\abs{\bm \lambda}} \hookrightarrow \mc H$, taking $V_{\infty}$ as auxiliary tame module.

\subsubsection{Tame modules in the finite part}
\label{sec:reduced_connection_wild_at_infinity}

Now allow $r_{\infty} \geq 1$ to be arbitrary.
What changes is
\begin{equation}
	X \otimes \tau_{\infty} (z_i^{-1}) \wh w_{\infty} = \sum_{l = 0}^{r_{\infty} - 2} t_i^l Xz_{\infty}^{l+1} \cdot \wh w_{\infty} \, ,
\end{equation}
for $X \in \mf g$, $\wh w_{\infty} \in W_{\infty}$ and $i \in J'$,  using the case $m = 1$ of~\eqref{eq:pole_expansion_finite}.
So the action is nontrivial at infinity if $r_{\infty} \geq 2$.

Then by~\eqref{eq:identity_coinvariants} one has the following identity inside $\ms H$---with tacit use of $\pi_{\ms H}$:
\begin{equation}
	\bigl(X \otimes z_i^{-1} \bigr)^{(i)} \wh{\bm v} \otimes \wh w_{\infty} = \left( \sum_{j \in J' \setminus \Set{i}} \frac{ X^{(j)} }{t_i - t_j} - \sum_{l = 0}^{r_{\infty} - 2} t^l_i \bigl( X z^{l+1} \bigr)^{(\infty)} \right) \wh{\bm v} \otimes \wh w_{\infty} \, .
\end{equation}
Thus looking at~\eqref{eq:action_sugawara_minus_one_finite_module_tame} one finds $L_{-1}^{(i)} (\wh{\bm v} \otimes \wh w_{\infty}) = \wh{\bm v}_i \otimes w_{\infty} + \mc D_i (\wh{\bm v} \otimes \wh w)$, where $\wh{\bm v}_i$ is as in~\eqref{eq:kz_connection}, and
\begin{equation}
	\mc D_i (\wh{\bm v} \otimes \wh w_{\infty}) %
	= \frac{1}{\kappa + h^{\dual}} \sum_{l = 0}^{r_{\infty} - 2} t_i^l \Omega^{(i\infty)}_{0,l+1} (\wh{\bm v} \otimes \wh w_{\infty}) \, ,
\end{equation}
using again the embedding $\iota^{(i\infty)} (\Omega_{0,l+1})$ of~\eqref{eq:quadratic_tensor} defined by~\eqref{eq:embedding_quadratic_tensors}.

\begin{remark*}
	E.g. if $r_{\infty} = 2$ then the new operator acts by
	\begin{equation}
		\label{eq:action_infinity_order_two}
		\mc D_i (\bm{\wh v} \otimes \widehat{w}_{\infty}) = \frac{\Omega_{01}^{(i\infty)} \wh{\bm v} \otimes \wh w_{\infty}}{\kappa + h^{\dual}}  \, .
	\end{equation}
	In this case the reduced connection is close to the dynamical KZ connection, i.e.~\cite[Eq.~3]{felder_markov_tarasov_varchenko_2000_differential_equations_compatible_with_kz_equations}.
	We will recover the very same `dynamical' Cartan term in \S~\ref{sec:dynamical_term}.
\end{remark*}

\subsubsection{Tame module at infinity}
\label{sec:reduced_connection_tame_at_infinity}

Suppose symmetrically that $r_{\infty} = 1$, but let $r_j$ be arbitrary for $j \in J'$.

\begin{proposition}
	\label{prop:irregular_kz_connection}
	One has $L_{-1}^{(i)} \wh{\bm w} \otimes \wh v_{\infty} = \wh{\bm w}_i \otimes \wh v_{\infty}$, with
	\begin{equation}
		\label{eq:irregular_kz_connection}
		\wh{\bm w}_i = %
		-\frac{1}{\kappa + h^{\dual}} \sum_{j \in J' \setminus \Set{i}} \Biggl( \sum_{m = 0}^{r_i-1}  \sum_{l = 0}^{r_j-1} \binom{m+l}{l} \frac{ \Omega^{(ij)}_{ml} \wh{\bm w}}{(t_i - t_j)^l(t_j - t_i)^{m+1}} \Biggr) \, .
	\end{equation}
\end{proposition}

\begin{proof}
	Postponed to \S~\ref{sec:proof_prop_irregular_kz_connection}.
\end{proof}

This is an irregular generalisation of the KZ connection, corresponding to an action of the universal connection of ~\cite{reshetikhin_1992_the_knizhnik_zamolodchikov_system_as_a_deformation_of_the_isomonodromy_problem}.\footnote{Compare also~\eqref{eq:irregular_kz_connection} with~\cite[Eqs.~B.6 and B.7]{gaiotto_lamypoirier_2013_irregular_singularities_in_the_h3plus_wzw_model}, where $\mf g = \mf{sl}(2,\mb C)$: this should be a formalisation of fn.~6 of op. cit.}

\begin{remark*}
	By op.~cit. the flat connection~\eqref{eq:irregular_kz_connection} admits an isomonodromy system as semiclassical limit: precisely the irregular isomonodromy system on $\mb CP^1$ for variations of the positions of the poles (the tame isomonodromy times, as considered in~\cite{klares_1979_sur_une_classe_de_connexions_relatives}).
	This generalises the same fact from the tame case: the quantisation of the Schlesinger system~\cite{schlesinger_1905_ueber_die_loesungen_gewisser_linearer_differentialgleichungen_als_funktionen_der_singularen_punkte} yields the KZ connection~\cite{reshetikhin_1992_the_knizhnik_zamolodchikov_system_as_a_deformation_of_the_isomonodromy_problem,harnad_1996_quantum_isomonodromic_deformations_and_the_knizhnik_zamolodchikov_equations}.
\end{remark*}

\subsubsection{General case}
\label{sec:reduced_connection_general}

Finally take $r_{\infty} \geq 1$ to be generic as well.

\begin{proposition}
	\label{prop:irregular_kz_connection_wild_infinity}
	One has $L_{-1}^{(i)} \wh{\bm w} \otimes \wh w_{\infty} = \wh{\bm w}_i \otimes \wh w_{\infty} + \mc D_i (\wh{\bm w} \otimes \wh w_{\infty})$, with $\wh{\bm w}_i$ as in~\eqref{eq:irregular_kz_connection} and
	\begin{equation}
		\label{eq:irregular_dynamical_term}
		\mc D_i (\wh{\bm w} \otimes \wh w_{\infty})
		= \frac{1}{\kappa + h^{\dual}} \sum_{m = 0}^{r_i-1} \sum_{l = 0}^{r_{\infty} - m - 1} \binom{m+l}{l} t_i^l \Omega^{(i\infty)}_{m,m+l+1} (\wh{\bm w} \otimes \wh w_{\infty}) \, .
	\end{equation}
\end{proposition}

\begin{proof}
	This is a generalisation of Prop.~\ref{prop:irregular_kz_connection} where moreover \begin{equation}
		X \otimes \tau_{\infty} (z_i^{-m}) \wh w_{\infty} = \sum_{l = 0}^{r_{\infty} - m - 1} \binom{m+l-1}{l} t_i^l Xz_{\infty}^{m+l} \cdot \wh w_{\infty} \, ,
	\end{equation}
	for $X \in \mf g$, $\wh w_{\infty} \in W_{\infty}$ and $i \in J'$,  using the general case of~\eqref{eq:pole_expansion_finite}.
	Now the action is nontrivial at infinity for $r_{\infty} \geq m+1$, and the result still follows from~\eqref{eq:action_sugawara_minus_one_finite_module_general}.
\end{proof}

\subsection{Description on finite modules: second version}

Finally one may consider the setup of \S~\ref{sec:irregular_conformal_blocks_space_with_duals}, i.e. put a $\theta$-dual module $W'_{\theta}$ at infinity.
In the analogue of \S\S~\ref{sec:reduced_connection_tame} and~\ref{sec:reduced_connection_tame_at_infinity}---when the module at infinity is tame---the description of the reduced connection does not change, using~\eqref{eq:annihilator_cyclic_vector_theta_dual}.
In the remaining cases one finds the action of the same quadratic tensors on the last slot, acting on the $\theta$-dual.

Hence in the next section we will introduce a \emph{universal} versions of the reduced connection, looking at~\eqref{eq:irregular_kz_connection} and~\eqref{eq:irregular_dynamical_term}, to treat the two versions on the same footing.

\section{Universal connections}
\label{sec:yang_baxter}

Fix again a depth $p \geq 1$, an integer $n \geq 1$, and the finite ordered sets
\begin{equation}
	\Set{ 1, \dc, n} = J' \subseteq J = \Set{ 1, \dc, n,\infty} \, .
\end{equation}
Consider then the nonautomous (quantum) Hamiltonian systems
\begin{equation}
	\wh H_i = \wh H^{(p)}_i \colon C_n(\mb C) \longrightarrow U \bigl( \mf g_p \bigr)^{\otimes \,  \abs J} \, ,
\end{equation}
with Hamiltonians $\wh H_i = \wh H'_i + \wh H''_i$ for $i \in J'$, where
\begin{equation}
	\label{eq:hamiltonians_finite}
	\wh H'_i(\bm t) \ceqq -\frac{1}{\kappa + h^{\dual}} \sum_{j \in J' \setminus \Set{i}} \Biggl(  \sum_{m,l = 0}^{p-1} \Omega^{(ij)}_{ml} \binom{m+l}{l} (-1)^m (t_i - t_j)^{-1-m-l} \Biggr) \, ,
\end{equation}
and
\begin{equation}
	\label{eq:hamiltonians_infinity}
	\wh H''_i(\bm t) \ceqq \frac{1}{\kappa + h^{\dual}} \sum_{m,l = 0}^{p-1} \Omega^{(i\infty)}_{m,m+l+1} \binom{m+l}{l} t_i^l \, ,
\end{equation}
as suggested by~\eqref{eq:irregular_kz_connection} and~\eqref{eq:irregular_dynamical_term}.

These Hamiltonians are equivalent to the \emph{universal} connection (at depth $p$):
\begin{equation}
	\label{eq:universal_connection}
	\nabla_p = \dif - \varpi_p \, , \quad \varpi_p = \varpi'_p + \varpi''_p \, , \quad \varpi'_p \ceqq \sum_{J'} \wh H'_i \dif t_i \, , \quad \varpi''_p \ceqq \sum_{J'} \wh H''_i \dif t_i \, ,
\end{equation}
defined on the trivial vector bundle $U(J,p) \ceqq C_n(\mb C) \times U \bigl( \mf g_p \bigr)^{\otimes \, \abs J} \to C_n(\mb C)$ by means of the $U(\mf g_p)^{\otimes \, \abs J}$-valued 1-forms $\varpi'_p$ and $\varpi''_p$ on the base space.
This generalises~\cite{reshetikhin_1992_the_knizhnik_zamolodchikov_system_as_a_deformation_of_the_isomonodromy_problem} with a nontrivial action at infinity.

Then for every choice of singular modules labeled by $J$ there is an action of~\eqref{eq:universal_connection} on $\bm{\mc H}$, for $p \gg 0$, which reproduces the most general case of \S~\ref{sec:description_finite_modules} (with $\theta$-duals or not).
In particular there are induced \emph{integrable} quantum Hamiltonian systems.
Hence one expects~\eqref{eq:universal_connection} to be flat before taking representations, as we will show.

\begin{remark*}
	One directly checks that
	\begin{equation}
		\pd{\wh H'_j}{t_i} - \pd{\wh H'_i}{t_j} = 0 \, , \qquad \text{and} \qquad \pd{\wh H''_j}{t_i} = \delta_{ij} \, , \quad \text{for } i,j \in J' \, ,
	\end{equation}
	so (strong) flatness is equivalent to the commutativity of the quantum Hamiltonians.
\end{remark*}

\subsection{Flatness at finite distance}

The 1-form defining the Hamiltonians~\eqref{eq:hamiltonians_finite} can be written
\begin{equation}
	\label{eq:1_form_universal_connection_finite}
	\varpi'_p = \frac{1}{\kappa + h^{\dual}} \sum_{i \neq j \in J'} r_p^{(ij)} (t_i - t_j) \dif \, (t_i - t_j) \, ,
\end{equation}
where $r_p \colon \mb C \setminus \Set{0} \to \mf g_p^{\otimes 2}$ is the following rational function:
\begin{equation}
	\label{eq:rational_function_finite}
	r_p(t) \ceqq -\sum_{m,l = 0}^{p-1} \Omega_{ml} \otimes (-1)^m\binom{m+l}{l} t^{-1-m-l} \, .
\end{equation}

\begin{remark*}
	It is easy to see that $r_p$ is skew-symmetric, meaning
	\begin{equation}
		\label{eq:skew_symmetry_classical_r_matrix}
		r_p^{(ij)}(t) + r_p^{(ji)}(-t) = 0 \, , \qquad \text{for } t \in \mb C \setminus \Set{0}, \, i,j \in J' \, .
	\end{equation}
\end{remark*}

The study of the connection $\nabla'_p \ceqq \dif - \varpi'_p$ is closely related to the theory of the classical Yang--Baxter equation (CYBE)~\cite{belavin_drinfeld_1982_solutions_of_the_classical_yang_baxter_equation_for_simple_lie_algebras}.
In particular flatness  (for $\abs{J'} \geq 3$) is equivalent to the CYBE for~\eqref{eq:rational_function_finite} in the Lie algebra $\mf g_p$, i.e. to the following identity inside $\mf g_p^{\otimes 3}$:
\begin{equation}
	\bigl[ r_p^{(12)}(t_{12}),r_p^{(13)}(t_{13}) \bigr] + \bigl[ r_p^{(13)}(t_{13}),r_p^{(23)}(t_{23}) \bigr] +
	\bigl[ r_p^{(12)}(t_{12}),r_p^{(23)}(t_{23}) \bigr] = 0 \, ,
\end{equation}
where $t_{ij} \ceqq t_i - t_j$.

\begin{theorem}[cf.~\cite{reshetikhin_1992_the_knizhnik_zamolodchikov_system_as_a_deformation_of_the_isomonodromy_problem}]
	\label{thm:cybe}
	The rational function~\eqref{eq:rational_function_finite} is a solution of the CYBE.
\end{theorem}

\begin{proof}
	We will reduce the proof to the well-known case $p = 1$, where $\mf g_p = \mf g$.
	In this case we have the classical result that the rational function $r_1(t) = \Omega t^{-1}$ is a skew-symmetric solution of the CYBE~\cite{belavin_drinfeld_1982_solutions_of_the_classical_yang_baxter_equation_for_simple_lie_algebras}, which is an easy consequence of the Drinfeld--Kohno relations
	$\bigl[ \Omega^{(ij)},\Omega^{(ik)}+\Omega^{(jk)} \bigr] = 0$, and the Arnold relations~\cite{arnold_1969_the_cohomology_ring_of_the_group_of_dyed_braids}:
	\begin{equation}
		\label{eq:arnold_identity}
		\frac{1}{t_{ij}t_{jk}} + \frac{1}{t_{jk}t_{ki}} + \frac{1}{t_{ki}t_{ij}}  = 0 \, .
	\end{equation}

	To prove the general case consider the identification $\mf g_p^{\otimes 2} \simeq \mf g^{\otimes 2} \otimes A(2,p)$, where $A(n,p) \ceqq \mb C \llbracket w_1,\dc,w_n \rrbracket \big\slash \mf{I}_p$ is the quotient of the power-series ring by the ideal $\mf{I}_p = \bigl( w_1^p,\dc,w_n^p \bigr)$ generated by $\Set{ w_1^p, \dc, w_n^p }$.
	In this identification $\Omega_{ml} = \Omega \otimes w_1^mw_2^l$, and~\eqref{eq:rational_function_finite} can be written
	\begin{equation}
		r_p(t) = \Omega \otimes \tau_{(0,0)}^{(p)} (f_t) \in \mf g_p^{\otimes 2} \, , \qquad \text{where} \qquad f_t(w_i,w_j) \ceqq \frac{1}{t + w_i - w_j} \, ,
	\end{equation}
	and where $\tau_{(0,0)}^{(p)}(f_t)$ is the class mod $\mf{I}_p$ of the Taylor expansion of $f_t$ at the origin.
	Then, up to the identification $\mf g_p^{\otimes 3} \simeq \mf g^{\otimes 3} \otimes A(3,p)$, the CYBE follows again from \eqref{eq:arnold_identity}, with $t_i$ replaced by $t_i - w_i$, for $i \in \Set{1,2,3}$.
\end{proof}

Hence we have an inverse system of classical $r$-matrices, with respect to the canonical projections $\mf g \llbracket z \rrbracket \big\slash z^{\bullet + 1} \mf g \llbracket z \rrbracket \twoheadrightarrow \mf g \llbracket z \rrbracket \big\slash z^{\bullet} \mf g \llbracket z \rrbracket$, corresponding to an inverse system of flat vector bundles $\bigl( U(n,p),\nabla'_p \bigr)$ over the space of configurations of $J'$-tuples of points in the complex plane.
The inverse limit of the vector bundles is naturally identified with the trivial vector bundle with fibre $U \bigl( \mf g \llbracket z \rrbracket \bigr)^{\!\wh{\otimes} \, \abs J}$, the completion of the $n$-th tensor power of the positive part of the loop algebra.

\begin{remark*}
	The inverse limit $r_\infty(t) = \varprojlim_p r_p(t) \in \mf g^{\otimes 2} [t^{-1}] \llbracket z_1,z_2 \rrbracket$ is a solution of the CYBE in a completion of $\mf g \llbracket z \rrbracket^{\otimes 3} \otimes \ms O_{C_{3} (\mb C)} \bigl( C_{3}(\mb C) \bigr)$.

	Analogously on the representation-theoretic side one may consider characters of the Lie subalgebra
	\begin{equation}
		\mf S^{(\infty)} \ceqq \bigcap_{p \geq 1} \mf S^{(p)} = \mf b^+ \llbracket z \rrbracket \oplus \mb C K \subseteq \wh{\mf g} \, ,
	\end{equation}
	using~\eqref{eq:singular_lie_algebra}.
	Then $\mf S^{(\infty)}_{\ab} \simeq \mf h \llbracket z \rrbracket \oplus \mb C K$, so the induced non-smooth modules $\wh W^{(\infty)}$ depend on infinitely many Cartan parameters (and a level $\kappa$), and are generated over $U(\mc L\mf n^-)$ by a single vector annihilated by $\mf n^+ \llbracket z \rrbracket$.
	Under~\eqref{eq:residue_pairing} the parameters of these modules correspond to principal parts of connections with \emph{essential} singularities.
\end{remark*}

\subsection{Flatness overall}

The 1-form defining the Hamiltonians~\eqref{eq:hamiltonians_infinity} can be written
\begin{equation}
	\varpi''_p = \frac{1}{\kappa + h^{\dual}} \sum_{i \in J'} s_p^{(i\infty)}(t_i) \dif t_i \, ,
\end{equation}
where $s_p \colon \mb C \setminus \Set{0} \to \mf g_p^{\otimes 2}$ is the following rational function:
\begin{equation}
	\label{eq:rational_function_infinity}
	s_p(t) \ceqq \sum_{m,l = 0}^{p-1} \Omega_{m,m+l+1} \otimes \binom{m+l}{l} t^l \, .
\end{equation}

\begin{theorem}
	\label{thm:universal_flatness}
	The universal connection $\nabla_p$ is flat for $p \geq 1$.
\end{theorem}

\begin{proof}
	Reasoning as in the proof of Thm.~\ref{thm:cybe} consider the function
	\begin{equation}
		g_t(w_i,w_j) \ceqq \frac{w_j}{1 - w_j(t + w_i)} \, .
	\end{equation}
	Then one directly checks that the Taylor expansion of $g_t$ at the origin satisfies
	\begin{equation}
		s_p(t) = \Omega \otimes \tau^{(p)}_{(0,0)} (g_t) \, ,
	\end{equation}
	and we can conclude by proving a version of the CYBE in the Lie algebra $\mf g_p$.

	Namely by Thm.~\ref{thm:cybe} the commutator of two Hamiltonians becomes
	\begin{equation}
		\bigl[ \wh H_i,\wh H_j \bigr]
		= \bigl[ r_p^{(ij)}(t_{ij}),s_p^{(i\infty)}(t_i) \bigr] + \bigl[ r_p^{(ij)}(t_{ij}),s_p^{(j\infty)}(t_j) \bigr] + \bigl[ s_p^{(i\infty)}(t_i),s_p^{(j\infty)}(t_j) \bigr] \, ,
	\end{equation}
	using the fact that actions on disjoint pairs of slots commute, and the skew-symmetry~\eqref{eq:skew_symmetry_classical_r_matrix}.
	Now we can use the standard Drinfeld--Kohno relations to reduce flatness (for all $p \geq 1$) to a variation of the Arnold relations~\eqref{eq:arnold_identity}, namely to the following identity:
	\begin{equation}
		g_{t_i} (w_i,w_{\infty}) g_{t_j} (w_j,w_{\infty}) + f_{t_{ij}}(w_i,w_j) \bigl( g_{t_i}(w_i,w_{\infty}) - g_{t_j}(w_j,w_{\infty}) \bigr) = 0 \, ,
	\end{equation}
	where $f_t = f_t(w_i,w_j)$ is as in the proof of Thm.~\ref{thm:cybe}.
\end{proof}

\begin{remark}
	One can give a more symmetric expression of~\eqref{eq:universal_connection}, with no special role for the marked point at infinity.

	To this end consider the `generating' function
	\begin{equation}
		\label{eq:generating_function}
		\varphi(w_i,w_j) \ceqq \frac{1}{w_i - w_j} \, ,
	\end{equation}
	which is a meromorphic function on $\mb C^2$ with poles along $\Set{ w_i = w_j } \subseteq \mb C^2$---and only there.
	It can be extended (by zero) to a meromorphic function on the complex surface $\Sigma^2 \setminus \Set{ (\infty,\infty) }$, so we can take Taylor expansions $\tau_{(p_i,p_j)} (\varphi)$ of $\varphi$ at any pair of \emph{distinct} points $p_i,p_j \in \Sigma$---using the local coordinates $w_i^{-1}$ and $w_j^{-1}$ at infinity.

	Then analogously to the above one checks that
	\begin{equation}
		\tau^{(p)}_{(p_i,p_j)} (\varphi) = r_p(t_{ij}) \, , \qquad \tau^{(p)}_{(p_i,\infty)}(\varphi) = s_p(t_i) \, ,
	\end{equation}
	for points $p_i, p_j \in \Sigma$ at finite distance of coordinates $t_i,t_j \in \mb C$, respectively.
	Hence
	\begin{equation}
		\varpi_p = \frac{1}{\kappa + h^{\dual}} \sum_{i \neq j \in J} \tau^{(p)}_{(p_i,p_j)}(\varphi) \dif t_{ij} \, ,
	\end{equation}
	and all marked points are treated the same.

	Then the flatness of~\eqref{eq:universal_connection} for $p \geq 1$ is equivalent to \emph{generalised} Arnold relations, relating the Taylor expansions of~\eqref{eq:generating_function} at pairs extracted from a triple of distinct points on the Riemann sphere.
\end{remark}

Hence we find again an inverse system of flat vector bundles $\bigl( U(J,p),\nabla_p \bigr)$, over the space of configurations of $J'$-tuples of points in the complex plane.

\subsection{Connection on coinvariants}

The universal connection~\eqref{eq:universal_connection} has a well-defined action on sections taking values in the space of $\mf g$-coinvariants of $U (\mf g_p)^{\otimes n}$.

To prove this consider the canonical embedding $\mf g \hookrightarrow \mf g_p \simeq \mf g \ltimes \mf b_p$ and the universal embedding $\mf g_p \hookrightarrow U (\mf g_p)$.
Composing them we let $\mf g$ act on $U(\mf g_p)$ in the regular representation, and finally take the tensor product action (as in~\eqref{eq:action_meromorphic_functions} in the case of constant functions).
Then we get a $\mf g$-action on differential forms with values in the flat vector bundle $\bigl( U(J,p), \nabla_p \bigr)$.

\begin{proposition}
	\label{prop:g_invariance_universal_connection}
	The $\mf g$-action is flat for all $p \geq 1$.
\end{proposition}

Note this is a particular case of a compatibility such as~\eqref{eq:compatibility_connections}, for constant sections of the trivial bundle $C_n(\mb C) \times \mf g \to C_n(\mb C)$, equipped with the trivial connection.

\begin{proof}
	Postponed to \S~\ref{sec:proof_prop_g_invariance_universal_connection}.
\end{proof}

It follows that~\eqref{eq:universal_connection} preserves sections with values in $\mf g U (\mf g_p)^{\otimes \, \abs J} \subseteq U (\mf g_p)^{\otimes \, \abs J}$, so that a reduced (flat) connection is well-defined on the space of $\mf g$-coinvariants of the tensor product.

\section{On conformal transformations}
\label{sec:conformal_transformations}

Consider the action of M\"{o}bius transformations on $\Sigma = \mb{P}(\mb C^2)$, that is
\begin{equation}
	g.\bigl[ t_1 \colon t_2 \bigr] = \bigl[ at_1 + bt_2 \colon c t_1 + d t_2 \bigr] \, ,
\end{equation}
for $(t_1,t_2) \in \mb C^2 \setminus \Set{\bm{0}}$, with $g = g(a,b,c,d)$ given by numbers $a,b,c,d \in \mb C$ such that $ad - bc = 1$.
In the standard affine chart $U = \Sigma \setminus \set{ [1 \colon 0] } \xrightarrow{t}{} \mb C$ we then have the subgroup of affine transformation of the complex plane, with diagonal action on $C_n(\mb C) \subseteq \mb C^n$, and with induced pullback (right) action on sections of vector bundles over that base.

In particular the translations $t \mapsto t + b$ correspond to $a = d = 1$ and $c = 0$.
This is the 1-parameter subgroup corresponding to the infinitesimal generator $E \in \Lie\bigl( \PSL(2,\mb C)\bigr) = \mf{sl}(2,\mb C)$, and the associated infinitesimal action reads
\begin{equation}
	\label{eq:infinitesimal_translations}
	\eval[2]{\od{ \bigl( \wh{\bm w} \circ \gamma \bigr) (\varepsilon)}{\varepsilon}}_{\varepsilon = 0} = \eval[2]{\od{ \wh{\bm w} (\bm t + \varepsilon)}{\varepsilon}}_{\varepsilon = 0} = \sum_{i \in J'} \pd{\wh{\bm w}}{t_i} \, ,
\end{equation}
considering the path $\gamma \colon \varepsilon \mapsto g(1,\varepsilon,0,1)$.

Analogously dilations correspond to the 1-parameter subgroup generated by $H \in \mf{sl}(2,\mb C)$, and the associated infinitesimal action is given by the Euler vector field
\begin{equation}
	\label{eq:infinitesimal_dilations}
	\eval[2]{\od{ \bigl( \wh{\bm w} \circ \gamma \bigr) (\varepsilon)}{\varepsilon}}_{\varepsilon = 0} = \eval[2]{\od{ \wh{\bm w} \bigl( (1 + \varepsilon)^2 \bm t \bigr)}{\varepsilon}}_{\varepsilon = 0} = 2 \sum_{i \in J'} t_i \pd{\wh{\bm w}}{t_i} \, ,
\end{equation}
considering the path $\gamma \colon \varepsilon \mapsto g \bigl( 1+\varepsilon,0,0,(1+\varepsilon)^{-1} \bigr)$.

\begin{proposition}
	\label{prop:affine_equivariance}

	Suppose the module at infinity is tame.
	Then the action of affine transformations on horizontal sections of the bundle of covacua reads
	\begin{equation}
		\label{eq:affine_transformations}
		\wh{\bm w} (\bm t') = \prod_{i \in J'} \exp \bigl( a L_0^{(i)} \bigr) \cdot \wh{\bm w} (\bm t) \, ,
	\end{equation}
	where $\bm t' = (t'_i)_{i \in J'}$ with $t'_i = e^{2a} t_i + b$.
	In particular horizontal sections are invariant under translations.
\end{proposition}

\begin{proof}
	Postponed to~\ref{sec:proof_prop_affine_equivariance}.
\end{proof}

\begin{remark*}
	As in the tame case, the $\mf g$-coinvariance implies
	\begin{equation}
		\sum_{i \neq j \in J} \Omega^{(ij)} \wh{\bm w} + \sum_{k \in J} \Omega^{(kk)} \wh{\bm w} = 0 \, ,
	\end{equation}
	in the space $\ms H$.
	The action of $\Omega^{(kk)}$ is that of the quadratic Casimir~\eqref{eq:quadratic_casimir} on the $k$-th slot, so this term acts diagonally and can be exponentiated to find the usual conformal weight (cf. Rmk.~\ref{rem:conformal_weight}).
	The point is that in general the dilation action has further nonscalar terms.
\end{remark*}

\section{A different dynamical term at infinity}
\label{sec:dynamical_term}

In this section we generalise the dynamical KZ connection~\cite{felder_markov_tarasov_varchenko_2000_differential_equations_compatible_with_kz_equations}, varying the setup of \S~\ref{sec:setup}.

Namely note another natural family of Lie algebras $\ul{\mf S}^{(p)} \subseteq \mf S^{(p)} \subseteq \wh{\mf g}$ is given by
\begin{equation}
	\ul{\mf S}^{(p)} \ceqq \mf h \llbracket z \rrbracket + z^p \mf g \llbracket z \rrbracket \oplus \mb C K\, .
\end{equation}
The derived Lie algebra of $\ul{\mf S}^{(1)}$ yields the first `level subalgebra' of~\cite{fedorov_2010_irregular_wakimoto_modules_and_the_casimir_connection}, then the two differ for $p \geq 2$.
One can then define (smooth) induced modules $\ul{\wh W}$ as in \S~\ref{sec:setup}, where $\ul{\wh W} = \ul{\wh W}^{(p)}_{\ul \chi}$ depends on a character $\ul \chi \colon \ul{\mf S}^{(p)} \to \mb C$.
However one does not recover the standard affine Verma module as the starting element of the family, contrary to~\eqref{eq:affine_singular_module}---which is one motivation behind Def.~\ref{def:affine_singular_module}.

Moreover one has $\ul{\mf S}^{(p)} \simeq \mf h_{2p} \oplus \mb CK$, analogously to Lem.~\ref{lem:abelianisation}, so for $p = 1$ a character is defined by elements $\ul{\lambda} \in \mf h^{\dual}$ and by the irregular Cartan term $\mu \in ( \mf h \otimes z )^{\dual}$ (plus the choice of a level $\kappa$).
Hence for $p = 1$ we see that~\eqref{eq:residue_pairing} matches up the parameters of $\ul{\wh W}$ with (formal normal forms of) principal parts of meromorphic connections at poles of order two, but in general only poles of even order can be obtained with this construction, contrary to~\eqref{eq:affine_singular_module}---which is another motivation behind Def.~\ref{def:affine_singular_module}.

So we can put the module $\ul{\wh W} = \ul{\wh W}^{(1)}_{\ul \chi}$ at infinity in the tensor product $\wh{\mc H}$, and consider the spaces of coinvariants $\ms H$ as in \S~\ref{sec:irregular_conformal_blocks}.
The proofs of Props.~\ref{prop:surjectivity_coinvariants}--\ref{prop:auxiliary_tame_module} can be adapted by introducing suitable filtrations on $\ul{\wh W}$ and $\ul W = U (\mf g \llbracket z \rrbracket ) \ul w$, where $\ul w \in \ul{\wh W}$ is the cyclic vector, as well as the whole of \S~\ref{sec:compatibility_connections}.
Hence in brief one can use $\ul W$ as auxiliary module at infinity, which yields a different `dynamical' Cartan term in the reduced connection---with respect to~\eqref{eq:action_infinity_order_two}.
Namely~\eqref{eq:action_infinity_order_two} simplifies to
\begin{equation}
	\mc D_i (\bm{\wh v} \otimes \ul w) = \frac{1}{\kappa + h^{\dual}} \sum_k \mu_k H_k^{(i)} \cdot \wh{\bm v} \otimes \ul w \, ,
\end{equation}
where $(H_k)_k$ is a $(\cdot \mid \cdot)$-orthonormal basis of $\mf h$, using $(\mf n^+ \oplus \mf n^-) \otimes z_{\infty} \cdot \ul w = 0$, $H_k z_{\infty} \cdot \ul w = \mu_k \ul w$, and writing $\mu_k = \Braket{ \mu | H_k z_{\infty} }$.

We see the reduced connection generalises the dynamical KZ equations, i.e.~\cite[Eq.~3]{felder_markov_tarasov_varchenko_2000_differential_equations_compatible_with_kz_equations}, and it coincides with it when the modules over finite points are tame.\footnote{Replace $\kappa + h^{\dual} \in \mb C$ with `$\kappa$' and $\sum_k \mu_k H_k \in \mf h$ with `$\mu$' to retrieve the exact~\cite[Eq.~3]{felder_markov_tarasov_varchenko_2000_differential_equations_compatible_with_kz_equations}.}
So we recover the connection of Felder--Markov--Tarasov--Varchenko (FMTV), over variations of marked points, as a particular case of this construction.

\begin{remark}
	Note the whole of the FMTV connection also allows for variations of the irregular part $\mu \in (\mf h \otimes z)^{\dual}$, in addition to the deformations \`{a} la Klar\`es considered here~\cite{klares_1979_sur_une_classe_de_connexions_relatives}.
	In particular when there is only one simple pole the resulting flat connection for variations of $\mu$ is the DMT connection~\cite{millson_toledanolaredo_2005_casimir_operators_and_monodromy_representations_of_generalised_braid_groups,toledanolaredo_2002_a_kohno_drinfeld_theorem_for_quantum_weyl_groups}, which is derived from a representation-theoretic setup in~\cite[\S~3.11]{fedorov_2010_irregular_wakimoto_modules_and_the_casimir_connection}, and~\cite[\S~3.7]{feigin_frenkel_toledanolaredo_2010_gaudin_models_with_irregular_singularities} (for the latter see also~\cite{vinberg_1990_some_commutative_subalgebras_of_a_universal_enveloping_algebra}).
\end{remark}

\begin{remark}[On quantisation of isomonodromy connections]
	Just as in the case of the KZ connection, a different derivation of these flat connections has been obtained by (filtered) deformation quantisation of isomonodromy systems, this time importantly for \emph{irregular} meromorphic connections.

	Namely~\cite{boalch_2002_g_bundles_isomonodromy_and_quantum_weyl_groups} derived the DMT connection from the quantisation of a dual version of the Schlesinger system.
	This is related to the usual Schlesinger system by the Harnad duality~\cite{harnad_1994_dual_isomonodromic_deformations_and_moment_maps_to_loop_algebras}, i.e. the Fourier--Laplace transform (cf.~\cite{tarasov_varchenko_2002_duality_for_knizhnik_zamolodchikov_and_dynamical_equations} on the quantum side).
	In the same spirit, the whole of FMTV connection can be obtained by quantising the isomonodromy system of Jimbo--Miwa--M\^{o}ri--Sato~\cite{jimbo_miwa_mori_sato_1980_density_matrix_of_an_impenetrable_bose_gas_and_the_fifth_painleve_transcendent} (see~\cite[\S~11]{rembado_2019_simply_laced_quantum_connections_generalising_kz}; more generally see op. cit. and~\cite{rembado_2020_symmetries_of_the_simply_laced_quantum_connections_and_quantisation_of_quiver_varieties} for a further extension to connections with poles of order three including all the above cases).
\end{remark}

\section*{Outlook}

As explained in the introduction we also wish to consider flat quantum connections along variations of irregular types (i.e. variations of `wild' Riemann surface structures on the sphere~\cite{boalch_2012_hyperkaehler_manifolds_and_nonbelian_hodge_theory_of_irregular_curves}).
Two viable viewpoints to introduce them are:
\begin{enumerate}
	\item the quantisation of the \emph{full} irregular isomonodromy connections, in the spirit of~\cite{boalch_2002_g_bundles_isomonodromy_and_quantum_weyl_groups,rembado_2019_simply_laced_quantum_connections_generalising_kz}, generalising the simply-laced quantum connections (which quantise the simply-laced isomonodromy systems~\cite{boalch_2012_simply_laced_isomonodromy_systems});

	\item  considering quantum symmetries: the quantum/Howe duality~\cite{baumann_1999_the_q_weyl_group_of_a_q_schur_algebra} was used in~\cite{toledanolaredo_2002_a_kohno_drinfeld_theorem_for_quantum_weyl_groups} to relate KZ and the `Casimir' connection of De Concini and Millson--Toledano Laredo (DMT)~\cite{millson_toledanolaredo_2005_casimir_operators_and_monodromy_representations_of_generalised_braid_groups}, and at the level of isomonodromy systems corresponds to the Harnad duality~\cite{harnad_1994_dual_isomonodromic_deformations_and_moment_maps_to_loop_algebras}.
	      An analogous quantisation of the Fourier--Laplace transform may be taken here in order to turn the variations of marked points into variations of irregular types, extending the viewpoint of~\cite{boalch_2007_regge_and_okamoto_symmetries,rembado_2020_symmetries_of_the_simply_laced_quantum_connections_and_quantisation_of_quiver_varieties}.
\end{enumerate}

Another natural direction to pursue is the higher-genus case, noting in that case the moduli spaces of connections on holomorphically trivial bundles have positive codimension inside the full de Rham spaces.

Finally one may try to introduce integrality conditions, and lift this Lie-algebra representation setup to Lie groups, with a view towards the geometric quantisation of coadjoint $G_p$-orbits (along the lines of the Borel--Weyl--Bott theorem~\cite{serre_1995_representations_lineaires_et_espaces_homogenes_kaehleriens_des_groupes_de_lie_compacts_d_apres_armand_borel_et_andre_weil,tits_1955_sur_certaines_classes_d_espaces_homogens_de_groupes_de_lie,bott_1957_homogeneous_vector_bundles}, or more generally of the orbit method~\cite{kirillov_2004_lectures_on_the_orbit_method}).
Another approach we will try in this direction is that of the quantisation of the nilpotent Birkhoff orbits $\mc O_B \subseteq \mf b_p^{\dual}$~\cite{boalch_2001_symplectic_manifolds_and_isomonodromic_deformations}.\fn{The `semisimple nonresonant' coadjoint orbits in dual truncated current Lie algebras have been quantized in~\cite{calaque_felder_rembado_wentworth_2024_wild_orbits_and_generalised_singularity_modules_stratifications_and_quantisation}.}

\section*{Acknowledgments}

We thank Philip Boalch and Leonid Rybnikov for helpful discussions.

\appendix

\section{Standard notions/notations}
\label{sec:appendix_1}

\subsection*{Duals}

The (algebraic) dual of a vector space $W$ is denoted by $W^{\dual} = \Hom_{\mb C}(W,\mb C)$, and the natural pairing $W^{\dual} \otimes W \to \mb C$ by $\alpha \otimes w \mapsto \Braket{\alpha | w }$.
If $I$ is a set and $W = \bigoplus_i \mc F_i(W)$ an $I$-graded vector space then the \emph{restricted/graded} dual of $(W,\mc F_{\bullet})$ is the $I$-graded vector space $W^* \ceqq \bigoplus_{i \in I} \mc F_i(W)^{\dual} \subseteq \prod_{i \in I} \mc F_i(W)^{\dual} \simeq W^{\dual}$.

\subsection*{Gradings and filtrations}

If $(I,\leq)$ is a totally ordered set and $W = \bigoplus_{i \in I} \mc F_i(W)$ an $I$-graded vector space, the associated $I$-filtration on $W$ is defined by the subspaces $\mc F_{\leq i} \ceqq \bigoplus_{j \leq i} \mc F_j(W)$.

If $I$ and $J$ are sets and $W_j = \bigoplus_{i \in I} \mc F^{(j)}_i (W_j)$ a $J$-family of $I$-graded vector spaces, the \emph{tensor product} $I^{J}$-grading on $W = \bigotimes_{j \in J} W_j$ is defined by the subspaces
\begin{equation}
	\bm{\mc F}_{\bm i} \ceqq \Biggl( \, \bigotimes_{j \in J} \mc F^{(j)}_{\bm i(j)} \Biggr) \, , \qquad \text{for } \bm i \colon J \to I \, .
\end{equation}
If further $(I,\leq)$ is a totally-ordered $\mb Z$-module then the \emph{tensor product} $I$-filtration on $W$ is defined by the subspaces
\begin{equation}
	\bm{\mc F}_{\leq i} \ceqq \bigoplus_{\sum_{j \in J} \bm i(j) \leq i} \bm{\mc F}_{\bm i} \, , \qquad \text{for } i \in I \, .
\end{equation}

\subsection*{Lie-algebraic constructions}

Let $L$ be a Lie algebra.
The \emph{abelianisation} of $L$ is the abelian Lie algebra $L_{\ab} \ceqq L \big\slash \bigl[ L,L \bigr]$, and the \emph{opposite} of $L$ is the Lie algebra $L^{\op}$ on the same vector space, with bracket $\bigl[ X,Y \bigr]_{L^{\op}} \ceqq \bigl[ Y,X \bigr]_L$ for $X,Y \in L$.

If $p \geq 1$ is an integer and `$z$' a variable then the associated truncated-current Lie algebra of \emph{depth} $p$ is
\begin{equation}
	L_p \ceqq L\llbracket z \rrbracket \big\slash z^p L\llbracket z \rrbracket \simeq L \otimes \bigl( \mb C \llbracket z \rrbracket \big\slash z^p \mb C \llbracket z \rrbracket \bigr) \, ,
\end{equation}
coming with a projection $L_p \twoheadrightarrow L_1 = L$.
There is then a canonical vector-space isomorphism $L_p \simeq \bigoplus_{i = 0}^{p-1} L \otimes z^i$, which can be upgraded to an isomorphism of Lie algebras if one defines a Lie bracket on the direct sum by truncating terms of degree greater than $p-1$.

If $W$ is a left $L$-module then the space of $L$-\emph{coinvariants} is $W_L \ceqq W \big\slash L W$, where $L W \ceqq \sum_{X \in L} X W \subseteq W$---in particular $L_{\ab}$ is the space of $\ad_L$-coinvariants.

\section{Computations}
\label{sec:appendix_2}

\subsection{Proof of Prop.~\ref{prop:sugawara_eigenvalues}}
\label{sec:proof_prop_sugawara_eigenvalues}

\begin{proof}
	Set $a_k^{(j)} \ceqq \Braket{ a_j | H_k z^j }$ for $k \in \Set{ 0, \dc, r }$ and $j \in \Set{ 0, \dc, p-1 }$, and further $a_{\alpha}^{(j)} \ceqq \Braket{ a_j | H_{\alpha} z^j }$ for $\alpha \in \mc R$.

	By~\eqref{eq:annihilator_cyclic_vector} we see that $\colon X_kz^{-j} X^k z^{n+j} \colon w \neq 0$ implies $1 - p \leq j \leq p - 1 - n$, so $n \leq 2(p - 1)$ is necessary for nonvanishing terms.

	Now for $n \in \Set{ p-1, \dc, 2(p-1) }$ and $j \in \Set{ 1-p, \dc, p-1-n}$ one has
	\begin{equation*}
		-j, n + j \in \Set{ 1 - p + n, \dc, p-1 } \subseteq \Set{ 0, \dc, p-1 } \, ,
	\end{equation*}
	so that the normal-ordered products are void in~\eqref{eq:sugawara_operator}.
	Then for $\alpha \in \mc R^+$ and $i \in \Set{ 1, \dc, r }$ one computes
	\begin{equation}
		H_k z^{-j} H_kz^{j+n} w = a_k^{(-j)}a_k^{(j+n)} w \, , \qquad E^{\alpha} z^{-j} E_{\alpha} z^{n + j} w = 0 \, ,
	\end{equation}
	and
	\begin{equation}
		\frac{(H_{\alpha} \mid H_{\alpha})}{2} E_{\alpha} z^{-j} E^{\alpha} z^{n + j} w = H_{\alpha} z^n w = \delta_{n,p-1} a_{\alpha}^{(n)} w \, .
	\end{equation}

	Hence
	\begin{equation}
		\begin{split}
			2\bigl( \kappa & + h^{\dual} \bigr) L_n w
			= \sum_{j = 1-p}^{p-1-n} \left( \sum_{k = 1}^{r} \bigl( H_k z^{-j} H_k z^{j+n} \bigr) + \sum_{\alpha \in \mc R^+} \bigl( E_{\alpha} z^{-j} E^{\alpha} z^{j+n} \bigr) \right) w                             \\
			               & = \left( \sum_{j,k} \bigl( a_k^{(-j)} a_k^{(j+n)} \bigr) + \delta_{n,p-1} (2p-n-1) \sum_{\alpha \in \mc R^+} \Bigl( \frac{(\alpha \mid \alpha)}{2} a_{\alpha}^{(n)} \Bigr) \right) w \, ,
		\end{split}
	\end{equation}
	which implies~\eqref{eq:sugawara_eigenvalues_I} and~\eqref{eq:sugawara_eigenvalues_II} using $\frac{ (\alpha \mid \alpha)}{2} \Braket{ \mu | H_{\alpha} z^i } = ( \alpha \mid \mu )$, for $\mu \in \mf h^{\dual} \otimes z^i$.
\end{proof}

\subsection{Proof of Prop.~\ref{prop:irregular_kz_connection}}
\label{sec:proof_prop_irregular_kz_connection}

\begin{proof}
	Using the general case of~\eqref{eq:pole_expansion_finite} yields
	\begin{equation}
		X \otimes \tau_j (z_i^{-m}) \wh w_j = \sum_{l = 0}^{r_j-1} \binom{m+l-1}{l} \frac{ X z_j^l \wh w_j}{(t_i - t_j)^l(t_j - t_i)^m} \, , \quad X \otimes \tau_{\infty} (z_i^{-m}) \wh v_{\infty} = 0 \, ,
	\end{equation}
	for $X \in \mf g$, $i \neq j \in J'$, $\wh w_j \in W_j$ and $\wh v_{\infty} \in V_{\infty}$---since $z_{\infty} \mf g \llbracket z_{\infty} \rrbracket V_{\infty} = 0 = z_j^{r_j} \mf g \llbracket z_j \rrbracket V_j$.

	Hence by~\eqref{eq:identity_coinvariants} one has the identity $\bigl( X \otimes z_i^{-m} \bigr)^{(i)} \wh{\bm w} \otimes \wh v_{\infty} = \wh{\bm w}_{i,m,X} \otimes \wh v_{\infty}$ inside $\ms H$, where
	\begin{equation}
		\wh{\bm w}_{i,m,X} = -\sum_{j \in J' \setminus \Set{i}} \Biggl( \sum_{l = 0}^{r_j - 1}  \binom{m+l-1}{l} \frac{ \bigl( X z^l \bigr)^{(j)} \wh{\bm w} }{(t_i - t_j)^l(t_j - t_i)^m} \Biggr) \, .
	\end{equation}
	The result then follows from~\eqref{eq:action_sugawara_minus_one_finite_module_general}.
\end{proof}

\subsection{Proof of Prop.~\ref{prop:g_invariance_universal_connection}}
\label{sec:proof_prop_g_invariance_universal_connection}

\begin{proof}
	We prove the $\mf g$-action commutes with $\nabla_p \colon \Omega^0 \bigl( U(n,p) \bigr) \to \Omega^1 \bigl( U(n,p) \bigr)$.
	Since the $\mf g$-action is independent of the point on the base space, this is equivalent to $\varpi'_p \otimes X \psi - X \bigl( \varpi'_p \otimes \psi \bigr) = 0 = \varpi''_p \otimes X \psi - X \bigl( \varpi''_p \otimes \psi \bigr)$, for $X \in \mf g$.

	Now by~\eqref{eq:hamiltonians_finite} one has
	\begin{equation}
		\begin{split}
			(\kappa & + h^{\dual}) \Bigl( \varpi'_p \otimes X \psi - X \bigl( \varpi'_p\otimes \psi \bigr) \Bigr)                                                                                            \\
			        & = \sum_{i \neq j} \sum_{m,l = 0}^{p-1} (-1)^m \binom{m+l}{l} t_{ij}^{-1-m-l} \dif t_{ij} \otimes \Biggl( \, \sum_{k \in J'} \bigl[ \Omega^{(ij)}_{ml} X^{(k)} \bigr] \psi \Biggr) \, ,
		\end{split}
	\end{equation}
	and analogously by~\eqref{eq:hamiltonians_infinity}
	\begin{equation}
		\begin{split}
			(\kappa + h^{\dual}) \Bigl( & \varpi''_p \otimes X \psi - X \bigl( \varpi''_p\otimes \psi \bigr) \Bigr)                                                                                      \\
			                            & = \sum_{i \in J'} \sum_{m,l = 0}^{p-1} \binom{m+l}{l} t_i^l \otimes \Biggl( \, \sum_{k \in J'} \bigl[ \Omega^{(i\infty)}_{ml},X^{(k)} \bigr] \psi \Biggr) \, .
		\end{split}
	\end{equation}
	Hence it is enough to show that
	\begin{equation}
		\sum_{k \in J'} \bigl[ \Omega_{ml}^{(ij)},X^{(k)} \bigr] = 0 \in U ( \mf g_p )^{\otimes n} \, ,
	\end{equation}
	for all $i \neq j \in J$ and for all $m,l \in \mb Z$.
	Finally by~\eqref{eq:quadratic_tensor} we have
	\begin{equation}
		\sum_{k \in J'} \bigl[ \Omega^{(ij)}_{ml},X^{(k)} \bigr]
		= \sum_r \Bigl( \bigl[ X_r,X \bigr] z^m \Bigr)^{(i)} \bigl( X_r z^l \bigr)^{(j)} + \bigl( X_r z^m \bigr)^{(i)} \Bigl( \bigl[ X_r,X \bigr] z^l \Bigr)^{(j)} \, ,
	\end{equation}
	where we let $(X_r)_r$ be a $(\cdot \mid \cdot)$-orthonormal basis of $\mf g$, which vanishes by~\eqref{eq:commutation_quadratic_tensor}.
\end{proof}

\subsection{Proof of Prop.~\ref{prop:affine_equivariance}}
\label{sec:proof_prop_affine_equivariance}

\begin{proof}

	Indeed if $\wh{\bm w}$ is a $\nabla'_p$-horizontal section of $U(J,p) \to C_n(\mb C)$ then
	\begin{equation}
		E \wh{\bm w} = \sum_{i \in J'} \Biggl( \sum_{j \in J' \setminus \Set{i}} r_p^{(ij)}(t_{ij}) \Biggr) \bm{\wh w} \, ,
	\end{equation}
	which vanishes by the skew-symmetry~\eqref{eq:skew_symmetry_classical_r_matrix}, and which implies the statement about translations after taking $\mf g_p$-modules.

	As for dilations, in the universal case of a $\nabla'_p$-horizontal section one finds
	\begin{equation}
		\frac{H}{2} \wh{\bm w} = \sum_{i \in J'} \Biggl( \sum_{j \in J' \setminus \Set{i}} t_i r_p^{(ij)} (t_{ij}) \Biggr) \wh{\bm w} = \sum_{i \neq j \in J'} t_{ij} r_p^{(ij)}(t_{ij}) \wh{\bm w} \, ,
	\end{equation}
	and we must consider the induced action on finite singular modules.
	Now one computes
	\begin{equation}
		L_0 \wh w = \frac{1}{\kappa + h^{\dual}} \sum_{j = 1}^p \Biggl( \sum_k X_k z^{-j} X_k z^j \Biggr) \wh w, \quad \text{ for } \wh w \in W \, ,
	\end{equation}
	analogously to~\eqref{eq:action_sugawara_minus_one_finite_module_general}, using a $( \cdot \mid \cdot )$-orthonormal basis $(X_k)_k$ of $\mf g$.
	Then reasoning as in \S~\ref{sec:reduced_connection_tame_at_infinity} the induced slot-wise action on coinvariants is
	\begin{equation}
		L_0^{(i)} \wh{\bm w} = -\frac{1}{\kappa + h^{\dual}} \sum_{j \in J' \setminus \Set{i}} \Biggl( \sum_{m = 0}^{r_i - 1} \sum_{l = 0}^{r_j - 1} \binom{m+l}{l} (-1)^m t_{ij}^{-m-l} \Omega^{(ij)}_{ml} \Biggr) \wh{\bm w} \, ,
	\end{equation}
	with tacit use of the projection $\pi_{\ms H} \colon \mc H'_{\abs{\bm \lambda}} \to \ms H$, and on the whole
	\begin{equation}
		\frac{H}{2} \wh{\bm w} = \bigl( L_0 - L_0^{(\infty)} \bigr) \wh{\bm w} = \sum_{i \in J'} L_0^{(i)} \wh{\bm w} \, ,
	\end{equation}
	by~\eqref{eq:rational_function_finite}.
	This is the action of an endomorphism on the \emph{finite-dimensional} vector space $\ms H$, and the statement follows by integrating the resulting (linear, first-order) differential equation.
\end{proof}

\bibliographystyle{amsplain}
\bibliography{/home/gabriele/Desktop/bibliography_macros/bibliography}
\end{document}